\documentclass[11pt,reqno]{amsart}

\usepackage{amsmath}
\usepackage{amsthm}
\usepackage{amssymb}
\usepackage{amsfonts}
\usepackage{amsxtra}
\usepackage{fullpage}
\usepackage{amscd}
\usepackage{mathrsfs}
\usepackage{setspace}
\usepackage[all]{xypic}
\usepackage{enumerate}
\usepackage[noautoscale,enableskew]{youngtab}
\usepackage{graphicx}
\usepackage[mathscr]{eucal}
\usepackage{verbatim}
\usepackage{tikz}

\usepackage{hyperref}
\hypersetup{pdftitle=2lect}

\newtheorem{thm}{Theorem}
\newtheorem{prop}[thm]{Proposition}
\newtheorem{lem}[thm]{Lemma}

\newtheorem{cor}[thm]{Corollary}

\theoremstyle{remark}
\newtheorem{remark}[thm]{Remark}
\newtheorem{claim}[thm]{Claim}

\newtheorem{example}[thm]{Example}
\newtheorem{exercise}[thm]{Exercise}

\theoremstyle{definition}
\newtheorem{conjecture}[thm]{Conjecture}
\newtheorem{defn}[thm]{Definition}
\newtheorem{question}[thm]{Question}

\numberwithin{thm}{subsection}
\numberwithin{figure}{subsection}

\let\nc\newcommand
\let\renc\renewcommand

\newcommand{\Lmod}[1]{#1\text{-}{\mathsf{mod}}}

\newcommand{\idot}{{\:\raisebox{2pt}{\text{\circle*{1.5}}}}}

\DeclareMathOperator{\Ext}{\mathrm{Ext}}

\renewcommand{\Im}{\mathrm{Im}}

\DeclareMathOperator{\Ker}{\mathrm{Ker}}

\DeclareMathOperator{\End}{\mathrm{End}}

\DeclareMathOperator{\gr}{\mathrm{gr}}

\DeclareMathOperator{\rad}{\mathrm{rad}}

\DeclareMathOperator{\ann}{\mathtt{Ann}}

\newcommand{\beq}{\begin{equation}\label}
\newcommand{\eeq}{\end{equation}}

\DeclareMathOperator{\Spec}{\mathrm{Spec}}

\DeclareMathOperator{\Hom}{\mathrm{Hom}}

\DeclareMathOperator{\Loc}{\mathrm{Loc}}

\nc{\Z}{\mathbb{Z}}
\newcommand{\N}{\mathbb{N}}
\newcommand{\Q}{\mathbb{Q}}
\newcommand{\R}{\mathbb{R}}
\newcommand{\C}{\mathbb{C}}
\newcommand{\h}{\mathfrak{h}}

\newcommand{\eu}{\mathbf{eu}}

\nc{\rank}{\textrm{rank} \,}
\nc{\ds}{\dots}

\let\mc\mathcal
\let\mf\mathfrak

\nc{\mbf}{\mathbf}
\nc{\LK}{\textsf{Irr}(K)}
\nc{\LW}{\textsf{Irr}(W)}
\nc{\Res}{\mathsf{Res} \, }
\nc{\Ind}{\mathsf{Ind} \, }

\nc{\cont}{\textrm{cont}}
\renewcommand{\mod}{\textrm{mod}}
\nc{\eWb}{\mathbf{e}_{W_b}}
\nc{\sing}{\mathrm{sing}}
\nc{\msf}{\mathsf}
\nc{\Ui}{\mc{U}_{i,+}}
\nc{\Uone}{\mc{U}_{1,+}}
\nc{\Utwo}{\mc{U}_{2,+}}

\nc{\minusone}{-1}
\nc{\minustwo}{-2}
\nc{\minusthree}{-3}
\nc{\Mod}{\mathrm{Mod} \,}
\nc{\ms}{\mathscr}
\nc{\Frac}{\mathrm{Frac} \,}
\nc{\ra}{\rightarrow}
\nc{\hra}{\hookrightarrow}
\nc{\lab}{\label}
\renc{\O}{\mc{O}}
\nc{\Tan}{\Theta}
\nc{\ul}{\underline}
\nc{\s}{\mathfrak{S}}
\nc{\g}{\mf{g}}
\nc{\pa}{\partial}
\nc{\tit}{\textit}
\nc{\Maxspec}{\mathrm{Maxspec} \, }
\nc{\gldim}{\mathrm{gl.dim}}
\nc{\rkm}{\mathrm{rk} \, (\mf{m})}
\nc{\sm}{\mathrm{sm}}
\nc{\PD}{\mathbb{PD}}
\nc{\Hilb}{\textrm{Hilb}}
\nc{\<}{\langle}
\renc{\>}{\rangle}
\nc{\T}{\mathbb{T}}
\nc{\X}{\mathbb{X}}
\nc{\W}{\mathscr{W}}
\nc{\kt}{\mbf{k}}
\nc{\ko}{\mbf{k}(0)}
\nc{\Ok}{\mc{O}_G \boxtimes \kt_X}
\nc{\Oko}{\mc{O}_G \boxtimes \ko_X}
\nc{\OYk}{\mc{O}_Y \boxtimes \kt_X}
\nc{\id}{\msf{id}}
\nc{\A}{\mathbb{A}}
\nc{\Grel}{\mc{G}^{\mathrm{rel}}}
\nc{\Grat}{\mc{G}^{\mathrm{rat}}}
\nc{\Squo}[1]{\A^{(#1)}}
\nc{\twist}{\mathrm{twist}}
\nc{\Cd}{\mc{C}}
\nc{\Span}{\mathrm{Span}}
\nc{\Grass}{\mathrm{Gr}}
\nc{\Grad}{\mathrm{Gr}^{ad}}
\nc{\V}{\mc{V}}
\nc{\Y}{\mathbb{Y}}
\nc{\rightsim}{\stackrel{\sim}{\longrightarrow}}
\nc{\Psq}{\Ps^2_q}
\nc{\Ps}{\mathbb{P}}
\nc{\prim}{\mathrm{prim}}
\renc{\o}{\otimes}
\renc{\H}{\mathsf{H}}

\newcommand{\dd}{{\mathscr{D}}}

\newcommand{\ol}{\overline}

\newcommand{\Irr}{\mathrm{Irr}}

\newcommand{\Cs}{\C^{\times}}
\newcommand{\ch}{\mathrm{ch}}
\newcommand{\reg}{\mathrm{reg}}
\newcommand{\tor}{\mathrm{tor}}
\newcommand{\Sol}{\mathrm{Sol}}
\newcommand{\KZ}{\mathsf{KZ}}
\newcommand{\ZH}{\mathsf{Z}}

\numberwithin{equation}{subsection}

\begin{document}

\pagenumbering{roman}

\title{Symplectic reflection algebras}
\author{Gwyn Bellamy}
\address{School of Mathematics and Statistics, University of Glasgow, Univeristy Gardens, Glasgow, G12 8QW.}
\email{gwyn.bellamy@glasgow.ac.uk}
\urladdr{http://www.maths.gla.ac.uk/~gbellamy/}


\date{June 25, 2012}

\begin{abstract}
These lecture notes are based on an introductory course given by the author at the summer school ``Noncommutative Algebraic Geometry'' at MSRI in June 2012. The emphasis throughout is on examples to illustrate the many different facets of symplectic reflection algebras. Exercises are included at the end of each lecture in order for the student to get a better feel for these algebras.  
\end{abstract}

\maketitle

\tableofcontents

\section*{Introduction}
The purpose of these notes is to give the reader a flavor of, and basic grounding in, the theory of symplectic reflection algebras. These algebras, which were introduced by Etingof and Ginzburg in \cite{EG}, are related to an astonishingly large number of apparently disparate areas of mathematics such as combinatorics, integrable systems, real algebraic geometry, quiver varieties, resolutions of symplectic singularities and, of course, representation theory. As such, their exploration entails a journey through a beautiful and exciting landscape of mathematical constructions. In particular, as we hope to
illustrate throughout these notes, studying symplectic reflection algebras involves a deep interplay between geometry and representation theory.\\

A brief outline of the content of each lecture is as follows. In the first lecture we motivate the definition of symplectic reflection algebras by considering deformations of certain quotient singularities. Once the definition is given, we state the Poincar\'e-Birkhoff-Witt theorem, which is of fundamental importance in the theory of symplectic reflection algebras. This is the first of many analogies between Lie theory and symplectic reflection algebras. We also introduce a special class of symplectic reflection algebras, the \textit{rational Cherednik algebras}. This class of algebras gives us many interesting examples of symplectic reflection algebras that we can begin to play with. We end the lecture by describing the double centralizer theorem, which allows us to relate the symplectic reflection algebra with its spherical subalgebra, and also by describing the centre of these algebras. 

In the second lecture, we consider symplectic reflection algebras at $t = 1$. We focus mainly on rational Cherednik algebras and, in particular, on category $\mc{O}$ for these algebras. This category of finitely generated $\H_{\mbf{c}}(W)$-modules has a rich, combinatorial representation theory and good homological properties. We show that it is a highest weight category with finitely many simple objects. 

Our understanding of category $\mc{O}$ is most complete when the corresponding complex reflection group is the symmetric group. In the third chapter we study this case in greater detail. It is explained how results of Rouquier, Vasserot-Varangolo and Leclerc-Thibon allow us to express the multiplicities of simple modules inside standard modules in terms of the canonical basis of the ``level one'' Fock space for the quantum affine Lie algebra of type $A$. A corollary of this result is a character formula for the simple modules in category $\mc{O}$. We end the lecture by stating Yvonne's conjecture which explains how the above mentioned result should be extended to the case where $W$ is the wreath product $\s_n \wr \Z_m$. 

The fourth lecture deals with the Knizhnik-Zamolodchikov ($\KZ$) functor. This remarkable functor allows one to relate category $\mc{O}$ to modules over the corresponding cyclotomic Hecke algebra. In fact, it is an example of a quasi-hereditary cover, as introduced by Rouquier. The basic properties of the functor are described and we illustrate these properties by calculating explicitly what happens in the rank one case. 

The final lecture deals with symplectic reflection algebras at $t = 0$. For these parameters, the algebras are finite modules over their centres. We explain how the geometry of the centre is related to the representation theory of the algebras. We also describe the Poisson structure on the centre and explain its relevance to representation theory. For rational Cherednik algebras, we briefly explain how one can use the notion of Calogero-Moser partitions, as introduced by Gordon and Martino, in order to decide when the centre of these algebras is regular. \\

There are several very good lecture notes and survey papers on symplectic reflection algebras and rational Cherednik algebras. For instance \cite{EtingofCherednikNotes}, \cite{EtingofCalogeroMoser}, \cite{IainSurvey}, \cite{IainIMC} and \cite{RouquierSurvey}. I would also strongly suggest to anyone interested in learning about symplectic reflection algebras to read the orginal paper \cite{EG} by P. Etingof and V. Ginzburg, where symplectic reflection algebras were first defined\footnote{A few years after the publication of \cite{EG} it transpired  that the definition of symplectic reflection algebras had already appeared in a short paper \cite{Drinfeld} written by V. Drinfeld in the eighties.}. It makes a great introduction to the subject and is jam packed with ideas and clever arguments. As noted briefly above, there are strong connections between symplectic reflection algebras and several other areas of mathematics. Due to lack of time and energy, we haven't touched upon those connections here. The interested reader should consult one of the surveys mentioned above. A final remark: to make the lectures as readable as possible, there is only a light sprinkling of references in the body of the text. Detailed references can be found at the end of each lecture.  

\subsection*{Acknowledgments}

I would like to express my sincerest thanks to my fellow lectures Dan Rogalski, Michael Wemyss and Travis Schedler all the work they put into organizing the summer school at MSRI and for their support over the two weeks. Thanks also to all the students who so enthusiastically attended the course and worked so hard. I hope that it was as productive and enjoyable for them as it was for me.

\pagenumbering{arabic}

\newpage

\section{Symplectic reflection algebras}\label{sec:one}

The action of groups on spaces has been studied for centuries, going back at least to Sophus Lie's fundamental work on transformation groups. In such a situation, one can also study the \textit{orbit space} i.e. the set of all orbits. This space encodes a lot of the information about the action of a group on a space and provides an effective tool for constructing new spaces out of old ones. The underlying motivation for symplectic reflection algebras is to try and use representation theory to understand a large class of examples of orbit spaces that arise naturally in algebraic geometry. 

\subsection{Motivation}

Let $V$ be a finite dimensional vector space over $\C$ and $G \subset GL(V)$ a finite group. Fix $\dim V = m$. It is a classical problem in algebraic geometry to try and understand the orbit space $V / G = \Spec \C[V]^G$, see Wemyss' lectures \cite{Wemysslectures}. At the most basic level, we would like to try and answer the questions 

\begin{question}
Is the space $V / G$ singular?
\end{question}

Or, more generally

\begin{question}\label{q:2}
How singular is $V / G$?
\end{question}

The answer to the first question is a classical theorem due to Chevalley and Shephard-Todd. Before I can state their theorem, we need the notion of a complex reflection, which generalizes the classical definition of reflection encountered in Euclidean geometry.   

\begin{defn}
An element $s \in G$ is said to be a \textit{complex reflection} if $\mathrm{rk} (1 - s) = 1$. Then $G$ is said to be a \textit{complex reflection group} if $G$ is generated by $S$, the set of all complex reflections contained in $G$. 
\end{defn}

Combining the results of Chevalley and Shephard-Todd, we have:

\begin{thm}\label{thm:CSTthm}
The space $V/G$ is smooth if and only if $G$ is a complex reflection group. If $V/G$ is smooth then it is isomorphic to $\mathbb{A}^m$, where $\dim V = m$. 
\end{thm}

A complex reflection group $G$ is said to be \textit{irreducible} if the reflection representation $V$ is an irreducible $G$-module. It is an easy exercise (try it!) to show that if $V = V_1 \oplus \cdots \oplus V_k$ is the decomposition of $V$ into irreducible $G$-modules, then $G = G_1 \times \cdots \times G_k$, where $G_i$ acts trivially on $V_j$ for all $j \neq i$ and $(G_i,V_i)$ is an irreducible complex reflection group. The irreducible complex reflection groups have been classified by Shephard and Todd, \cite{ST}.

\begin{example}
Let $\s_n$ the symmetric group act on $\C^n$ by permuting the coordinates. Then, the reflections in $\s_n$ are exactly the transpositions $(i,j)$, which clearly generate the group. Hence $\s_n$ is a complex reflection group. If $\C^n = \Spec \C[x_1, \ds, x_n]$, then the ring of invariants $\C[x_1, \ds, x_n]^{\s_n}$ is a polynomial ring with generators $e_1, \ds, e_n$, where 
$$
e_k = \sum_{1 \le i_1 < \ds < i_k \le n} x_{i_1} \cdots x_{i_k}
$$
is the $i$th elementary symmetric polynomial. Notice however that $(\s_n,\C^n)$ is not an irreducible complex reflection group.  
\end{example}

\begin{example}[Non-example]
Take $m = 2$ i.e. $V = \C^2$ and $G$ a finite subgroup of $SL_2(\C)$. Then it is easy to see that $S = \emptyset$ so $G$ cannot be a complex reflection group. The singular space $\C^2 / G$ is called a Kleinian (or Du Val) singularity. The groups $G$ are classified by simply laced Dynkin diagrams i.e. those diagrams of type ADE, and the singularity $\C^2 / G$ is an isolated hypersurface singularity in $\C^3$. 
\end{example}

The previous non-example is part of a large class of groups called symplectic reflection groups. This is the class of groups for which one can try to understand the space $V/G$ using symplectic reflection algebras. In particular, we can try to give a reasonable answer to Question \ref{q:2} for these groups. Let $(V,\omega)$ be a symplectic vector space i.e. $\omega$ is a non-degenerate, skew symmetric bilinear form on $V$, and $Sp(V) = \{ g \in GL(V) \ | \ \omega(g u, g v) = \omega(u,v), \ \forall \ u,v \in V \}$, the symplectic linear group. If $G \subset Sp(V)$ is a finite subgroup then $G$ cannot contain any reflections since the determinant of every element $g$ in $Sp(V)$ is equal to one, thus it is never a complex reflection group. However, one can define $s \in G$ to be a \textit{symplectic reflection} if $\mathrm{rk}(1 - s) = 2$. The idea here being that a symplectic reflection is the nearest thing to a genuine complex reflection that one can hope for in a subgroup of $Sp(V)$.  

\begin{defn}
The triple $(V,\omega,G)$ is a \textit{symplectic reflection group} if $(V,\omega)$ is a symplectic vector space and $G \subset Sp(V)$ is a finite group that is generated by $\mc{S}$, the set of all symplectic reflections in $G$. 
\end{defn}

Since a symplectic reflection group $(V,\omega,G)$ is not ``too far'' from being a complex reflection group, one might expect $V/G$ to be ``not too singular''. A measure of the severity of the singularities in $V/G$ is given by much effort is required to remove them (to "resolve" the singularities). One way to make this precise is to ask whether $V/G$ admits what's called a crepant resolution (the actual definition of a crepant resolution won't be important to us in this course). This is indeed the case for many (but not all!) symplectic reflection groups\footnote{Skip to the end of the final lecture for a precise statement.}. In order to classify those groups $G$ for which the space $V/G$ admits a crepant resolution, the first key idea is to try to understand $V/G$ by looking at deformations of the space i.e. some affine variety $\pi : X \rightarrow \C^k$ such that $\pi^{-1}(0) \simeq V/G$ and the map $\pi$ is \textit{flat}. Intuitively, this is asking that the dimension of the fibers of $\pi$ don't change. Then it is reasonable to hope that a generic fiber of $\pi$ is easier to describe, but still tells us something about the geometry of $V/G$.

\begin{figure}\label{fig:resdef}
\begin{tikzpicture}
\draw [dotted,violet,thick] (-3,5) to [out=0,in=90] (-2.5,4) to [out=-90,in=0] (-3,3);
\draw [violet,thick] (-3,3) to [out=180,in=-90] (-3.5,4) to [out=90,in=180] (-3,5); 
\draw[dotted] (-5,5) to [out=0,in=90] (-4.5,4) to [out=-90,in=0] (-5,3);
\draw (-5,3) to [out=180,in=-90] (-5.5,4) to [out=90,in=180] (-5,5);
\draw (-1,5) to [out=0,in=90] (-0.5,4) to [out=-90,in=0] (-1,3) to [out=180,in=-90] (-1.5,4) to [out=90,in=180] (-1,5);
\draw (-5,5) -- (-1,5);
\draw (-5,3) -- (-1,3);

\node at (0.5,4) {$T^* \mathbb{P}^1$}; 
\draw [->,thick] (-3,2.5) -- (-3,0.5);

\draw[dotted] (-5,1) to [out=0,in=90] (-4.5,0) to [out=-90,in=0] (-5,-1);
\draw (-5,-1) to [out=180,in=-90] (-5.5,0) to [out=90,in=180] (-5,1);
\draw (-1,1) to [out=0,in=90] (-0.5,0) to [out=-90,in=0] (-1,-1) to [out=180,in=-90] (-1.5,0) to [out=90,in=180] (-1,1);
\draw (-4.85,0.98) -- (-1.15,-0.98);
\draw (-4.85,-0.98) -- (-1.15,0.98);
\draw (-5,1) to [out=0,in=-20] (-4.85,0.98);
\draw (-4.85,-0.98) to [out=20,in=0] (-5,-1);

\draw [fill,red] (-3,0) circle [radius=0.05];
\node at (-3,-0.8) {$\C^2 / \Z_2$};

\draw [<-,thick,dotted] (-0.1,0) to [out=30,in=150] (0.3,0) to [out=-30,in=-150] (0.7,0) to  [out=30,in=150] (1.1,0);

\draw[dotted] (2,1) to [out=0,in=90] (2.5,0) to [out=-90,in=0] (2,-1);
\draw (2,-1) to [out=180,in=-90] (1.5,0) to [out=90,in=180] (2,1);
\draw (6,1) to [out=0,in=90] (6.5,0) to [out=-90,in=0] (6,-1) to [out=180,in=-90] (5.5,0) to [out=90,in=180] (6,1);
\draw (2,1) to [out=0,in=180] (4,0.5) to [out=0,in=-180] (6,1);
\draw (2,-1) to [out=0,in=180] (4,-0.5) to [out=0,in=-180] (6,-1);
\node at (4,1.2) {$X_1(\Z_2)$};

\node at (0.5,2.5) {$\mathbb{P}^1$};
\draw [->] (0.2,2.5) to [out=180,in=-45] (-3.4,4);

\end{tikzpicture}
\caption{The resolution and the deformation of the $\Z_2$ quotient singularity.}
\end{figure}
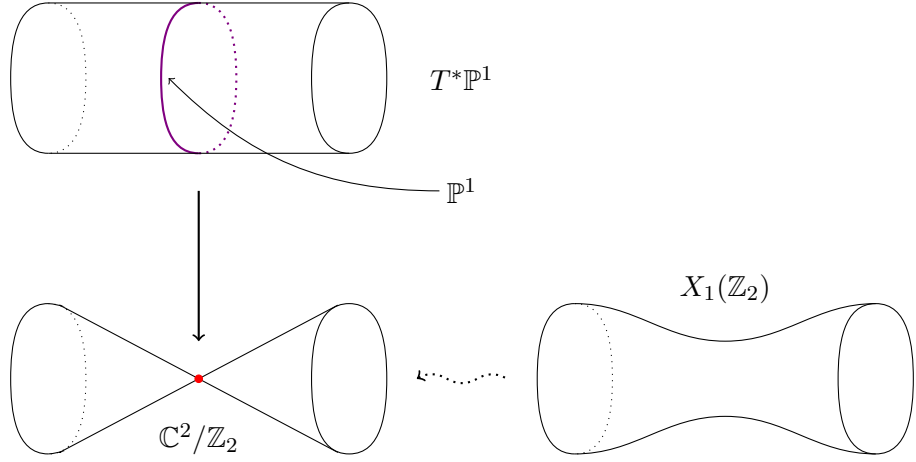

However there is a fundamental problem with this idea. We cannot hope to be able to write down generators and relations for the ring $\C[V]^G$ in general. So it seems like a hopeless task to try and write down deformations of the ring. The second key idea is to try and overcome this problem by introduce non-commutative geometry into the picture. In our case, the relevant non-commutative algebra is the skew group ring.  

\begin{defn}
The \textit{skew group ring} $\C[V] \rtimes G$ is, as a vector space, equal to $\C[V] \o \C G$ and the multiplication is given by 
$$
g \cdot f = {}^g f \cdot g, \quad \forall \ f \in \C[V], \ g \in G,
$$
where ${}^g f(v) = f(g^{-1} v)$ for $v \in V$. 
\end{defn}

\begin{exercise}\label{ex:centre}
Show that the centre $Z(\C[V] \rtimes G)$ of $\C[V] \rtimes G$ equals $\C[V]^G$. 
\end{exercise}

The above exercise shows that the information of the ring $\C[V]^G$ is encoded in the definition of the skew group ring. On the other hand, the skew group ring has a very explicit, simple presentation. Therefore, we can try to deform $\C[V] \rtimes G$ instead, in the hope that the centre of the deformed algebra is itself a deformation of $\C[V]^G$. We refer the reader to Schedler's lectures \cite{Schedlerlecture} for information on the theory of deformations of algebras. 

\subsection{Symplectic reflection algebras} 

Thus, symplectic reflection algebras are a particular family of deformations of the skew group ring $\C[V] \rtimes G$, when $G$ is a symplectic reflection group. Fix $(V,\omega,G)$, a symplectic reflection group. Let $\mc{S}$ be the set of symplectic reflections in $G$. For each $s \in \mc{S}$, the spaces $\Im ( 1 - s)$ and $\Ker (1-s)$ are symplectic subspaces of $V$ with $V = \Im ( 1 - s) \oplus \Ker (1-s)$ and $\dim \Im ( 1 - s) = 2$. We denote by $\omega_s$ the $2$-form on $V$ whose restriction to $\Im ( 1 - s)$ is $\omega$ and whose restriction to $\Ker (1-s)$ is zero. Let $\mbf{c} \, : \, \mc{S} \rightarrow \C$ be a conjugate invariant function i.e. 
$$
\mbf{c}(g s g^{-1}) = \mbf{c}(s), \quad \forall \ s \in \mc{S}, \ g \in G.
$$
The space of all such functions equals $\C[\mc{S}]^G$. Let $T V^* = \C \oplus V^* \oplus (V^* \o V^*) \oplus \cdots$ be the tensor algebra on $V^*$. 

\begin{defn}
Let $t \in \C$. The \textit{symplectic reflection algebra} $\H_{t,\mbf{c}}(G)$ is define to be 
\begin{equation}\label{eq:srarel}
\H_{t,\mbf{c}}(G) = T V^* \rtimes G \big/ \ \left\langle u \o v - v \o u = t \omega(u,v) - 2 \sum_{s \in \mc{S}} \mbf{c}(s) \omega_s(u,v) \cdot s \ | \ u,v \in V^* \right\rangle.
\end{equation}
\end{defn}

Notice that the defining relations of the symplectic reflection algebra are trying to tell you how to commute two vectors in $V^*$. The expression on the right hand side of (\ref{eq:srarel}) belongs to the group algebra $\C G$, with $t \omega(u,v) = t \omega(u,v) 1_G$, so the price for commuting $u$ and $v$ is that one gets an extra term living in $\C G$. 

\begin{example}\label{example:s2}
The simplest non-trivial example is $\Z_2 = \langle s \rangle$ acting on $\C^2$. Let $(\C^2)^* = \mathrm{Span} (x,y)$, where $s \cdot x = -x$, $s \cdot y = -y$ and $\omega(y,x) = 1$. Then $\H_{t,\mbf{c}}(\s_2)$ is the algebra
$$
\C \langle x,y,s \rangle / \langle s^2 = 1, s x = -xs, s y = - y s, [y,x] = t - 2 \mbf{c} s \rangle.
$$
This example, our ``favorite example'', will reappear throughout the course.  
\end{example}  

When $t$ and $\mbf{c}$ are both zero, we have $\H_{0,0}(G) = \C[V] \rtimes G$, so that $\H_{t,\mbf{c}}(G)$ really is a deformation of the skew group ring. If $\lambda \in \Cs$ then $\H_{\lambda t, \lambda \mbf{c}}(G) \simeq \H_{t,\mbf{c}}(G)$ so we normally only consider the cases $t = 0,1$. The Weyl algebra associated to the symplectic vector space $(V,\omega)$ is the non-commutative algebra 
$$
\mathsf{Weyl}(V,\omega) = T V^* / \left\langle u \o v - v \o u = \omega(u,v) \right\rangle.
$$
If $\h$ is a subspace of $V$, with $\dim \h = \frac{1}{2} \dim V$ and $\omega(\h,\h) = 0$, then $\mathsf{Weyl}(V,\omega) = \dd(\h)$, the ring of differential operators on $\h$. When $t = 1$ but $\mbf{c} = 0$, we have $\H_{1,0}(G) = \mathsf{Weyl}(V,\omega) \rtimes G$, the skew group ring associated to the Weyl algebra and $\H_{1,\mbf{c}}(G)$ is a deformation of this ring.  

\begin{example}
Again take $V = \C^2$, then $Sp(V) = SL_2(\C)$ so we can take $G$ to be any finite subgroup of $SL_2(\C)$. Every $g \neq 1$ in $G$ is a symplectic reflection and $\omega_g = \omega$. Let $x,y$ be a basis of $(\C^2)^*$ such that $\omega(y,x) = 1$. Then  
$$
\H_{t,\mbf{c}}(G) = \C \langle x,y \rangle \rtimes G \big/ \ \left\langle [y,x] = t  - 2 \sum_{g \in G \backslash \{ 1 \} } \mbf{c}(g) g \right\rangle.
$$
Since $\mbf{c}$ is $G$-equivariant, the element $z := t - 2 \sum_{g \in G \backslash \{ 1 \} } \mbf{c}(g) g$ belongs to the centre $Z(G)$ of the group algebra of $G$. Conversely, any element $z \in Z(G)$ can be expressed as $t - 2 \sum_{g \in G \backslash \{ 1 \} } \mbf{c}(g) g$ for some unique $t$ and $\mbf{c}$. Hence the main relation can simply be expressed as $[y,x] = z$ for some (fixed) $z \in Z(G)$. For (many) more properties of the algebras $\H_{t,\mbf{c}}(G)$, see \cite{CrawleyBoeveyHolland} where these algebras were first defined and studied. 
\end{example}

\subsection{Quantization} In this section we'll see how symplectic reflection algebras provide examples of quantization as described in Travis' lectures. Consider $\mbf{t}$ as a variable and $\H_{\mbf{t},\mbf{c}}(G)$ a $\C[\mbf{t}]$-algebra. Similarly, we consider $\mbf{e} \H_{\mbf{t},\mbf{c}}(G) \mbf{e}$ as a $\C[\mbf{t}]$-algebra. Then we may complete $\H_{\mbf{t},\mbf{c}}(G)$ and $\mbf{e} \H_{\mbf{t},\mbf{c}}(G) \mbf{e}$ respectively with respect to the two-sided ideals generated by the powers of $\mbf{t}$:
$$
\widehat{\H}_{\mbf{t},\mbf{c}}(G) = \lim_{\infty \leftarrow n} \H_{\mbf{t},\mbf{c}}(G) / (\mbf{t}^n), \quad \mbf{e} \widehat{\H}_{\mbf{t},\mbf{c}}(G) \mbf{e} = \lim_{\infty \leftarrow n} \mbf{e} \H_{\mbf{t},\mbf{c}}(G) \mbf{e} / (\mbf{t}^n) .
$$
The PBW Theorem implies that $\H_{\mbf{t},\mbf{c}}(G)$ and $\mbf{e} \H_{\mbf{t},\mbf{c}}(G) \mbf{e}$ are free $\C[\mbf{t}]$-modules. Therefore, $\widehat{\H}_{\mbf{t},\mbf{c}}(G)$ and $\mbf{e} \widehat{\H}_{\mbf{t},\mbf{c}}(G) \mbf{e}$ are flat, complete $\C[[\mbf{t}]]$-modules. Hence:

\begin{prop}
The algebra $\widehat{\H}_{\mbf{t},\mbf{c}}(G)$ is a formal deformation of $\H_{0,\mbf{c}}(G)$ and $\mbf{e} \widehat{\H}_{\mbf{t},\mbf{c}}(G) \mbf{e}$ is a formal quantization of $\mbf{e} \H_{0,\mbf{c}}(G) \mbf{e}$. 
\end{prop}

By formal quantization of $\mbf{e} \H_{0,\mbf{c}}(G) \mbf{e}$, we mean that there is a Poisson bracket on the commutative algebra $\mbf{e} \H_{0,\mbf{c}}(G) \mbf{e}$ such that the first order term of the quantization $\mbf{e} \widehat{\H}_{\mbf{t},\mbf{c}}(G) \mbf{e}$ is this bracket. This Poisson bracket is described in section \ref{sec:Poisson}. The algebras $\H_{\mbf{t},\mbf{c}}(G)$ are "better" than $\widehat{\H}_{\mbf{t},\mbf{c}}(G)$ in the sense that one can specialize, in the former, $\mbf{t}$ to any complex number, however in the latter only the specialization $\mbf{t} \rightarrow 0$ is well-defined.

\subsection{Filtrations}\label{sec:Filt} To try and study algebras, such as symplectic reflection algebras, that are given in terms of generators and relations, one would like to approximate the algebra by a simpler one, perhaps given by simpler relations, and hope that many properties of the algebra are invariant under this approximation process. An effective way of doing this is by defining a filtration on the algebra and passing to the associated graded algebra, which plays the role of the approximation. Let $A$ be a ring. A \textit{filtration} on $A$ is a nested sequence of abelian subgroups $0 = \mc{F}_{-1} A \subset \mc{F}_0 A \subset \mc{F}_{1} A \subset \cdots $ such that $(\mc{F}_i A)(\mc{F}_j A) \subseteq \mc{F}_{i+j} A$ for all $i,j$ and $A = \bigcup_{i \in \N} \mc{F}_i A$. The \textit{associated graded} of $A$ with respect to $\mc{F}_{\idot}$ is $\gr_{\mc{F}} A = \bigoplus_{i \in \N} \mc{F}_i A / \mc{F}_{i-1} A$. For each $0 \neq a \in A$ there is a unique $i \in \N$ such that $a \in \mc{F}_i A$ and $ a \notin \mc{F}_{i-1} A$. We say that $a$ lies in degree $i$; $\deg(a) = i$. We define $\sigma(a)$ to be the image of $a$ in $\mc{F}_i A  / \mc{F}_{i-1} A$; $\sigma(a)$ is called the \textit{symbol} of $a$. This defines a map $\sigma : A \rightarrow \gr_{\mc{F}} A$. It's important to note that the map $\sigma : A \rightarrow \gr_{\mc{F}} A$ is \textit{not} a morphism of abelian groups even though both domain and target are abelian groups. Let $\overline{a} \in \mc{F}_i A / \mc{F}_{i-1} A$ and $\overline{b} \in \mc{F}_j  A / \mc{F}_{j-1} A$. One can check that the rule $\overline{a} \cdot \overline{b} := \overline{ab}$, where $\overline{ab}$ denotes the image of $ab$ in $\mc{F}_{i+j} A / \mc{F}_{i+j-1} A$ extends by additivity to give a well-defined multiplication on $\gr_{\mc{F}} A$, making it into a ring. This multiplication is a "shadow" of that on $A$ and we're reducing the complexity of the situation by forgetting terms of lower order. 

There is a natural filtration $\mc{F}$ on $\H_{t,\mbf{c}}(G)$, given by putting $V^*$ in degree one and $G$ in degree zero. The crucial result by Etingof and Ginzburg, on which the whole of the theory of symplectic reflection algebras is built, is the Poincar\'e-Birkhoff-Witt (PBW) Theorem (the name comes from the fact that each of Poincar\'e, Birkhoff and Witt gave proofs of the analogous result for the enveloping algebra of a Lie algebra).

\begin{thm}\label{eq:PBW}
The map $\sigma (v) \mapsto v, \sigma(g) \mapsto g$ defines an isomorphism of algebras 
$$
\gr_{\mc{F}}(\H_{t,\mbf{c}}(G)) \simeq \C[V] \rtimes G,
$$
where $\sigma(D)$ denotes the symbol, or leading term, of $D \in \H_{t,\mbf{c}}(G)$ in $\gr_{\mc{F}}(\H_{t,\mbf{c}}(G))$. 
\end{thm}
 
One of the key point of the PBW theorem is that it gives us an explicit basis of the symplectic reflection algebra. Namely, if one fixes an \textit{ordered} basis of $V$, then the PBW theorem implies that there is an isomorphism of vector spaces $\H_{t,\mbf{c}}(G) \simeq \C[V] \o \C G$. One can also think of the PBW theorem as saying that no information is lost in deforming $\C[V] \rtimes G$ to $\H_{t,\mbf{c}}(G)$, since we can recover $\C[V] \rtimes G$ from $\H_{t,\mbf{c}}(G)$. 

The proof of this theorem is an application of a general result by Braverman and Gaitsgory \cite{PBW}. If $I$ is a two-sided ideal of $T V^* \rtimes G$ generated by a space $U$ of (not necessarily homogeneous) elements of degree at most two then \cite[Theorem 0.5]{PBW} gives necessary and sufficient conditions on $U$ so that the quotient $TV^* \rtimes G / I$ has the PBW property. The PBW property immediately implies that $\H_{t,\mbf{c}}(G)$ enjoys some good ring-theoretic properties, for instance: 

\begin{cor}\label{cor:finiteglobal}
\begin{enumerate}
\item The algebra $\H_{t,\mbf{c}}(G)$ is a prime, (left and right) Noetherian ring. 
\item The algebra $\mbf{e} \H_{t,\mbf{c}}(G) \mbf{e}$ is left and right Noetherian integral domain. 
\item $\H_{t,\mbf{c}}(G)$ has finite global dimension (in fact, $\mathrm{gl.dim} \ \H_{t,\mbf{c}}(G) \le \dim V$). 
\end{enumerate}
\end{cor}

We'll sketch a proof of the corollary, to illustrate the use of filtrations. 

\begin{proof}
All three of the results follow from the fact that corresponding statement hold for the skew-group ring $\C[V] \rtimes G$. We'll leave it to the reader to show these statements for $\C[V] \rtimes G$ - in particular, in part (3) the claim is that the global dimension of $\C[V] \rtimes G$ equals the global dimension of $\C[V]$, which is well-known to equal $\dim V$. So we may assume that $(A,\mc{F}_{\idot})$ is a filtered ring such that the above statements hold for $\gr_{\mc{F}} A$. Let $I_1 \subset I_2 \subset \cdots$ be a chain of left ideals in $A$. Then it follows from the definition of multiplication in $\gr_{\mc{F}} A$ that $\sigma(I_1) \subset \sigma(I_2) \subset \cdots$ is a chain of left ideals in $\gr_{\mc{F}} A$. Therefore, there is some $N$ such that $\sigma(I_{N+i}) = \sigma(I_{N})$ for all $i \ge 0$. This implies that $I_{N+i} = I_{N}$ for all $i \ge 0$, and hence $A$ is left Noetherian. The argument for right Noetherian is identical. Now assume that $I,J$ are ideals of $A$ such that $I \cdot J = 0$. Then certainly $\sigma(I) \sigma(J) = 0$. Since $\gr_{\mc{F}} A$ is assumed to be prime, this implies that either $\sigma(I) = 0$ or $\sigma(J) = 0$. But this can only happen if $I = 0$ or $J = 0$. Hence $A$ is prime.  

For part (2), we note that the algebra $\mbf{e} (\C[V] \rtimes W) \mbf{e}$ is isomorphic to $\C[V]^G \subset \C[V]$ and hence is a Noetherian integral domain. So we may assume that $(A,\mc{F}_{\idot})$ is a filtered ring such that $\gr_{\mc{F}} A$ is an integral domain. If $a , b \in A$ such that $a \cdot b = 0$ then certainly $\sigma(a) \cdot \sigma(b) = 0$ in $\gr_{\mc{F}} A$. Hence, without loss of generality $\sigma(a) = 0$. But this implies that $a = 0$, as required. 

To prove that $\mathrm{gl.dim} \ A \le \mathrm{gl.dim} \  \gr_{\mc{F}} A$ is a bit more involved (but not too difficult) because it involves the notion of filtration on $A$-modules, compatible with the filtration on $A$. The proof is given in \cite[Section 7.6]{MR}. 
\end{proof}

\subsection{The rational Cherednik algebra}\label{exam:RCAdefinition}

There is a standard way to construct a large number of symplectic reflection groups - by creating them out of complex reflection groups. This class of symplectic reflection algebras is by far the most important and, thus, have been most intensively studied out of all symplectic reflection algebras. So let $W$ be a complex reflection group, acting on its reflection representation $\h$. Then $W$ acts diagonally on $\h \times \h^*$. To be explicit, $W$ acts on $\h^*$ by $(w \cdot x)(y) = x( w^{-1} y)$, where $x \in \h^*$ and $y \in \h$. Then, $w \cdot (y,x) = (w \cdot y, w \cdot x)$. The space $\h \times \h^*$ has a natural pairing $( \cdot, \cdot ) : \mathfrak{h} \times \mathfrak{h}^* \rightarrow \C$ defined by $(y,x) = x(y)$, and 
$$
\omega ((y_1,x_1),(y_2,x_2)) := (y_1,x_2) - (y_2,x_1)
$$
defines a $W$-equivariant symplectic form on $\h \times \h^*$. One can easily check that the set of symplectic reflection $\mc{S}$ in $W$, consider as a symplectic reflection group $(\h \times \h^*,\omega,W)$, is the same as the set of complex reflections $S$ in $W$, considered as a complex reflection group $(W,\h)$. Therefore, $W$ acts on the symplectic vector space $\h \times \h^*$ as a symplectic reflection group if and only if it acts on $\h$ as a complex reflection group.  

The \textit{rational Cherednik algebra}, as introduced by Etingof and Ginzburg \cite[page 250]{EG}, is the symplectic reflection algebra associated to the triple $(\h \times \h^*, \omega,W)$. In this situation, one can simplify somewhat the defining relation (\ref{eq:srarel}). For each $s \in \mathcal{S}$, fix $\alpha_s \in \mathfrak{h}^*$ to be a basis of the one dimensional space $\Im (s - 1)|_{\mathfrak{h}^*}$ and $\alpha_s^{\vee} \in \mathfrak{h}$ a basis of the one dimensional space $\Im (s - 1)|_{\mathfrak{h}}$, normalized so that $\alpha_s(\alpha_s^\vee) = 2$. Then the relation (\ref{eq:srarel}) can be expressed as: 
\begin{equation}\label{eq:rel}
[x_1,x_2] = 0, \qquad [y_1,y_2] = 0, \qquad [y_1,x_1] = t (y_1,x_1) - \sum_{s \in \mathcal{S}} \mathbf{c}(s) (y_1,\alpha_s)(\alpha_s^\vee,x_1) s, 
\end{equation}
for all $x_1,x_2 \in \mathfrak{h}^* \textrm{ and } y_1,y_2 \in \mathfrak{h}$. Notice that these relations imply that $\C[\h]$ and $\C[\h^*]$ are polynomial subalgebras of $\H_{t,\mbf{c}}(W)$. 

\begin{example}\label{example:symmetric}
In the previous example we can take $W = \s_n$, the symmetric group. Choose a basis $x_1, \ds, x_n$ of $\mf{h}^*$ and dual basis $y_1, \ds, y_n$ of $\mf{h}$ so that 
$$
\sigma x_i = x_{\sigma(i)} \sigma, \quad \sigma y_i = y_{\sigma(i)} \sigma, \quad \forall \ \sigma \in \s_n. 
$$
Then $\mc{S} = \{ s_{i,j} \ | \ 1 \le i < j \le n \}$ is the set of all transpositions in $\s_n$. This is a single conjugacy class, so $\mbf{c} \in \C$. Fix
$$
\alpha_{i,j} = x_i - x_j, \quad \alpha_{i,j}^{\vee} = y_i - y_j, \quad \forall \ 1 \le i < j \le n.
$$
Then the relations for $\H_{t,\mbf{c}}(\s_n)$ become $[x_i,x_j] = [y_i,y_j] = 0$ and 
$$
[y_i,x_j] = \mbf{c} s_{i,j}, \quad \forall \ 1 \le i < j \le n,
$$
$$
[y_i,x_i] = t - \mbf{c} \sum_{j \neq i} s_{i,j}, \quad \forall \ 1 \le i \le n.
$$
\end{example}

\begin{exercise}\label{ex:PBW}
To see why the PBW theorem is quite a subtle statement, consider the algebra $\mathsf{L}(\s_2)$ defined to be 
$$
\C \langle x,y,s \rangle / \langle s^2 = 1, s x = -xs, s y = - y s, [y,x] = 1, (y - s) x = xy + s \rangle.
$$
Show that $\mathsf{L}(\s_2) = 0$. 
\end{exercise}

\subsection{Double Centralizer property}

Let $\mbf{e} = \frac{1}{|G|} \sum_{g \in G} g$ denote the trivial idempotent in $\C G$. The subalgebra $\mbf{e} \H_{t,\mbf{c}}(G) \mbf{e} \subset \H_{t,\mbf{c}}(G)$ is called the \textit{spherical subalgebra} of $\H_{t,\mbf{c}}(G)$. Being a subalgebra, it inherits a filtration from $\H_{t,\mbf{c}}(G)$. It is a consequence of the PBW theorem that $\gr_{\mc{F}}(\mbf{e} \H_{t,\mbf{c}}(G) \mbf{e}) \simeq \C[V]^G$. Thus, the spherical subalgebra of $\H_{t,\mbf{c}}(G)$ is a (not necessarily commutative!) flat deformation of the coordinate ring of $V/G$; almost exactly what we've been looking for! 

The space $\H_{t,\mbf{c}}(G) \mbf{e}$ is a $(\H_{t,\mbf{c}}(G), \mbf{e} \H_{t,\mbf{c}}(G) \mbf{e})$-bimodule, it is called the Etingof-Ginzburg sheaf. The following result shows that one can recover $\H_{t,\mbf{c}}(G)$ from knowing $\mbf{e} \H_{t,\mbf{c}}(G) \mbf{e}$ and $\H_{t,\mbf{c}}(G) \mbf{e}$.

\begin{thm}\label{thm:endomorphismring}
\begin{enumerate}
\item The map $\mbf{e} h \mapsto (\phi_{\mbf{e} h} : f \mbf{e} \mapsto \mbf{e} h f \mbf{e})$ is an isomorphism of \textit{left} $\mbf{e} \H_{t,\mbf{c}}(G) \mbf{e}$-modules 
$$
\mbf{e} \H_{t,\mbf{c}}(G)\stackrel{\sim}{\longrightarrow} \Hom_{\mbf{e} \H_{t,\mbf{c}}(G) \mbf{e}}(\H_{t,\mbf{c}}(G)\mbf{e}, \mbf{e} \H_{t,\mbf{c}}(G)\mbf{e}).
$$  
\item $\End_{\H_{t,\mbf{c}}(G)}(\H_{t,\mbf{c}}(G) \mbf{e})^{op} \simeq \mbf{e} \H_{t,\mbf{c}}(G) \mbf{e}$. 
\item $\End_{(\mbf{e} \H_{t,\mbf{c}} \mbf{e})^{op}}( \H_{t,\mbf{c}}(G) \mbf{e}) \simeq \H_{t,\mbf{c}}(G)$. 
\end{enumerate}
\end{thm}

\begin{remark}
As in (1), the natural map of \textit{right} $\mbf{e} \H_{t,\mbf{c}}(G) \mbf{e}$-modules 
$$
\H_{t,\mbf{c}}(G) \mbf{e} \rightarrow \Hom_{\mbf{e} \H_{t,\mbf{c}}(G) \mbf{e}}(\mbf{e}\H_{t,\mbf{c}}(G),\mbf{e} \H_{t,\mbf{c}}(G)\mbf{e})
$$
is an isomorphism. This, together with (1) imply that $\mbf{e} \H_{t,\mbf{c}}(G)$ and $\H_{t,\mbf{c}}(G) \mbf{e}$ are \textit{reflexive} left and right $\mbf{e} \H_{t,\mbf{c}}(G) \mbf{e}$-modules respectively  (see \cite[Section 5.1.7]{MR}).
\end{remark}

The above result is extremely useful because, unlike the spherical subalgebra, we have an explicit presentation of $\H_{t,\mbf{c}}(G)$. Therefore, we can try to implicitly study $\mbf{e} \H_{t,\mbf{c}}(G) \mbf{e}$ by studying instead the algebra $\H_{t,\mbf{c}}(G)$. Though the rings $\mbf{e} \H_{t,\mbf{c}}(G) \mbf{e}$ and $\H_{t,\mbf{c}}(G)$ are never isomorphic, very often the next best thing is true, namely that they are Mortia equivalent. This means that the categories of left $\mbf{e} \H_{t,\mbf{c}}(G) \mbf{e}$-modules and of left $\H_{t,\mbf{c}}(G)$-modules are equivalent. Let $A$ be an algebra. We'll denote by $\Lmod{A}$ the category of finitely generated left $A$-modules. If $A$ is Noetherian (which will always be the case for us) then $\Lmod{A}$ is abelian. 

\begin{cor}\label{cor:Morita}
The algebras $\H_{t,\mbf{c}}(G)$ and $\mbf{e} \H_{t,\mbf{c}}(G) \mbf{e}$ are Morita equivalent if and only if $\mbf{e} \cdot M = 0$ implies $M = 0$ for all $M \in \Lmod{\H_{t,\mbf{c}}(G)}$.
\end{cor}

\begin{proof}
Theorem \ref{thm:endomorphismring}, together with a basic result in Morita theory e.g. \cite[Section 3.5]{MR}, says that the bimodule $\H_{t,\mbf{c}}(G) \mbf{e}$ will induce an equivalence of categories $\mbf{e} \cdot - : \Lmod{\H_{t,\mbf{c}}(G)} \stackrel{\sim}{\longrightarrow} \Lmod{\mbf{e} \H_{t,\mbf{c}}(G) \mbf{e}}$ if and only if $\H_{t,\mbf{c}}(G) \mbf{e}$ is both a generator of the category $\Lmod{\H_{t,\mbf{c}}(G)}$ and a projective $\H_{t,\mbf{c}}(G)$-module. Since $\H_{t,\mbf{c}}(G) \mbf{e}$ is a direct summand of $\H_{t,\mbf{c}}(G)$ it is projective. Therefore we just need to show that it generates the category $\Lmod{\H_{t,\mbf{c}}(G)}$. This condition can be expressed as saying that
$$
\H_{t,\mbf{c}}(G) \mbf{e} \o_{\mbf{e} \H_{t,\mbf{c}}(G) \mbf{e}} \mbf{e} M \simeq M, \quad \forall \ M \in \Lmod{\H_{t,\mbf{c}}(G)}. 
$$
Equivalently, we require that $\H_{t,\mbf{c}}(G) \cdot \mbf{e} \cdot \H_{t,\mbf{c}}(G) = \H_{t,\mbf{c}}(G)$. If this is not the case then 
$$
I := \H_{t,\mbf{c}}(G) \cdot \mbf{e} \cdot \H_{t,\mbf{c}}(G)
$$
is a proper two-sided ideal of $\H_{t,\mbf{c}}(G)$. Hence there exists some module $M$ such that $I \cdot M = 0$. But this is equivalent to $\mbf{e} \cdot M = 0$. 
\end{proof}

The following notion is very important in the study of rational Cherednik algebas at $t = 1$. 

\begin{defn}
The parameter $(t,\mbf{c})$ is said to be \textit{aspherical} for $G$ if there exists a non-zero $\H_{t,\mbf{c}}(G)$-module $M$ such that $\mbf{e} \cdot M = 0$. 
\end{defn}

The value $(t,\mbf{c}) = (0,0)$ is an example of an aspherical value for $G$. 

\subsection{The centre of $\H_{t,\mbf{c}}(G)$}

One may think of the parameter $t$ as a ``quantum parameter''. When $t = 0$, we are in the ``quasi-classical\footnote{The words "quasi-classical" appear often in deformation theory. Why are things "quasi-classical" as opposed to "classical"? Well, roughly speaking, quantization is the process of making a commutative algebra (or a space) into a non-commutative algebra (or "non-commutative space"). The word quasi refers to the fact that when it is possible to quantize an algebra, this algebra (or space) has some additional structure. Namely, the fact that an algebra is quantizable implies that it is a Poisson algebra, and not just any old commutative algebra.}  situation'' and when $t = 1$ we are in the ``quantum situation'' - illustrated by the fact that $\H_{0,0}(G) = \C[V] \rtimes G$ and $\H_{1,0}(G) = \mathsf{Weyl}(V,\omega) \rtimes G$. The following result gives meaning to such a vague statement. It also shows that the symplectic reflection algebra produces a genuine commutative deformation of the space $V/G$ when $t = 0$. 

\begin{thm}\label{thm:centrespherical}
\begin{enumerate}
\item If $t = 0$ then the spherical subalgebra $\mbf{e} \H_{t,\mbf{c}}(G) \mbf{e}$ is commutative.   
\item If $t \neq 0$ then the centre of $\mbf{e} \H_{t,\mbf{c}}(G) \mbf{e}$ is $\C$. 
\end{enumerate}
\end{thm}

One can now use the double centralizer property, Theorem \ref{thm:endomorphismring}, to lift Theorem \ref{thm:centrespherical} to a result about the centre of $\H_{t,\mbf{c}}(G)$. 

\begin{thm}[The Satake isomorphism]\label{thm:Satake}
The map $z \mapsto z \cdot \mbf{e}$ defines an algebra isomorphism $Z(\H_{t,\mbf{c}}(G)) \stackrel{\sim}{\longrightarrow} Z(\mbf{e} \H_{t,\mbf{c}}(G) \mbf{e})$ for all parameters $(t,\mbf{c})$. 
\end{thm}

\begin{proof}
Clearly $z \mapsto z \cdot \mbf{e}$ is a morphism $Z(\H_{t,\mbf{c}}(G)) \rightarrow Z(\mbf{e} \H_{t,\mbf{c}}(G) \mbf{e})$. Right multiplication on $\H_{t,\mbf{c}}(G) \cdot \mbf{e}$ by an element $a$ in $Z(\mbf{e} \H_{t,\mbf{c}}(G) \mbf{e})$ defines a right $\mbf{e} \H_{t,\mbf{c}}(G) \mbf{e}$-linear endomorphism of $H_{t,\mbf{c}}(G) \cdot \mbf{e}$. Therefore Theorem \ref{thm:endomorphismring} says that there exists some $\zeta(a) \in \H_{t,\mbf{c}}(G)$ such that right multiplication by $a$ equals left multiplication on $\H_{t,\mbf{c}}(G) \cdot \mbf{e}$ by $\zeta(a)$. The action of $a$ on the right commutes with left multiplication by any element of $\H_{t,\mbf{c}}(G)$ hence $\zeta(a) \in Z(\H_{t,\mbf{c}}(G))$. The homomorphism $\zeta \, : \, Z(\mbf{e} \H_{t,\mbf{c}}(G) \mbf{e}) \rightarrow Z(\H_{t,\mbf{c}}(G))$ is the inverse to the Satake isomorphism.   
\end{proof}

When $t = 0$, the Satake isomorphism becomes an isomorphism $Z(\H_{0,\mbf{c}}(G)) \stackrel{\sim}{\longrightarrow} \mbf{e} \H_{0,\mbf{c}}(G) \mbf{e}$, and is in fact an isomorphism of Poisson algebras. Theorems \ref{thm:centrespherical} and \ref{thm:Satake} also imply that $\H_{0,\mbf{c}}(G)$ is a finite module over its centre. As one might guess, the behavior of symplectic reflection algebras is very different depending on whether $t = 0$ or $1$. It is also a very interesting problem to try and relate the representation theory of the algebras $\H_{0,\mbf{c}}(G)$ and $\H_{1,\mbf{c}}(G)$ in some meaningful way. 

\subsection{The Dunkl embedding}

In this section, parts of which are designed to be an exercise for the reader, we show how one can use the Dunkl embedding to give easy proofs for rational Cherednik algebras of many of the important theorems described in the lecture. In particular, one can give elementary proofs of both the PBW theorem and of the fact that the spherical subalgebra is commutative when $t = 0$. Therefore, we let $(W,\h)$ be a complex reflection group and $\H_{t,\mbf{c}}(W)$ the associated rational Cherednik algebra. Let $\dd_{t}(\h)$ be the algebra generated by $\h$ and $\h^*$, satisfying the relations
$$
[x,x'] = [y,y'] = 0, \quad \forall \ x,x' \in \h^*, \ y , y' \in \h
$$
and
$$
[y,x] = t (y,x), \quad \forall \ x \in \h^*, \ y \in \h. 
$$
When $t \neq 0$, $\dd_{t}(\h)$ is isomorphic to $\dd(\h)$, the ring of differential operators on $\h$. But when $t = 0$, the algebra $\dd_{t}(\h) = \C[\h \times \h^*]$ is commutative. Let $\h_{\reg}$ be the \textit{affine} (that $\h_{\reg}$ is affine is a consequence of the fact that $W$ is a complex reflection group, it is not true in general) open subset of $\h$ on which $W$ acts freely. We can localize $\dd_{t}(\h)$ to $\dd_{t}(\h_{\reg})$. For each $s \in \mc{S}$ define $\lambda_s \in \Cs$ by $s (\alpha_s) = \lambda_s \alpha_s$. For each $y \in \h$, the Dunkl operator 
$$
D_y = y - \sum_{s \in \mc{S}} \frac{2 \mbf{c}(s)}{1 - \lambda_s} \frac{(y,\alpha_s)}{\alpha_s} (1 - s)
$$
is an element in $\dd_{\mbf{t}}(\h_{\reg}) \rtimes W$ since $\alpha_s$ is invertible on $\h_{\reg}$. 

\begin{exercise}\label{ex:ex1.1}
\begin{enumerate}
\item Show that the Dunkl operators act on $\C[\h]$. Hint: it's strongly recommended that you do the example $W = \Z_2$ first, where 
$$
D_y = y - \frac{\mbf{c}}{x} (1 - s). 
$$ 
\item Using the fact that 
$$
s(x) = x - \frac{(\alpha_s^{\vee},x)}{2} (1 - \lambda_s) \alpha_s, \quad \forall \ x \in \h^*,
$$
show that $x \mapsto x$, $w \mapsto w$ and $y \mapsto D_y$ defines a morphism $\phi : \H_{t,\mbf{c}}(W) \rightarrow \dd_{t}(\h_{\reg}) \rtimes W$ i.e. show that the commutation relation 
$$
[D_y,x] = t (y,x) - \sum_{s \in \mathcal{S}} \mathbf{c}(s) (y,\alpha_s)(\alpha_s^\vee,x) s 
$$
holds for all $x \in \h^*$ and $y \in \h$. Hint: as above, try the case $\Z_2$ first.
\end{enumerate}
\end{exercise}

Now we have everything needed to prove the Poincar\'e-Birkhoff-Witt theorem for rational Cherednik algebras. The algebra $\dd_{\mbf{t}}(\h_{\reg}) \rtimes W$ has a natural filtration given by putting $\C[\h_{\reg}] \rtimes W$ in degree zero and $\h \subset \C[\h^*]$ in degree one. Similarly, we define a filtration on the rational Cherednik algebra by putting the generators $\h^*$ and $W$ in degree zero and $\h$ in degree one. 

\begin{lem}
The associated graded of $\H_{t,\mbf{c}}(W)$ with respect to the above filtration is isomorphic to $\C[\h \times \h^*] \rtimes W$. 
\end{lem}

\begin{proof}
We begin by showing that the Dunkl embedding (which we have yet to prove is actually an embedding) preserves filtrations i.e. $\phi(\mc{F}_i \H_{t,\mbf{c}}(W)) \subset \mc{F}_i (\dd_{\mbf{t}}(\h_{\reg}) \rtimes W)$ for all $i$. It is clear that $\phi$ maps $\h^*$ and $W$ into $\mc{F}_0 (\dd_{\mbf{t}}(\h_{\reg}) \rtimes W)$ and that $\phi(\h) \subset \mc{F}_1 (\dd_{\mbf{t}}(\h_{\reg}) \rtimes W)$. Since the filtration on $\H_{t,\mbf{c}}(W)$ is defined in terms of the generators $\h^*$, $W$ and $\h$, the inclusion $\phi(\mc{F}_i \H_{t,\mbf{c}}(W)) \subset \mc{F}_i (\dd_{\mbf{t}}(\h_{\reg}) \rtimes W)$ follows. Therefore the map $\phi$ induces a morphism 
\begin{equation}\label{eq:compiso}
\C[\h \times \h^* ] \rtimes W \rightarrow \gr_{\mc{F}} \H_{t,\mbf{c}}(W) \stackrel{\gr \phi }{\longrightarrow} \gr_{\mc{F}} (\dd_{\mbf{t}}(\h_{\reg}) \rtimes W) \rightarrow \C[\h_{\reg} \times \h^* ] \rtimes W,
\end{equation}
where the right-hand morphism is the inverse to the map $\C[\h_{\reg} \times \h^* ] \rtimes W \rightarrow \gr_{\mc{F}} (\dd_{\mbf{t}}(\h_{\reg}) \rtimes W)$, which is an isomorphism. The morphism (\ref{eq:compiso}) maps $\h$ to $\h$, $\h^*$ to $\h^*$ and $W$ to $W$. Therefore, it is the natural embedding $\C[\h \times \h^* ] \rtimes W \hookrightarrow \C[\h_{\reg} \times \h^* ] \rtimes W$. Thus, since the map $\C[\h \times \h^* ] \rtimes W \rightarrow \gr_{\mc{F}} \H_{t,\mbf{c}}(W)$ is surjective, it must be an isomorphism as required. 

Notice that we have also shown that $\gr_{\mc{F}} \H_{t,\mbf{c}}(W) \stackrel{\gr \phi }{\longrightarrow} \gr_{\mc{F}} (\dd_{\mbf{t}}(\h_{\reg}) \rtimes W)$ is an embedding. This implies that the Dunkl embedding is actually an embedding. As a consequence, $\C[\h]$ is a faithful $\H_{t,\mbf{c}}(W)$-module. 
\end{proof}

The fact that $\mbf{e} \H_{t,\mbf{c}}(G) \mbf{e}$ is commutative when $t = 0$ for an arbitrary symplectic reflection group relies on a very clever but difficult argument by Etingof and Ginzburg. However, for rational Cherednik algebras we have: 

\begin{exercise}\label{ex:ex1.3}
By considering its image under the Dunkl embedding, show that the spherical subalgebra $\mbf{e} \H_{t,\mbf{c}}(W) \mbf{e}$ is commutative when $t = 0$.
\end{exercise}

A function $f \in \C[\h]$ is called a $W$-\textit{semi-invariant} if, for each $w \in W$, $w \cdot f = \chi(w) f$ for some linear character $\chi : W \rightarrow \C^{\times}$. The element $\delta := \prod_{s \in \mc{S}} \alpha_s \in \C[\h]$ is a $W$-semi-invariant. To show this, let $w \in W$ and $s \in \mc{S}$. Then $w s w^{-1} \in \mc{S}$ is again a reflection. This implies that there is some non-zero scalar $\beta$ such that $w (\alpha_s) = \beta \alpha_{w s w^{-1}}$. Thus, $w (\delta) = \gamma_w \delta$ for some non-zero scalar $\gamma_w$. One can check that $\gamma_{w_1 w_2} = \gamma_{w_1} \gamma_{w_2}$, which implies that $\delta$ is a semi-invariant. Therefore, there exists some $r>0$ such that $\delta^r \in \C[\h]^W$. The powers of $\delta^r$ form an Ore set in $\H_{t,\mbf{c}}(W)$ and we may localize $\H_{t,\mbf{c}}(W)$ at $\delta^r$. Since the terms
$$
\sum_{s \in \mc{S}} \frac{2 \mbf{c}(s)}{1 - \lambda_s} \frac{(y,\alpha_s)}{\alpha_s} (1 - s)
$$
from the definition of $D_y$ belong to $\H_{t,\mbf{c}}(W)[\delta^{-r}]$, this implies that each $y \in \h \subset \dd_{t}(\h_{\reg}) \rtimes W$ belongs to $\H_{t,\mbf{c}}(W)[\delta^{-r}]$ too. Hence the Dunkl embedding becomes an isomorphism 
$$
\H_{t,\mbf{c}}(W)[\delta^{-r}] \stackrel{\sim}{\longrightarrow} \dd_{t}(\h_{\reg}) \rtimes W.
$$ 

Here is another application of the Dunkl embedding. Recall that a ring is said to be simple if it contains no proper two-sided ideals. 

\begin{prop}
The rings $\dd(\h)$, $\dd(\h)^W$ and $\dd(\h_{\reg})^W$ are simple. 
\end{prop}

\begin{proof}
To show that $\dd(\h)$ is simple, it suffices to show that $\C \cap I \neq 0$ for any two-sided ideal. This can be show by taking the commutator of a non-zero element $h \in I$ with suitable elements in $\dd(\h)$. The fact that this implies that $\dd(\h)^W$ is simple is standard, but maybe difficult to find in the literature. Firstly one notes that the fact that $\dd(\h)$ is simple implies that $\dd(\h)^W \simeq \mbf{e} (\dd(\h) \rtimes W) \mbf{e}$ is Morita equivalent to $\mbf{e} (\dd(\h) \rtimes W) \mbf{e}$; we've seen the argument already in the proof of Corollary \ref{cor:Morita}. This implies that $\dd(\h)^W$ is simple if and only if $\dd(\h) \rtimes W$ is simple; see \cite[Theorem 3.5.9]{MR}. Finally, one can show directly that $\dd(\h) \rtimes W$ is simple: see \cite[Proposition 7.8.12]{MR} ($W$ acts by outer automorphisms on $\dd(\h)$ since the only invertible elements in $\dd(\h)$ are the non-zero scalars). 

Finally, since $\dd(\h_{\reg})^W$ is the localization of $\dd(\h)^W$ at the two-sided Ore set generated by $\delta^r$, a two-sided ideal $J$ in $\dd(\h_{\reg})^W$ is proper if and only if $J \cap \dd(\h)^W$ is a proper two-sided ideal. But we have already shown that $\dd(\h)^W$ is simple. 
\end{proof}

\begin{cor}
The centre of $\mbf{e} \H_{1,\mbf{c}}(W) \mbf{e}$ equals $\C$. 
\end{cor}

\begin{proof}
By Corollary \ref{cor:finiteglobal} (2), $\mbf{e} \H_{1,\mbf{c}}(W) \mbf{e}$ is an integral domain. Choose $z \in Z(\mbf{e} \H_{1,\mbf{c}}(W) \mbf{e} )$. The Dunkl embedding defines an isomorphism 
$$
\mbf{e} \H_{1,\mbf{c}}(W) \mbf{e} [(\mbf{e} \delta^r)^{-1}] \stackrel{\sim}{\longrightarrow} \mbf{e} (\dd(\h_{\reg}) \rtimes W) \mbf{e} \simeq \dd(\h_{\reg})^W. 
$$
Since $\dd(\h_{\reg})^W$ is simple, every non-zero central element is either a unit or zero (otherwise it would generate a proper two-sided ideal). Therefore the image of $z$ in $\mbf{e} \H_{1,\mbf{c}}(W) \mbf{e} [(\mbf{e} \delta^r)^{-1}]$ is either a unit or zero. The fact that the only units in $\gr_{\mc{F}} \mbf{e} \H_{1,\mbf{c}}(W) \mbf{e} = \C[\h \times \h^*]^W$ are the scalars implies that the scalars are the only units in $\mbf{e} \H_{1,\mbf{c}}(W) \mbf{e}$. If $z$ is a unit then $\alpha (\mbf{e} \delta^r)^a \cdot z = 1$ in $\mbf{e} \H_{1,\mbf{c}}(W) \mbf{e}$, for some $\alpha \in \Cs$ and $a \in \N$. But $\mbf{e} \delta^r$ is not a unit in $\mbf{e} \H_{1,\mbf{c}}(W) \mbf{e}$ (since the symbol of $\mbf{e} \delta^r$ in $\C[\h \times \h^*]^W$ is not a unit). Therefore $a = 0$ and $z \in \Cs$. On the other hand if $(\mbf{e} \delta^r)^a \cdot z = 0$ for some $a$ then the fact that $\mbf{e} \H_{1,\mbf{c}}(W) \mbf{e}$ is an integral domain implies that $z = 0$.  
\end{proof}

The complex reflection group is said to be \textit{real} if there exists a real vector subspace $\h^{\mathrm{re}}$ of $\h$ such that $(W,\h^{\mathrm{re}}) $ is a real reflection group and $\h = \h^{\mathrm{re}} \o_{\R} \C$. In this case there exists a $W$-invariant inner product $( - , - )_{\mathrm{re}}$ on $\h^{\mathrm{re}}$ i.e. an inner product $( - , - )_{\mathrm{re}}$ such that $( wu , wv )_{\mathrm{re}} = ( u , v)_{\mathrm{re}}$ for all $w \in W$ and $u,v \in ( - , - )_{\mathrm{re}}$. We extend it by linearity to a $W$-invariant bilinear form $(- , - )$ on $\h$. The following fact is also very useful when studying rational Cherednik algebras at $t = 0$. 

\begin{exercise}\label{ex:ex1.6}
\begin{enumerate}
\item Assume now that $W$ is a real reflection group. Show that the rule $x \mapsto \tilde{x} = (x, - )$, $y \mapsto \tilde{y} = (y, - )$ and $w \mapsto w$ defines an automorphism of $\H_{t,\mbf{c}}(W)$, swapping $\C[\h]$ and $\C[\h^*]$. 

\item Show that $\C[\h]^W$ and $\C[\h^*]^W$ are central subalgebras of $\H_{0,\mbf{c}}(W)$ (hint: first use the Dunkl embedding to show that $\C[\h]^W$ is central, then use the automorphism defined in (7)). 
\end{enumerate}
\end{exercise}

\subsection{Additional remark}

\begin{itemize}
\item In his origin paper, \cite{Chevalley}, Chevalley showed that if $(W,\h)$ is a complex reflection group then $\C[\h]^W$ is a polynomial ring. The converse was shown by Shephard and Todd in \cite{ST}. . 

\item The definition of symplectic reflection algebras first appear in \cite{EG}. 

\item The PBW theorem, Theorem \ref{eq:PBW}, and its proof are Theorem 1.3 of \cite{EG}. 

\item Theorems \ref{thm:endomorphismring} and \ref{thm:Satake} are also contained in \cite{EG}, as Theorem 1.5 and Theorem 3.1 respectively. 

\item The first part of Theorem \ref{thm:centrespherical} is due to Etingof and Ginzburg, \cite[Theorem 1.6]{EG}. The second part is due to Brown and Gordon, \cite[Proposition 7.2]{PoissonOrders}. Both proof rely in a crucial way on the Poisson structure of $\C[V]^G$. 

\end{itemize}

\newpage

\section{Rational Cherednik algebras at $t = 1$}\label{sec:two}

For the remainder of lectures $2$ to $4$, we will only consider $t = 1$ and omit it from the notation. We will also only be considering rational Cherednik algebras because relatively little is know about general symplectic reflection algebras at $t = 1$. Therefore, we let $(W, \h)$ be a complex reflection group and $\H_{\mbf{c}}(W)$ the associated rational Cherednik algebra.\\

As noted in the previous lecture, the centre of $\H_{\mbf{c}}(W)$ equals $\C$. Therefore, its behavior is very different from the case $t = 0$. If we take $\mbf{c} = 0$ then $\H_{0}(W) = \dd(\h) \rtimes W$ and the category of modules for $\dd(\h) \rtimes W$ is precisely the category of $W$-equivariant $\dd$-modules on $\h$. In particular, there are \textit{no} finite dimensional representations of this algebra. In general, the algebra $\H_{\mbf{c}}(W)$ has very few finite dimensional representations.

\subsection{}  For rational Cherednik algebras, the PBW theorem implies that, as a vector space, $\H_{\mbf{c}}(W) \simeq \C[\mf{h}] \o \C W \o \C[\mf{h}^*]$; there is no need to choose an ordered basis of $\h$ and $\h^*$ for this to hold. Since $\C[\h]$ is in some sense opposite to $\C[\h^*]$, this is an example of a \textit{triangular decomposition}, just like the triangular decomposition $U(\mf{g}) = U(\mf{n}_-) \o U(\mf{h}) \o U(\mf{n}_+)$ encountered in Lie theory, where $\mf{g}$ is a finite dimensional, semi-simple Lie algebra over $\C$, $\mf{g} = \mf{n}_- \oplus \mf{h} \oplus \mf{n}_+$ is a decomposition into a Cartan subalgebra $\mf{h}$, the nilpotent radical $\mf{n}_+$ of the Borel $\mf{b} = \mf{h} \oplus \mf{n}_+$, and the opposite $\mf{n}_-$ of the nilpotent radical $\mf{n}_+$. This suggests that it might be fruitful to try and mimic some of the common constructions in Lie theory. In the representation theory of $\mf{g}$, one of the categories of modules most intensely studied, and best understood, is category $\mc{O}$, the abelian category of finitely generated $\g$-modules that are semi-simple as $\h$-modules and $\mf{n}_+$-locally nilpotent. Therefore, it is natural to try and study an analogue of category $\mc{O}$ for rational Cherednik algebras. This is what we will do in this lecture. 

\subsection{Category $\mc{O}$}

Let $\Lmod{\H_{\mbf{c}}(W)}$ be the category of all finitely generated left $\H_{\mbf{c}}(W)$-modules. It is a hopeless task to try and understand in any detail the whole category $\Lmod{\H_{\mbf{c}}(W)}$. Therefore, one would like to try and understand certain interesting, but manageable, subcategories. The PBW theorem suggests the following very natural definition.  

\begin{defn}
\textit{Category} $\mc{O}$ is defined to be the full\footnote{Recall that a subcategory $\mc{B}$ of a category $\mc{A}$ is called \textit{full} if $\Hom_{\mc{B}}(M,N) = \Hom_{\mc{A}}(M,N)$ for all $M,N \in \mathrm{Obj} \mc{B}$.} subcategory of $\Lmod{\H_{\mbf{c}}(W)}$ consisting of all modules $M$ such that the action of $\h \subset \C[\h^*]$ is locally nilpotent. 
\end{defn}

\begin{remark}
A module $M$ is said to be locally nilpotent for $\h$ if, for each $m \in M$ there exists some $N \gg 0$ such that $\h^N \cdot m = 0$. 
\end{remark}

\begin{exercise}
Show that every module in category $\mc{O}$ is finitely generated as a $\C[\h]$-module. 
\end{exercise}

We will give the proofs of the fundamental properties of category $\mc{O}$, since they do not require any sophisticated machinery. However, this does make the lecture rather formal, so we first outline they key features of category $\mc{O}$ so that the reader can get their bearings. Recall that an abelian category is called finite length if every object satisfies the ascending chain condition and descending chain condition on subobjects. It is Krull-Schmit if every module has a unique decomposition (up to permuting summands) into a direct sum of indecomposable modules. 
 
\begin{itemize}
\item There are only finitely many simple modules in category $\mc{O}$. 
\item Category $\mc{O}$ is a finite length, Krull-Schmit category. 
\item Every simple module admits a projective cover, hence category $\mc{O}$ contains enough projectives. 
\item Category $\mc{O}$ contains ``standard modules'', making it a highest weight category. 
\end{itemize}

\subsection{Standard objects}

One can use induction to construct certain ``standard objects'' in category $\mc{O}$. The skew-group ring $\C[\h^*] \rtimes W$ is a subalgebra of $\H_{\mbf{c}}(W)$. Therefore, we can induce to category $\mc{O}$ those representations of $\C[\h^*] \rtimes W$ that are locally nilpotent for $\h$. Let $\mf{m} = \C[\h^*]_+$ be the augmentation ideal. Then, for $\lambda \in \Irr (W)$, define the $\C[\h^*] \rtimes W$-module $\widetilde{\lambda} = \C[\h^*] \rtimes W \o_{W} \lambda$. For each $r \in \N$, the subspace $\mf{m}^r \cdot \widetilde{\lambda}$ is a proper $\C[\h^*] \rtimes W$-submodule of $\widetilde{\lambda}$ and we set $\lambda_r := \widetilde{\lambda} / \mf{m}^r \cdot \widetilde{\lambda}$. This is a $\h$-locally nilpotent $\C[\h^*]\rtimes W$-module. We set 
$$
\Delta_r(\lambda) = \H_{\mbf{c}}(W) \o_{\C[\h^*] \rtimes W} \lambda_r. 
$$
It follows from Lemma \ref{lem:standardO} that $\Delta_r(\lambda)$ is a module in category $\mc{O}$. The module $\Delta(\lambda) := \Delta_1(\lambda)$ is called a \textit{standard module} (or, often, a Verma module) of category $\mc{O}$. The PBW theorem implies that $\Delta(\lambda) = \C[\h] \o_{\C} \lambda$ as a $\C[\h]$-module. 

\begin{example}
For our favorite example, $\Z_2$ acting on $\mf{h} = \C \cdot y$ and $\mf{h}^* = \C \cdot x$, we have $\Irr (\Z_2) = \{ \rho_0, \rho_1 \}$, where $\rho_0$ is the trivial representation and $\rho_1$ is the sign representation. Then, 
$$
\Delta(\rho_0) = \C[x] \o \rho_0, \quad \Delta(\rho_1) = \C[x] \o \rho_1.
$$
The subalgebra $\C[x] \rtimes \Z_2$ acts in the obvious way. The action of $y$ is given as follows
$$
y \cdot f(x) \o \rho_i = [y,f(x)] \o \rho_i + f(x) \o y \rho_i = [y,f(x)] \o \rho_i, \quad i = 0,1.
$$
Here $[y,f(x)] \in \C[x] \rtimes \Z_2$ is calculated in $\H_{\mbf{c}}(\Z_2)$. 
\end{example}

\subsection{The Euler element}\label{sec:euler}

In Lie theory, the fact that every module $M$ in category $\mc{O}$ is semi-simple as a $\h$-module is very important, since $M$ decomposes as a direct sum of weight spaces. We don't have a Cartan subalgebra in $\H_{\mbf{c}}(W)$, but the Euler element is a good substitute (equivalently one can think that the Cartan subalgebras of $\H_{\mbf{c}}(W)$ are one-dimensional). 

Let $x_1, \ds, x_n$ be a basis of $\h^*$ and $y_1, \ds, y_n \in \h$ the dual basis. Define the \textit{Euler element} in $\H_{\mbf{c}}(W)$ to be
$$
\eu = \sum_{i = 1}^n x_i y_i - \sum_{s \in S} \frac{2\mbf{c}(s)}{1 - \zeta_s} s ,
$$
where $\zeta_s$ is the non-trivial eigenvalue of $s$ acting on $\h$ i.e. $s \cdot \alpha_s^{\vee} = \zeta_s \alpha_s^{\vee}$. The relevance of the element $\eu$ is given by:

\begin{exercise}\label{ex:eu}
Show that $[\eu,x] = x$, $[\eu,y] = -y$ and $[\eu,w] = 0$ for all $x \in \h^*$, $y \in \h$ and $w \in W$. 
\end{exercise}

Therefore conjugation by $\eu$ defines a $\Z$-grading on $\H_{\mbf{c}}(W)$, where $\deg (x) = 1$, $\deg (y) = -1$ and $\deg (w) = 0$. The sum $- \sum_{s \in S} \frac{2 \mbf{c} (s)}{1 - \zeta_s} s$ belongs to $Z(W)$, the centre of the group algebra. Therefore, if $\lambda$ is an irreducible $W$-module, this central element will act by a scalar on $\lambda$. This scalar will be denoted $\mbf{c}_{\lambda}$.  

\begin{lem}\label{lem:standardO}
The modules $\Delta_r(\lambda)$ belong to category $\mc{O}$. 
\end{lem}

\begin{proof}
We begin by noting that category $\mc{O}$ is closed under extensions i.e. if we have a short exact sequence 
$$
0 \rightarrow M_1 \rightarrow M_2 \rightarrow M_3 \rightarrow 0
$$
of $\H_{\mbf{c}}(W)$-modules, where $M_1$ and $M_3$ belong to $\mc{O}$, then $M_2$ also belongs to $\mc{O}$. The PBW theorem implies that $\H_{\mbf{c}}(W)$ is a free right $\C[\h^*] \rtimes W$-module. Therefore the short exact sequence $0 \rightarrow \lambda_{r-1} \rightarrow \lambda_r \rightarrow \lambda_1 \rightarrow 0$ of $\C[\h^*] \rtimes W$-modules defines a short exact sequence 
\beq{eq:sesDelta}
0 \rightarrow \Delta_{r-1}(\lambda) \rightarrow \Delta_r(\lambda) \rightarrow \Delta(\lambda) \rightarrow 0
\eeq
of $\H_{\mbf{c}}(W)$-modules. Hence, but induction on $r$, it suffices to show that $\Delta(\lambda)$ belongs to $\mc{O}$. We can make $\Delta(\lambda)$ into a $\Z$-graded $\H_{\mbf{c}}(W)$-module by putting $1 \o \lambda$ in degree zero. Then, $\Delta(\lambda)$ is actually positively graded with each graded piece finite dimensional. Since $y \in \h$ maps $\Delta(\lambda)_i$ into $\Delta(\lambda)_{i-1}$, we have $\h^{i+1} \cdot \Delta(\lambda)_i = 0$. 
\end{proof}

As for weight modules in Lie theory, we have:

\begin{lem}\label{lem:directsum}
Each $M \in \mc{O}$ is the direct sum of its generalized $\eu$-eigenspaces 
$$
M = \bigoplus_{a \in \C} M_a,
$$
and $\dim M_a < \infty$ for all $a \in \C$. 
\end{lem}

\begin{proof}
Since $M$ is in category $\mc{O}$, we can choose some finite dimensional $\C[\h^*] \rtimes W$-submodule $M'$ of $M$ that generates $M$ as a $\H_{\mbf{c}}(W)$-module. Since $M'$ is finite dimensional, there exists some $r \gg 0$ such that $\h^r \cdot M' = 0$. Thus we may find $\lambda_1,\ds, \lambda_k \in \Irr (W)$ such that the sequence 
$$
\bigoplus_{i=1}^k \Delta_r(\lambda_i) \rightarrow M \rightarrow 0
$$
is exact. Each $\Delta(\lambda)_i$ is a generalized $\eu$-eigenspace with eigenvalue $i + \mbf{c}_{\lambda}$. Hence $\Delta(\lambda)$ is a direct sum of its $\eu$-eigenspaces. As in the proof of Lemma \ref{lem:standardO}, one can use this fact together with the short exact sequence (\ref{eq:sesDelta}) to conclude that each $\Delta_r(\lambda)$ is a direct sum of its generalized $\eu$-eigenspaces, with each eigenspace finite dimensional. This implies that $M$ has this property too. 
\end{proof}

\begin{exercise}\label{ex:example1}
\begin{enumerate}
\item Give an example of a module $M \in \Lmod{\H_{\mbf{c}}(W)}$ that is not the direct sum of its generalized $\eu$-eigenspaces. 
\item Using $\eu$, show that every finite dimensional $\H_{\mbf{c}}(W)$-module is in category $\mc{O}$ (hint: how does $y \in \h$ act on a generalized eigenspace for $\eu$?). 
\end{enumerate}
\end{exercise}

\subsection{Characters} 

Using the Euler operator $\eu$ we can define the character of a module $M \in \mc{O}$ by 
$$
\ch (M) = \sum_{a \in \C} (\dim M_a) t^a.
$$
The Euler element acts via the scaler $\mbf{c}_{\lambda}$ on $1 \o \lambda \subset \Delta(\lambda)$. This implies that 
$$
\ch(\Delta(\lambda)) = \frac{ \dim (\lambda) t^{\mbf{c}_{\lambda}}}{(1 - t)^n}.
$$ 

\begin{exercise}\label{ex:charpoly}
Show that $\ch (M) \in \bigoplus_{a \in \C} t^a \Z[[t]]$ for all $M \in \mc{O}$. (Harder) 
\end{exercise}

In fact one can do better than the above result. As shown in exercise \ref{ex:Grothbasis} below, the standard modules $\Delta(\lambda)$ are a $\Z$-basis of the Grothendieck group $K_0(\mc{O})$. The character of $M$ only depends on its image in $K_0(\mc{O})$. Therefore if 
$$
[M] = \sum_{\lambda \in \Irr (W)} n_{\lambda} [\Delta(\lambda)] \in K_0(\mc{O}),
$$
for some $n_{\lambda} \in \Z$, then the fact that $\ch(\Delta(\lambda)) = \frac{\dim (\lambda) t^{\mbf{c}_{\lambda}}}{(1 - t)^n}$ implies that $\ch(M) = \frac{1}{(1 - t)^n} \cdot f(t)$, where 
$$
f(t) = \sum_{\lambda \in \Irr (W)} n_{\lambda} \dim (\lambda) t^{\mbf{c}_{\lambda}} \in \Z[x^a \ | \ a \in \C].
$$
Hence, the rule $M \mapsto (1 - t)^n \cdot \ch (M)$ is a morphism of abelian groups $K_0(\mc{O}) \rightarrow \Z[(\C,+)]$. It is not in general an embedding.

\subsection{Simple modules}

Two of the basic problems motivating research in the theory of rational Cherednik algebras are: 
\begin{enumerate}
\item Classify the simple modules in $\mc{O}$. 
\item Calculate $\ch(L)$ for all simple modules $L \in \mc{O}$. 
\end{enumerate}

The first problem is easy, but the second is very difficult (and still open in general). 

\begin{lem}\label{lem:deltasurject}
Let $M$ be a non-zero module in category $\mc{O}$. Then, there exists some $\lambda \in \Irr (W)$ and non-zero homomorphism $\Delta(\lambda) \rightarrow M$. 
\end{lem}

\begin{proof}
Note that the real part of the weights of $M$ are bounded from below i.e. there exists some $K \in \R$ such that $M_a \neq 0$ implies $\mathrm{Re} (a) \ge K$. Therefore we may choose some $a \in \C$ such that $M_a \neq 0$ and $M_b = 0$ for all $b \in \C$ such that $a - b \in \R_{> 0}$. An element $m \in M$ is said to be \textit{singular} if $\h \cdot m = 0$ i.e. it is annihilated by all $y$'s. Our assumption implies that all elements in $M_a$ are singular. If $\lambda$ occurs in $M_a$ with non-zero multiplicity then there is a well defined homomorphism $\Delta(\lambda) \rightarrow M$, whose restriction to $1 \o \lambda$ injects into $M_a$. 
\end{proof}

\begin{lem}
Each standard module $\Delta(\lambda)$ has a simple head $L(\lambda)$ and the set 
$$
\{ L(\lambda) \ | \ \lambda \in \Irr (W) \}
$$
is a complete set of non-isomorphic simple modules of category $\mc{O}$. 
\end{lem}

\begin{proof}
Let $R$ be the sum of all proper submodules of $\Delta(\lambda)$. It suffices to show that $R \neq \Delta(\lambda)$. The weight subspace $\Delta(\lambda)_{\mbf{c}_{\lambda}} = 1 \o \lambda$ is irreducible as a $W$-module and generates $\Delta(\lambda)$. If $R_{\mbf{c}_{\lambda}} \neq 0$ then there exists some proper submodule $N$ of $\Delta(\lambda)$ such that $N_{\mbf{c}_{\lambda}} \neq 0$. But then $N = \Delta(\lambda)$. Thus, $R_{\mbf{c}_{\lambda}} = 0$, implying that $R$ itself is a proper submodule of $\Delta(\lambda)$. Now let $L$ be a simple module in category $\mc{O}$. By Lemma \ref{lem:deltasurject}, there exists a non-zero homomorphism $\Delta(\lambda) \rightarrow L$ for some $\lambda \in \Irr (W)$. Hence $L \simeq L(\lambda)$. The fact that $L(\lambda) \simeq L(\mu)$ implies $\lambda \simeq \mu$ follows from the fact that $L_{\sing}$, the space of singular vectors in $L$, is irreducible as a $W$-module. 
\end{proof}

\begin{example}
Let's consider $\H_{\mbf{c}}(\Z_2)$ at $\mbf{c} = \frac{-3}{2}$. Then, one can check that 
$$
\Delta(\rho_1) = \C[x] \o \rho_1 \twoheadrightarrow L(\rho_1) = (\C[x] \o \rho_1) / (x^3 \C[x] \o \rho_1).
$$
On the other hand, $\Delta(\rho_0) = L(\rho_0)$ is simple. The composition series of $\Delta(\rho_1)$ is $\begin{array}{c} L(\rho_1) \\ L(\rho_0) \end{array}$.  
\end{example}

\begin{cor}
Every module in category $\mc{O}$ has finite length. 
\end{cor}

\begin{proof}
Let $M$ be a non-zero object of category $\mc{O}$. Choose some real number $K \gg 0$ such that $\mathrm{Re} (\mbf{c}_{\lambda}) < K$ for all $\lambda \in \Irr (W)$. We write $M^{\le K}$ for the sum of all weight spaces $M_a$ such that $\mathrm{Re} (a) \le K$. It is a finite dimensional subspace. Lemma \ref{lem:deltasurject} implies that $N^{\le K} \neq 0$ for all non-zero submodules $N$ of $M$. Therefore, if $N_0 \supsetneq N_1 \supsetneq N_2 \supsetneq \cdots$ is a proper descending chain of submodule of $M$ then $N_0^{\le K} \supsetneq N_1^{\le K} \supsetneq N_2^{\le K} \supsetneq \cdots$ is a proper descending chain of subspaces of $M^{\le K}$. Hence the chain must have finite length.  
\end{proof}

\subsection{Projective modules}

A module $P \in \mc{O}$ is said to be projective if the functor $\Hom_{\H_{\mbf{c}}(W)}(P, - ) : \mc{O} \rightarrow \mathrm{Vect}(\C)$ is exact. It is important to note that a projective module $P \in \mc{O}$ is \textit{not} projective when considered as a module in $\Lmod{\H_{\mbf{c}}(W)}$ i.e. being projective is a relative concept. 

\begin{defn}
An object $Q$ in $\mc{O}$ is said to have a \textit{$\Delta$-filtration} if it has a finite filtration $0 = Q_0 \subset Q_1 \subset \cdots \subset Q_r = Q$ such that $Q_i / Q_{i-1} \simeq \Delta(\lambda_i)$ for some $\lambda_i \in \Irr (W)$ and all $1 \le i \le r$. 
\end{defn}

Let $L \in \mc{O}$ be simple. A \textit{projective cover} $P(L)$ of $L$ is a projective module in $\mc{O}$ together with a surjection $p : P(L) \rightarrow L$ such that any morphism $f : M \rightarrow P(L)$ is a surjection whenever $p \circ f : M \rightarrow L$ is a surjection. Equivalently, the head of $P(L)$ equals $L$. Projective covers, when they exist, are unique up to isomorphism. The following theorem, first shown in \cite{GuayO}, is of key importance in the study of category $\mc{O}$. We follow the proof given in \cite{ArikiSurvey}.

\begin{thm}\label{thm:projectives}
Every simple module $L(\lambda)$ in category $\mc{O}$ has a projective cover $P(\lambda)$. Moreover, each $P(\lambda)$ has a finite $\Delta$-filtration. 
\end{thm}

Unfortunately the proof of Theorem \ref{thm:projectives} is rather long and technical. We suggest the reader skips it on first reading. For $a \in \C$ we denote by $\bar{a}$ its image in $\C / \Z$. We write $\mc{O}^{\bar{a}}$ for the full subcategory of $\mc{O}$ consisting of all $M$ such that $M_b = 0$ for all $b \notin a + \Z$. 

\begin{exercise}\label{ex:decompO}
Using the fact that all weights of $\H_{\mbf{c}}(W)$ under the adjoint action of $\eu$ are in $\Z$, show that 
$$
\mc{O} = \bigoplus_{\bar{a} \in \C / \Z} \mc{O}^{\bar{a}}.
$$
\end{exercise}

\begin{proof}[Proof of Theorem \ref{thm:projectives}]
The above exercise shows that it suffices to construct a projective cover $P(\lambda)$ for $L(\lambda)$ in $\mc{O}^{\bar{a}}$. Fix a representative $a \in \C$ of $\bar{a}$. For each $k \in \Z$, let $\mc{O}^{\ge k}$ denote the full subcategory of $\mc{O}^{\bar{a}}$ consisting of modules $M$ such that $M_b \neq 0$ implies that $b - a \in \Z_{\ge k}$. Then, for $k \gg 0$, we have $\mc{O}^{\ge k} = 0$ and for $k \ll 0$ we have $\mc{O}^{\ge k} = \mc{O}^{\bar{a}}$. To see this, it suffices to show that such $k$ exist for the finitely many simple modules $L(\lambda)$ in $\mc{O}^{\bar{a}}$ - then an arbitrary module in $\mc{O}^{\bar{a}}$ has a finite composition series with factors the $L(\lambda)$, which implies that the corresponding statement holds for them too. Our proof of Theorem \ref{thm:projectives} will be by induction on $k$. Namely, for each $k$ and all $\lambda \in \Irr (W)$ such that $L(\lambda) \in \mc{O}^{\ge k}$, we will construct a projective cover $P_{k}(\lambda)$ of $L(\lambda)$ in $\mc{O}^{\ge k}$ such that $P_{k}(\lambda)$ is a \textit{quotient} of $\Delta_r(\lambda)$ for $r \gg 0$. At the end we'll deduce that $P_{k}(\lambda)$ itself has a $\Delta$-filtration. The idea is to try and lift each $P_k(\lambda)$ in $\mc{O}^{\ge k}$ to a corresponding $P_{k-1}(\lambda)$ in $\mc{O}^{\ge k-1}$. 

Let $k_0$ be the largest integer such that $\mc{O}^{\ge k_0} \neq 0$. We begin with: 

\begin{claim}\label{claim:bottomsimple}
The category $\mc{O}^{\ge k_0}$ is semi-simple with $P_{k_0}(\lambda) = \Delta(\lambda) = L(\lambda)$ for all $\lambda$ such that $L(\lambda) \in \mc{O}^{\ge k_0}$. 
\end{claim}

\begin{proof}[Proof of the claim]
Note that $\Delta(\lambda)_{\mbf{c}_{\lambda}} = 1 \o \lambda$ and $\Delta(\lambda)_b \neq 0$ implies that $b - \mbf{c}_{\lambda} \in \Z_{\ge 0}$. If the quotient map $\Delta(\lambda) \rightarrow L(\lambda)$ has a non-zero kernel $K$ then choose $L(\mu) \subset K$ a simple submodule. We have $\mbf{c}_{\mu} - \mbf{c}_{\lambda} \in \Z_{>0}$, contradicting the minimality of $k_0$. Thus $K = 0$. Since $\Delta(\lambda)$ is an induced module, adjunction implies that  
$$
\Hom_{\mc{O}^{\ge k_0}}(\Delta(\lambda),M) = \Hom_{\C[\h^*] \rtimes W}(\lambda,M) 
$$
for all $M \in \mc{O}^{\ge k_0}$. Again, since all the weights of $M$ are at least $\mbf{c}_{\lambda}$, this implies that 
$$
\Hom_{\C[\h^*] \rtimes W}(\lambda,M)  = \Hom_{W}(\lambda, M_{\mbf{c}_{\lambda}}).
$$
Since $M$ is a direct sum of its generalized $\eu$-eigenspaces, this implies that $\Hom_{\mc{O}^{\ge k_0}}(\Delta(\lambda),-)$ is an exact functor i.e. $\Delta(\lambda)$ is projective. 
\end{proof}

Now take $k < k_0$ and assume that we have constructed, for all $L(\lambda) \in \mc{O}^{\ge k+1}$, a projective cover $P_{k+1}(\lambda)$ of $L(\lambda)$ in $\mc{O}^{\ge k+1}$ with the desired properties. If $\mc{O}^{\ge k+1} = \mc{O}^{\ge k}$ then there is nothing to do so we may assume that there exist $\mu_1,\ds, \mu_r \in \Irr(W)$ such that $L(\mu_i)$ belongs to $\mc{O}^{\ge k}$, but not to $\mc{O}^{\ge k+1}$. Note that $\mbf{c}_{\mu_i} = a + k$ for all $i$. For all $M \in \mc{O}^{\ge k}$, either $M_{a+k} = 0$ (in which case $M \in \mc{O}^{\ge k+1}$) or $M_{a+k}$ consists of singular vectors i.e. $\h \cdot M_{a + k} = 0$. Therefore, as in the proof of Claim \ref{claim:bottomsimple}, we have $P_{k}(\mu_i) = \Delta(\mu_i)$ for $1 \le i \le r$. Notice that $P_{k}(\mu_i)$ is obviously a quotient of $\Delta(\mu_i)$. Thus, we are left with constructing the lifts $P_k(\lambda)$ of $P_{k+1}(\lambda)$ for all those $\lambda$ such that $L(\lambda) \in \mc{O}^{\ge k+1}$. Fix one such $\lambda$.

\begin{claim}
There exists some integer $N \gg 0$ such that $(\eu - \mbf{c}_{\lambda})^N \cdot m = 0$ for all $M \in \mc{O}^{\ge k}$ and all $m \in M_{\mbf{c}_{\lambda}}$. 
\end{claim}

By definition, for a given $M \in \mc{O}^{\ge k}$ and $m \in M_{\mbf{c}_{\lambda}}$, there exists some $N \gg 0$ such that $(\eu - \mbf{c}_{\lambda})^N \cdot m = 0$. The claim is stating that one can find a particular $N$ that works simultaneously for all $M \in \mc{O}^{\ge k}$ and all $m \in M_{\mbf{c}_{\lambda}}$.

\begin{proof}[Proof of the claim]
Since $\dim (M_{a + k}) < \infty$ there exist $n_i \in \N$ and a morphism 
$$
\phi : \bigoplus_{i = 1}^r \Delta(\mu_i)^{\oplus n_i} \longrightarrow M
$$
such that the cokernel $M'$ of $\phi$ is in $\mc{O}^{\ge k+1}$. Therefore we may construct a surjection
$$
\psi : \bigoplus_{\eta} P_{k+1}(\eta)^{s_{\eta}} \twoheadrightarrow M'
$$
where the sum is over all $\eta \in \Irr (W)$ such that $L(\eta) \in \mc{O}^{\ge k+1}$. Since each $\Delta_r(\eta)$ is in category $\mc{O}$, Lemma \ref{lem:directsum} implies that there is some $N \gg 0$ such that $(\eu - \mbf{c}_{\lambda})^{N-1} \cdot q = 0$ for all $q \in \Delta_r(\eta)_{\mbf{c}_{\lambda}}$. Since $P_{k+1}(\eta)$ is a quotient of $\Delta_r(\eta)$, $(\eu - \mbf{c}_{\lambda})^{N-1} \cdot p = 0$ for all $p \in P_{k+1}(\eta)_{\mbf{c}_{\lambda}}$ too. Therefore $(\eu - \mbf{c}_{\lambda})^{N-1} \cdot m$ lies in the image of $\phi$ for all $m \in M_{\mbf{c}_{\lambda}}$ and hence $(\eu - \mbf{c}_{\lambda})^{N} \cdot m = 0$.
\end{proof}

Now choose some new integer $r$ such that $r + \mbf{c}_{\lambda} \gg a + k$ and define
$$
R(\lambda) = \frac{\Delta_r(\lambda)}{\H_{\mbf{c}}(W) \cdot (\eu - \mbf{c}_{\lambda})^N (1 \o 1 \o \lambda)}.
$$
Then, for $M \in \mc{O}^{\ge k}$, 
$$
\Hom_{\H_{\mbf{c}}(W)}(R(\lambda), M) = \{ m \in M_{\mbf{c}_{\lambda}} \ | \ \h^r \cdot m = 0 \}_{\lambda},
$$
where the subscript $\{ - \}_{\lambda}$ refers to the $\lambda$-isotypic component. Since $\mbf{c}_{\lambda} - r \ll a + k$ and $M \in \mc{O}^{\ge k}$, we have $M_{\mbf{c}_{\lambda} - r} = 0$. But $\h^r \cdot m \in M_{\mbf{c}_{\lambda} - r}$ for all $m \in M_{\mbf{c}_{\lambda}}$ which means that $\Hom_{\H_{\mbf{c}}(W)}(R(\lambda), M)$ equals the $\lambda$-isotypic component of $M_{\mbf{c}_{\lambda}} $. Thus, $\Hom_{\H_{\mbf{c}}(W)}(R(\lambda), - )$ is exact on $\mc{O}^{\ge k}$ (the functor $M$ maps to the $\lambda$-isotypic component of $M_{\mbf{c}_{\lambda}}$ being exact). It is non-zero because it surjects onto $\Delta(\lambda)$. The only problem is that it does not necessarily belong to $\mc{O}^{\ge k}$. So we let $\tilde{R}(\lambda)$ be the $\H_{\mbf{c}}(W)$-submodule generated by all weight spaces $R(\lambda)_b$ with $b - a \notin \Z_{\ge k}$ and define $P_{k}(\lambda) := R(\lambda) / \tilde{R}(\lambda)$. By construction, it belongs to $\mc{O}^{\ge k}$ and if $f : R(\lambda) \rightarrow M$ is any morphism with $M \in \mc{O}^{\ge k}$ then $\tilde{R}(\lambda) \subset \Ker f$. Therefore it is the projective cover of $L(\lambda)$ in $\mc{O}^{\ge k}$. We have constructed $P_{k}(\lambda)$ as a quotient of $\Delta_r(\lambda)$, an object equipped with a $\Delta$-filtration. 

The only thing left to show is that if $k \ll 0$ such that $\mc{O}^{\ge k} = \mc{O}^{\bar{a}}$ then $P_k(\lambda)$ has a $\Delta$-filtration. By construction, it is a quotient of an object $M \in \mc{O}$ that is equipped with a $\Delta$-filtration. But our assumption on $k$ means that $P_k(\lambda)$ is projective in $\mc{O}$. Thus, it is a direct summand of $M$. Therefore, it suffices to note that if $M = M_1 \oplus M_2$ is an object of $\mc{O}$ equipped with a $\Delta$-filtration, then each $M_i$ also has a $\Delta$-filtration (this follows by induction on the length of $M$ from the fact that the modules $\Delta(\lambda)$ are \textit{indecomposable}). This completes the proof of \ref{thm:projectives}. 
\end{proof}

In general, it is very difficult to explicitly construct the projective covers $P(\lambda)$. The object $P = \bigoplus_{\lambda \in \Irr (W)} P(\lambda)$ is a projective generator of $\mc{O}$ i.e for each $M \in \mc{O}$ there exists some $N \gg 0$ and surjection $P^N \rightarrow M$. Therefore, we have an equivalence of abelian categories
$$
\mc{O} \simeq \Lmod{A},
$$
where $A = \End_{\H_{\mbf{c}}(W)}(P)$ is a finite dimensional $\C$-algebra. See \cite[Chapter 2]{ARS}, and in particular \cite[Corollary 2.6]{ARS}, for details. 

\subsection{Highest weight categories}

Just as for category $\mc{O}$ of a semi-simple Lie algebra $\g$ over $\C$, the existence of standard modules in category $\mc{O}$ implies that this category has a lot of additional structure. In particular, it is an example of a highest weight (or quasi-hereditary) category. The abstract notion of a highest weight category was introduced in \cite{CPS}. 

\begin{defn}\label{defn:hw}
Let $\mc{A}$ be an abelian, $\C$-linear and finite length category, and $\Lambda$ a poset. We say that $(\mc{A},\Lambda)$ is a \textit{highest weight category} if  
\begin{enumerate}
\item There is a complete set $\{ L(\lambda) \ | \ \lambda \in \Lambda \}$ of non-isomorphic simple objects labeled by $\Lambda$.
\item There is a collection of \textit{standard} objects $\{ \Delta(\lambda) \ | \ \lambda \in \Lambda \}$ of $\mc{A}$, with surjections $\phi_{\lambda} : \Delta(\lambda) \twoheadrightarrow L(\lambda)$ such that all composition factors $L(\mu)$ of $\Ker \phi_{\lambda}$ satisfy $\mu < \lambda$. 
\item Each $L(\lambda)$ has a projective cover $P(\lambda)$ in $\mc{A}$ and the projective cover $P(\lambda)$ admits a $\Delta$-filtration $0 = F_{0} P(\lambda) \subset F_1 P(\lambda) \subset \cdots \subset F_m P(\lambda) = P(\lambda)$ such that
\begin{itemize}
\item $F_m P(\lambda) / F_{m-1} P(\lambda) \simeq \Delta(\lambda)$. 
\item For $0 < i < m$, $F_i P(\lambda) / F_{i-1} P(\lambda) \simeq \Delta(\mu)$ for some $\mu > \lambda$. 
\end{itemize}  
\end{enumerate}
\end{defn}

Define a partial ordering on $\Irr (W)$ by setting 
$$
\lambda \le_{\mbf{c}} \mu \quad \Longleftrightarrow \quad \mbf{c}_{\mu} - \mbf{c}_{\lambda} \in \Z_{\ge 0}.
$$

\begin{lem}\label{lem:trivext}
Let $\lambda, \mu \in \Irr (W)$ such that $\lambda \not<_{\mbf{c}} \mu$. Then $\Ext^1_{\H_{\mbf{c}}(W)}(\Delta(\lambda),\Delta(\mu)) = 0$. 
\end{lem}

\begin{proof}
Recall that $\Ext^1_{\H_{\mbf{c}}(W)}(\Delta(\lambda),\Delta(\mu))$ can be identified with isomorphism classes of short exact sequences $0 \rightarrow \Delta(\mu) \rightarrow M \rightarrow \Delta (\lambda) \rightarrow 0$, such that $\Ext^1_{\H_{\mbf{c}}(W)}(\Delta(\lambda),\Delta(\mu)) = 0$ if and only if all such short exact sequences split. Assume that we are given a short exact sequence 
$$
0 \rightarrow \Delta(\mu) \rightarrow M \rightarrow \Delta (\lambda) \rightarrow 0
$$
for some $M \in \Lmod{\H_{\mbf{c}}(W)}$. Then $M \in \mc{O}$. Take $0 \neq v \in \Delta(\lambda)_{\mbf{c}_{\lambda}} = 1 \o \lambda$ and $m \in M_{\mbf{c}_{\lambda}}$ that maps onto $v$. Since $\lambda \not<_{\mbf{c}} \mu$, there is no $a \in \C$ such that $a - \mbf{c}_{\lambda} \in \Z_{> 0}$ and $M_a \neq 0$. Hence $\h \cdot m = 0$. Then $v \mapsto m$ defines a morphism $ \Delta (\lambda)  \rightarrow M$ which splits the above sequence. 
\end{proof}

\begin{thm}\label{thm:hwc}
Category $\mc{O}$ is a highest weight category under the ordering $\le_{\mbf{c}}$.
\end{thm}

\begin{proof}
The only thing left to check is that we can choose a $\Delta$-filtration on the projective covers $P(\lambda)$ such that the conditions of Definition \ref{defn:hw} (3) are satisfied. By Theorem \ref{thm:projectives}, we can always choose some $\Delta$-filtration $0 = F_0 \subset F_1 \subset \cdots \subset F_m = P(\lambda)$ with $F_i / F_{i-1} \simeq \Delta(\mu_i)$. Since $P(\lambda)$ surjects onto $L(\lambda)$, we must have $\mu_m = \lambda$. I claim that one can always choose the $\mu_i$ such that $\mu_i \ge_{\mbf{c}} \mu_{i+1}$ for all $0 < i < m$. The proof is by induction on $i$. But first we remark that the fact that $P(\lambda)$ is indecomposable implies that all $\mu_i$ are comparable under $<_{\mbf{c}}$. Assume that $\mu_j \ge_{\mbf{c}} \mu_{j+1}$ for all $j < i$. If $\mu_{i} <_{\mbf{c}} \mu_{i+1}$ then it suffices to show that there is another $\Delta$-filtration of $P(\lambda)$ with composition factors $\mu_j'$ such that $\mu_j' = \mu_j$ for all $j \neq i,i+1$ and $\mu_i' =\mu_{i+1}$, $\mu_{i+1}' = \mu_i$. We have
$$
0 \rightarrow F_{i} \rightarrow F_{i+1} \rightarrow \Delta(\mu_{i+1}) \rightarrow 0
$$
which quotienting out by $F_{i-1}$ gives
$$
0 \rightarrow \Delta(\mu_{i}) \rightarrow F_{i+1} / F_{i-1} \rightarrow \Delta(\mu_{i+1}) \rightarrow 0.
$$
Lemma \ref{lem:trivext} implies that the above sequence splits. Hence $ F_{i+1} / F_{i-1} \simeq \Delta(\mu_{i}) \oplus \Delta(\mu_{i+1})$. Thus, we may choose $F_{i-1} \subset F_i' \subset F_i$ such that $F_i' / F_{i-1} \simeq \Delta(\mu_{i+1})$ and $F_{i+1} / F'_i \simeq \Delta(\mu_{i})$ as required. Hence the claim is proved. This means that $\mu_{m-1} \ge_{\mbf{c}} \lambda$. If $\mu_{m-1} =_{\mbf{c}} \lambda$ then Lemma \ref{lem:trivext} implies that $P(\lambda) / F_{m-2} \simeq \Delta(\mu_{m-1}) \oplus \Delta(\lambda)$ and hence the head of $P(\lambda)$ is not simple. This contradicts the fact that $P(\lambda)$ is a projective cover.  
\end{proof}

We now describe some consequences of the fact that $\mc{O}$ is a highest weight category. 

\begin{cor}\label{ex:Grothbasis}
The standard modules $\Delta(\lambda)$ are a $\Z$-basis of the Grothendieck group $K_0(\mc{O})$.
\end{cor}

\begin{proof}
Since $\mc{O}$ is a finite length, abelian category with finitely many simple modules, the image of those simple modules $L(\lambda)$ in $K_0(\mc{O})$ are a $\Z$-basis of $K_0(\mc{O})$. Let $k = | \Irr (W) |$ and define the $k$ by $k$ matrix $A = (a_{\lambda,\mu}) \in \mathbb{N}$ by
$$
[\Delta(\lambda)] = \sum_{\mu \in \Irr (W)} a_{\lambda,\mu} [L(\mu)].
$$
We order $\Irr (W) = \{ \lambda_1, \ds, \lambda_k \}$ so that $i > j$ implies that $\lambda_i \le_{\mbf{c}} \lambda_j$. Then property (2) of definition \ref{defn:hw} implies that $A$ is upper triangular with ones all along the diagonal. This implies that $A$ is invertible over $\Z$ and hence the $[\Delta(\lambda)]$ are a basis of $K_0(\mc{O})$. 
\end{proof}

A non-trivial corollary of Theorem \ref{thm:hwc} is that Bernstein-Gelfand-Gelfand (BGG) reciprocity holds in category $\mc{O}$. 

\begin{cor}[BGG-reciprocity]\label{cor:BGG}
For $\lambda, \mu \in \Irr W$, 
$$
(P(\lambda) : \Delta(\mu)) = [\Delta(\mu) : L(\lambda)].
$$
\end{cor}

We also have the following:

\begin{cor}\label{cor:finiteglobal}
The global dimension of $\mc{O}$ is finite. 
\end{cor}

\begin{proof}
As in the proof of Theorem \ref{thm:projectives}, it suffices to consider modules in the block $\mc{O}^{\bar{a}}$ for some $a \in \C$. Recall that we constructed a filtration of this category 
$$
\mc{O}^{\ge k_0} \subset \mc{O}^{\ge k_0 - i_1 } \subset \cdots \subset \mc{O}^{\ge k_0 - i_n} = \mc{O}^{\bar{a}},
$$
where $0 < i_1 < \ds < i_n$ are chosen so that each inclusion $\mc{O}^{\ge k_0 - i_r } \subset \mc{O}^{\ge k_0 - i_{r+1}}$ is proper. For all $\lambda \in \Irr (W)$, define $N(\lambda)$ to be the positive integer such that $L(\lambda) \in \mc{O}^{\ge k_0 - i_{N(\lambda)}}$ but $L(\lambda) \notin \mc{O}^{\ge k_0 - i_{N(\lambda)-1}}$. We claim that $\mathrm{p.d.}(\Delta(\lambda)) \le n - N(\lambda)$ for all $\lambda$. The proof of Theorem \ref{thm:projectives} showed that $\Delta(\lambda) = P(\lambda)$ for all $\lambda$ such that $N(\lambda) = n$. Therefore we may assume that the claim is true for all $\mu$ such that $N(\mu) > N_0$. Choose $\lambda$ such that $N(\lambda) = N_0$. Then, as shown in the proof of Theorem $3.18$, we have a short exact sequence 
$$
0 \rightarrow K \rightarrow P(\lambda) \rightarrow \Delta(\lambda) \rightarrow 0,
$$
where $K$ admits a filtration by $\Delta(\mu)$'s for $\lambda <_{\mbf{c}} \mu$. The inductive hypothesis, together with standard homological results e.g. Chapter $4$ of \cite{Weibel}, imply that $\mathrm{p.d.}(K) \le N_0 - 1$ and hence $\mathrm{p.d.}(\Delta(\lambda)) \le N_0$ as required. Now we show, again by induction, that $\mathrm{p.d.}(L(\lambda)) \le n + N(\lambda)$. It was shown in exercise $3.13$ that $\Delta(\lambda) = L(\lambda)$ for all $\lambda$ such that $N(\lambda) = 0$. Therefore, we may assume by induction that the claim holds for all $\mu$ such that $N(\mu) < N_0$. Assume $N(\lambda) = N_0$. We have a short exact sequence 
$$
0 \rightarrow R \rightarrow \Delta(\lambda) \rightarrow L(\lambda) \rightarrow 0,
$$  
where $R$ admits a filtration by simple modules $L(\mu)$ with $\mu <_{\mbf{c}} \lambda$. Hence, as for standard modules, we may conclude that $\mathrm{p.d.}(L(\lambda)) \le n + N(\lambda)$. Note that we have actually shown that the global dimension of $\mc{O}^{\bar{a}}$ is at most $2n$. 
\end{proof}

\begin{cor}
Category $\mc{O}$ is semi-simple if and only if $\Delta(\lambda) = L(\lambda)$ for all $\lambda \in \Irr (W)$. 
\end{cor}

\begin{proof}
Since $\Delta(\lambda)$ is indecomposable and $L(\lambda)$ a quotient of $\Delta(\lambda)$, it is clear that $\mc{O}$ semi-simple implies that $\Delta(\lambda) = L(\lambda)$ for all $\lambda \in \Irr (W)$. Conversely, if $\Delta(\lambda) = L(\lambda)$ then BGG reciprocity implies that $P(\lambda) = L(\lambda)$ for all $\lambda$ which implies that $\mc{O}$ is semi-simple. 
\end{proof}

\begin{exercise}\label{ex:Oss2}
Show that if $\mbf{c}_{\lambda} - \mbf{c}_{\mu} \notin \Z_{> 0}$ for all $\lambda \neq \mu \in \Irr (W)$ then category $\mc{O}$ is semi-simple. Conclude that $\mc{O}$ is semi-simple for generic parameters $\mbf{c}$. 
\end{exercise}

\subsection{Category $\mc{O}$ for $\Z_2$}

The idea of this section, which consists mainly of exercises, is simply to try and better understand category $\mc{O}$ when $W = \Z_2$. Recall that, in this case, $W = \langle s \rangle$ with $s^2 = 1$ and the defining relations for $\H_{\mbf{c}}(\Z_2)$ are $s x = - x s$, $s y = - y s$ and 
$$
[y,x] = 1 - 2 \mbf{c} s.
$$

\begin{exercise}\label{ex:O21}
\begin{enumerate}
\item For each $\mbf{c}$, describe the simple modules $L(\lambda)$ as quotients of $\Delta(\lambda)$. For which values of $\mbf{c}$ is category $\mc{O}$ semi-simple?
\item For each $\mbf{c}$, describe the partial ordering on $\Irr (W)$ coming from the highest weight structure on $\mc{O}$.
\end{enumerate}
\end{exercise}

For $W = \Z_2$ one can also relate representations of $\H_{\mbf{c}}(W)$ to certain representations of $U(\mf{sl}_2)$, the enveloping algebra of $\mf{sl}_2$ using the spherical subalgebra $\mbf{e} \H_{\mbf{c}}(W) \mbf{e}$. Recall that $\mf{sl}_2 = \C \{ E, F, H \}$ with $[H,E] = 2 E$, $[H,F] = - 2 F$ and $[E,F] = H$. 

\begin{exercise}\label{ex:sl2}
\begin{enumerate}
\item Show that $E \mapsto \frac{1}{2} \mbf{e} x^2$, $F \mapsto - \frac{1}{2} \mbf{e} y^2$ and $H \mapsto \mbf{e} xy + \mbf{e} \mbf{c}$ is a morphism of Lie algebras $\mf{sl}_2 \rightarrow \mbf{e} \H_{\mbf{c}}(\Z_2) \mbf{e}$, where the right hand side is thought of as a Lie algebra under the commutator bracket. This extends to a morphism of algebras $\phi_{\mbf{c}} : U(\mf{sl}_2) \rightarrow \mbf{e} \H_{\mbf{c}}(\Z_2) \mbf{e}$. 
\item The centre of $U(\mf{sl}_2)$ is generated by the Casimir $\Omega = \frac{1}{2} H^2 + E F + FE$. Find $\alpha \in \C$ such that $\Omega - \alpha \in \Ker \phi_{\mbf{c}}$. 
\end{enumerate}   
\end{exercise}

Exercise \ref{ex:sl2} shows that $\phi_{\mbf{c}}$ descends to a morphism $\phi_{\mbf{c}}' : U(\mf{sl}_2) / (\Omega - \alpha) \rightarrow \mbf{e} \H_{\mbf{c}}(\Z_2) \mbf{e}$. 

\begin{lem}\label{lem:sl2iso}
The morphism $\phi_{\mbf{c}}' : U(\mf{sl}_2) / (\Omega - \alpha) \rightarrow \mbf{e} \H_{\mbf{c}}(\Z_2) \mbf{e}$ is an isomorphism. 
\end{lem}

\begin{proof}
This is a filtered morphism, where $\mbf{e} \H_{\mbf{c}}(\Z_2) \mbf{e}$ is given the filtration as in section \ref{sec:Filt} and the filtration on $U(\mf{sl}_2)$ is defined by putting $E,F$ and $H$ in degree two. The associated graded of $\mbf{e} \H_{\mbf{c}}(\Z_2) \mbf{e}$ equals $\C[x,y]^{\Z_2} = \C[x^2,y^2, xy]$ and the associated graded of $U(\mf{sl}_2) / (\Omega - \alpha)$ is a quotient of $\C[E,F,H] / (\frac{1}{2} H^2 + 2 EF)$. Since $\gr_{\mc{F}} \mbf{e} \H_{\mbf{c}}(\Z_2) \mbf{e}$ is generated by $x^2 = \sigma(\mbf{e} x^2)$, $y^2 = \sigma ( \mbf{e} y^2)$ and $xy = \sigma(\mbf{e} xy)$, we see that $ \mbf{e} \H_{\mbf{c}}(\Z_2) \mbf{e}$ is generated by $\mbf{e} x^2$, $\mbf{e} y^2$ and $\mbf{e} xy$. Hence $\phi_{\mbf{c}}'$ is surjective. On the other hand, the composite 
$$
\C[E,F,H] / \left( \frac{1}{2} H^2 + 2 EF \right) \rightarrow \gr_{\mc{F}} U(\mf{sl}_2) / (\Omega - \alpha) \stackrel{\gr_{\mc{F}} \phi_{\mbf{c}}'}{\longrightarrow} \C[x^2,y^2, xy]
$$
is given by $E \mapsto \frac{1}{2} x^2$, $F \mapsto - \frac{1}{2} y^2$ and $H \mapsto xy$ is an isomorphism. This implies that $\C[E,F,H] / (\frac{1}{2} H^2 + 2 EF) \rightarrow \gr_{\mc{F}} U(\mf{sl}_2) / (\Omega - \alpha)$ is an isomorphism and so too is $\gr_{\mc{F}} \phi_{\mbf{c}}'$. Hence $\phi_{\mbf{c}}'$ is an isomorphism. 
\end{proof}

Assume now that $W$ is any complex reflection group. Recall that $\mbf{c}$ is said to be \textit{aspherical} if there exists some non-zero $M \in \Lmod{\H_{\mbf{c}}(W)}$ such that $\mbf{e} \cdot M = 0$. As in Lie theory, we have a ``Generalized Duflo Theorem'':

\begin{thm}\label{thm:GDT}
Let $J$ be a primitive ideal in $\H_{\mbf{c}}(W)$. Then there exists some $\lambda \in \Irr (W)$ such that 
$$
J = \ann_{\H_{\mbf{c}}(W)} (L(\lambda)).
$$
\end{thm}

\begin{exercise}\label{ex:GDT}
\begin{enumerate}
\item Using the Generalized Duflo Theorem, and arguments as in the proof of Corollary $1.17$ of lecture one, show that $\mbf{c}$ is aspherical if and only if there exists a simple module $L(\lambda)$ in category $\mc{O}$ such that $\mbf{e} \cdot L(\lambda) = 0$. 
\item Calculate the aspherical values for $W = \Z_2$. 
\end{enumerate} 
\end{exercise}

\subsection{Quivers with relations}\label{sec:quiver}

As noted previously, there exists a finite dimensional algebra $A$ such that category $\mc{O}$ is equivalent to $\Lmod{A}$. In this section, we'll try to construct $A$ in terms of quivers with relations when $W = \Z_2$. This section is included for those who know about quivers and can be skipped if you are not familiar with these things. When $W = \Z_2$ one can explicitly describe what the projective covers $P(\lambda)$ of the simple modules in category $\mc{O}$ are, though, as the reader will see, this is a tricky calculation.

The idea is to first use BGG reciprocity to calculate the rank of $P(\lambda)$ as a (free) $\C[x]$-module. We will only consider the case $\mbf{c} = \frac{1}{2} + m$ for some $m \in \Z_{\ge 0}$. The situation $\mbf{c} = - \frac{1}{2} - m$ is completely analogous. We have already seen in exercise \ref{ex:O21} that $[\Delta(\rho_1) : L(\rho_0) ] = 0$, $[\Delta(\rho_1) : L(\rho_1) ] = 1$, $[\Delta(\rho_0) : L(\rho_1) ] = 1$ and $[\Delta(\rho_0) : L(\rho_0) ] = 1$. Therefore, BGG reciprocity implies that $\Delta(\rho_0) = P(\rho_0)$ and $P(\rho_1)$ is free of rank two over $\C[x]$. Moreover, we have a short exact sequence
\begin{equation}\label{eq:ses}
0 \rightarrow \Delta(\rho_0) \rightarrow P(\rho_1) \rightarrow \Delta(\rho_1) = L(\rho_1) \rightarrow 0.
\end{equation}
As graded $\Z_2$-modules, we write $P(\rho_1) = \C[x] \o \rho_0 \oplus \C[x] \o \rho_1$, where $\C[x] \o \rho_0$ is identified with $\Delta(\rho_0)$. Then the structure of $P(\rho_1)$ is completely determined by the action of $x$ and $y$ on $\rho_1$:
$$
y \cdot (1 \o \rho_1) = f_1(x) \o \rho_0, \quad x \cdot (1 \o \rho_1) = x \o \rho_1 + f_0(x) \o \rho_0
$$
for some $f_0,f_1 \in \C[x]$. For the action to be well-defined we must check the relation $[y,x] = 1 - 2 \mbf{c} s$, which reduces to the equation
$$
y \cdot (f_0(x) \o \rho_0) = 0.
$$
Also $s(f_i) = f_i$ for $i = 0,1$. This implies that $f_0(x) = 1$. The second condition we require is that $P(\rho_1)$ is indecomposable (this will uniquely characterize $P(\rho_1)$ up to isomorphism). This is equivalent to asking that the short exact sequence (\ref{eq:ses}) does not split. Choosing a splitting means choosing a vector $\rho_1 + f_2(x) \o \rho_0 \in P(\rho_1)$ such that $y \cdot (\rho_1 + f_2(x) \o \rho_0) = 0$. One can check that this is always possible, except when $f_1(x) = x^{2 m}$. Thus we must take $f_1(x) = x^{2 m}$. This completely describes $P(\rho_1)$ up to isomorphism. 

Recall that a finite dimensional $\C$-algebra $A$ is said to be \textit{basic} if the dimension of all simple $A$-modules is one. Every basic algebra can be described as a quiver with relations. One way to reconstruct $A$ from $\Lmod{A}$ is via the isomorphism
$$
A = \End_A \left( \bigoplus_{\lambda \in \Irr (A)} P(\lambda) \right),
$$
where $\Irr (A)$ is the set of isomorphism classes of simple $A$-modules and $P(\lambda)$ is the projective cover of $\lambda$. Next we will construct a basic $A$ in terms of a quiver with relations such that $\Lmod{A} \simeq \mc{O}$. The first step in doing this is to use BGG-reciprocity to calculate the dimension of $A$. Again, we will assume that $\mbf{c} = \frac{1}{2} + m$ for some $m \in \Z_{\ge 0}$. The case $\mbf{c} = -\frac{1}{2} - m$ is similar, and all other cases are trivial. We need to describe $A = \End_{\H_{\mbf{c}}(\Z_2)}(P(\rho_0) \oplus P(\rho_1))$. Using the general formula $\dim \Hom_{\H_{\mbf{c}}(W)}(P(\lambda),M) = [M : L(\lambda)]$, and BGG reciprocity, we see that 
\begin{equation}\label{eq:dim1}
\dim \End_{\H_{\mbf{c}}(\Z_2)}(P(\rho_0)) = 1, \quad  \dim \End_{\H_{\mbf{c}}(\Z_2)}(P(\rho_1)) = 2
\end{equation}
\begin{equation}\label{eq:dim2}
\dim \Hom_{\H_{\mbf{c}}(\Z_2)}(P(\rho_0),P(\rho_1)) = 1, \quad \dim \Hom_{\H_{\mbf{c}}(\Z_2)}(P(\rho_1),P(\rho_0)) = 1.
\end{equation}
Hence $\dim A = 5$. The algebra $A$ will equal $\C Q / I$, where $Q$ is some quiver and $I$ an admissible ideal\footnote{Recall that an ideal $I$ in a finite dimensional algebra $B$ is said to be \textit{admissible} if there exists some $m \ge 2$ such that $\mathrm{rad}(B)^m \subset I \subset \mathrm{rad}(B)^2$}. The vertices of $Q$ are labeled by the simple modules in $\mc{O}$, hence there are two: $e_0$ and $e_1$ (corresponding to $L(\rho_0)$ and $L(\rho_1)$ respectively). The number of arrows from $e_0$ to $e_1$ equals $\dim \Ext_{\H_{\mbf{c}}(\Z_2)}^1 (L(\rho_0),L(\rho_1))$ and the number of arrows from $e_1$ to $e_0$ equals $\dim \Ext_{\H_{\mbf{c}}(\Z_2)}^1 (L(\rho_1),L(\rho_0))$. Hence there is one arrow $e_1 \leftarrow e_0 : a$ and one arrow $e_0 \leftarrow e_1 : b$. The projective module $P(\rho_0)$ will be a quotient of 
$$
\C Q e_0 = \C \{ e_0, a e_0 = a, ba, aba, \ds \},
$$
and similarly for $P(\rho_1)$. Equations (\ref{eq:dim1}) imply that $ba = (ab)^2 = 0$ in $A$ (note that we cannot have $e_0 - \alpha ba = 0$ etc. because the endomorphism ring of an indecomposable is a local ring). Hence $A$ is a quotient of $\C Q / I$, where $I = \langle ba \rangle$. But $\C Q / I$ has a basis given by $\{ e_0, e_1, a,b,ab \}$. Hence $\dim \C Q / I = 5$ and the natural map $\C Q / I \rightarrow A$ is an isomorphism. Thus,

\begin{lem}
Let $Q$ be the quiver with vertices $\{e_0, e_1 \}$ and arrows $\{ e_1 \stackrel{a}{\longleftarrow} e_0, \ e_0 \stackrel{b}{\longleftarrow} e_1 \}$. Let $I$ be the admissible ideal $\langle ba \rangle$. Then, $\mc{O} \simeq \Lmod{\C Q / I }$. 
\end{lem}

\begin{remark}
The above lemma shows that category $\mc{O}$ at $\mbf{c} = \frac{1}{2} + m$ is equivalent to the regular block of category $\mc{O}$ for the Lie algebra $\mathfrak{sl}_2$; see \cite[Section 5.1.1 ]{StroppelQuivers}. One can prove directly that the two category $\mc{O}$'s are equivalent using Lemma \ref{lem:sl2iso}. 
\end{remark}

\subsection{Additional remark}

\begin{itemize}
\item The results of this lecture all come from (at least) one of the papers \cite{DunklOpdam}, \cite{GuayO} or \cite{GGOR}.

\item Theorem \ref{thm:hwc} is shown in \cite{GGOR}. 

\item The fact that BGG reciprocity, Corollary \ref{cor:BGG}, follows from Theorem \ref{thm:hwc} is shown in \cite[Proposition 3.3]{GGOR}.

\item The definition given in \cite[Definition 3.1]{CPS} is dual to the one give in Definition \ref{defn:hw}. It is also given in much greater generality.

\item The generalized Duflo theorem, Theorem \ref{thm:GDT}, is given in \cite{Primitive}. 
\end{itemize}

\newpage

\section{The symmetric group}\label{sec:three}

In this lecture we concentrate on category $\mc{O}$ for $W = \s_n$, the symmetric group. The reason for this is that category $\mc{O}$ is much better understood for this group than for other complex reflection groups, though there are still several open problems. For instance, it is known for which parameters $\mbf{c}$ the algebra $\H_{\mbf{c}}(\s_n)$ admits finite dimensional representations, and for any such parameter, exactly how many non-isomorphic simple modules there are, see \cite{Finitedimreps}. It turns out that $\H_{\mbf{c}}(\s_n)$ admits at most one finite dimensional, simple module. In fact, we now have a good understanding, \cite{Wilcox}, of the ``size'' (i.e. Gelfand-Kirillov dimension) of all simple modules in category $\mc{O}$. 

But, in this lecture, we will concentrate on the main problem mentioned in lecture two, that of calculating the character $\mathrm{ch}(L(\lambda))$ of the simple modules.

\subsection{Outline of the lecture} The first thing to notice is that, since we can easily write down the character of the standard modules $\Delta(\lambda)$, this problem is equivalent to the problem of calculating the multiplicities of simple modules in a composition series for standard modules; the "multiplicities problem". That is, we wish to find a combinatorial algorithm for calculating the numbers
$$
[\Delta(\lambda) : L(\mu)], \quad \forall \ \lambda, \mu \in \Irr (\s_n).
$$
For the symmetric group, we now also have a complete answer to this question: the multiplicities are given by evaluating at one the transition matrices between standard and canonical basis of a certain Fock space. In order to prove this remarkable result, one needs to introduce a whole host of new mathematical objects, including several new algebras. This can make the journey long and difficult. So we begin by outlining the whole story, so that the reader doesn't get lost along the way. The result relies up on work of several people, namely Rouquier, Varagnolo-Vasserot and Leclerc-Thibon. 

Motivated via quantum Schur-Weyl duality, we begin by defining the $\nu$-Schur algebra. This is a finite dimensional algebra. The category of finite dimensional modules over the $\nu$-Schur algebra is a highest weight category. Rouquier's equivalence says that category $\mc{O}$ for the rational Cherednik algebra of type $A$ is equivalent to the category of modules over the $\nu$-Schur algebra. Thus, we transfer the multiplicity problem for category $\mc{O}$ to the corresponding problem for the $\nu$-Schur algebra. 

The answer to this problem is known by a result of Vasserot and Varagnolo. However, their answer comes from a completely unexpected place. We forget about $\nu$-Schur algebras for a second and consider instead the Fock space $\mc{F}_q$ (a vector space over $\Q(q)$). This is an infinite dimensional representation of $\mc{U}_{q}(\widehat{\mf{sl}}_r)$, the quantum group associated to the \textit{affine} Lie algebra $\widehat{\mf{sl}}_r$. On the face of it, $\mc{F}_q$ has nothing to do with the $\nu$-Schur algebra, but bare with me! 

The Fock space $\mc{F}_q$ has a standard basis labeled by partitions. It was shown by Leclerc and Thibon that it also admits a \textit{canonical basis}, in the sense of Lusztig, also labeled by partitions. Hence there is a "change of basis" matrix that relates these two basis. The entries $d_{\lambda,\mu}(q)$ of this matrix are elements of $\Q(q)$. Remarkably, it turns out that they actually belong to $\Z[q]$. 

What Varagnolo and Vasserot showed was that the multiplicity of the simple module (for the $\nu$-Schur algebra) labeled by $\lambda$ in the standard module labeled by $\mu$ is given by $d_{\lambda',\mu'}(1)$. Thus, to calculate the numbers $[\Delta(\lambda) : L(\mu)]$, and hence the character of $L(\lambda)$, what we really need to do is calculate the change of basis matrix for the Fock space $\mc{F}_q$.   

\subsection{The rational Cherednik algebra associated to the symmetric group}

Recall from example \ref{example:symmetric} that the rational Cherednik algebra associated to the symmetric group $\s_n$ is the quotient of 
$$
T(\C^{2n}) \rtimes \s_n = \C \langle x_1,\ds, x_n , y_1, \ds, y_n \rangle \rtimes \s_n
$$
by the relations
$$
[x_i,x_j] = 0, \quad [y_i,y_j] = 0, \quad \forall \, i,j,
$$
$$
[y_i, x_j ] = \mbf{c} s_{ij}, \quad \forall \, i \neq j, 
$$
and
$$
[y_i,x_i] = 1 - \mbf{c} \sum_{j \neq i} s_{ij}. 
$$
Since the standard and simple modules in category $\mc{O}$ are labeled by the irreducible representations of $\s_n$, we begin by recalling the parameterization of these representations.

\subsection{Representations of $\s_n$}

It is a classical result, going back to Schur, that the irreducible representations of the symmetric group over $\C$ are naturally labeled by partitions of $n$. Therefore, we can (and will) identify $\Irr (\s_n)$ with $\mc{P}_n$, the set of all partitions of $n$ and denote by $\lambda$ both a partition of $n$ and the corresponding representation of $\s_n$. For more on the construction of the representations of $\s_n$, see \cite{FultonYOungTableaux}.

\begin{example}
The partition $(n)$ labels the trivial representation and $(1^n)$ labels the sign representation. The reflection representation $\h$ is labeled by $(n-1,1)$. More generally, each of the representations $\bigwedge^i \h$ is an irreducible $\s_n$-module and is labeled by $(n-i,1^i)$; see figure \ref{fig:partw}. 
\begin{figure}\label{fig:partw}
\begin{tikzpicture}

\draw[help lines] (0,0) grid (1,4);
\draw[help lines] (0,0) grid (5,1);

\node at (0.5,0.5) {$1$};
\node at (1.5,0.5) {$2$};
\node at (2.5,0.5) {$\cdots$};
\node at (3.5,0.5) {$\cdots$};
\node at (4.5,0.5) {$n-i$};
\node at (0.5,1.5) {$2$};
\node at (0.5,2.5) {$\vdots$};
\node at (0.5,3.5) {$i$};
\node at (-1.5,2) {$(n-i,1^i) =$};
\end{tikzpicture}
\caption{The partition $(n-i,1^i)$ corresponding to the irreducible $\s_n$-module $\wedge^i \h$.}
\end{figure}
Note that the trivial representation is $\bigwedge^0 \h$ and the sign representation is just $\bigwedge^{n-1} \h$. 
\end{example}

\subsection{Partitions} Associated to partitions is a wealth of beautiful combinatorics. We'll need to borrow a little of this combinatorics. Let $\lambda = (\lambda_1, \ds, \lambda_k)$ be a partition. We visualize $\lambda$ as a certain array of boxes, called a \textit{Young tableau}, as in the example\footnote{The numbers in the boxes are the residues of $\lambda$ modulo $3$, see section \ref{sec:affinefock}.} $\lambda = (4,3,1)$: 
$$
\Yboxdim18pt
\young(1:::,201:,0120)
$$
To be precise, the Young diagram of $\lambda$ is $Y(\lambda) := \{ (i,j) \in \Z^2 \, | \, 1 \le j \le k, \, 1 \le i \le \lambda_j \} \subset \Z^2$.

\begin{example}
There is a natural basis of the irreducible $\s_n$-module labeled by the partition $\lambda$, given by the set of all standard tableau of shape $\lambda$. Here a \textit{standard tableau} is a filling of the Young tableau of $\lambda$ by $\{1, \ds, n\}$ such that the numbers along the row and column, read from left to right, and bottom to top, are increasing. For instance, if $n = 5$ and $\lambda = (3,2)$, then $\dim \lambda = 5$, and has a basis labeled by all standard tableau, 
 $$
\Yboxdim18pt
\young(45:,123)  \quad \young(25:,134) \quad \young(24:,135) \quad \young(35:,124) \quad \young(34:,125)
$$

\end{example}

\subsection{The $\nu$-Schur algebra}

Recall that the first step on the journey is to translate the multiplicity problem for category $\mc{O}$ into the corresponding problem for the $\nu$-Schur algebra. By Weyl's complete reducibility theorem, the category $\mc{C}_n$ of finite-dimensional representations of $\mf{gl}_n$, or equivalently of its enveloping algebra $\mc{U}(\mf{gl}_n)$, is semi-simple. The simple modules in this category are the highest weight modules $L_{\lambda}$, where $\lambda \in \Z^{n}$ such that $\lambda_i - \lambda_{i+1} \ge 0$ for all $1 \le i \le n-1$. The set $\mc{P}(n)$ of all partitions with length at most $n$ can naturally be considered as a subset of this set. Let $V$ denote the vectorial representation of $\mf{gl}_n$. For each $d \ge 1$ there is an action of $\mf{gl}_n$ on $V^{\o d}$. The symmetric group also acts on $V^{\o d}$ on the right by 
$$
(v_1 \o \cdots \o v_d) \cdot \sigma = v_{\sigma^{-1}(1)} \o \cdots \o v_{\sigma^{-1}(d)}, \quad \forall \ \sigma \in \s_d. 
$$
It is known that these two actions commute. Thus, we have homomorphisms $\phi_d :  \mc{U}(\mf{gl}_n) \rightarrow \End_{\C \s_d}(V^{\o d})$ and $\psi_d : \C \s_d \rightarrow \End_{\mc{U}(\mf{gl}_n)}(V^{\o d})^{op}$. Schur-Weyl duality says that 

\begin{prop}
The homomorphisms $\phi_d$ and $\psi_d$ are surjective for all $d \ge 1$. 
\end{prop}

Let $\mathsf{S}(n,d) = \End_{\C \s_d}(V^{\o d})$ be the image of $\phi_d$. It is called the Schur algebra. We denote by $\mc{C}_n(d)$ the full subcategory of $\mc{C}_n$ consisting of all modules whose composition factors are of the form $L_{\lambda}$ for $\lambda \in \mc{P}_d(n)$, where $\mc{P}_d(n)$ is the set of all partitions of $d$ that belong to $\mc{P}(n)$. It is easy to check that $[V^{\o d} : L_{\lambda}] \neq 0$ if and only if $\lambda \in \mc{P}_d(n)$. Moreover, it is known that $\mc{C}_n(d) \simeq \Lmod{\mathsf{S}(n,d)}$. \\

The above construction can be quantized. Let $\nu \in \C^\times$. Then, the \textit{quantized enveloping algebra} $\mc{U}_{\nu}(\mf{gl}_n)$ is a deformation of $\mc{U}(\mf{gl}_n)$. The quantum enveloping algebra $\mc{U}_{\nu}(\mf{gl}_n)$ still acts on $V$. The group algebra $\C \s_n$ also has a natural deformation, the \textit{Hecke algebra} of type $A$, denoted $\mc{H}_{\nu}(d)$. This algebra is described in example \ref{examp:HeckeA}. As one might expect, it is also possible to deform the action of $\C \s_d$ on $V^{\o d}$ to an action of $\mc{H}_{\nu}(d)$ in such a way that this action commutes with the action of $\mc{U}_{\nu}(\mf{gl}_n)$. The quantum analogue of Schur-Weyl duality, see \cite{DPSQuantum}, says

\begin{prop}
We have surjective homomorphisms
$$
\phi_d :  \mc{U}_{\nu}(\mf{gl}_n) \rightarrow \End_{\mc{H}_{\nu}(d)}(V^{\o d}) \quad \textrm{and} \quad \psi_d : \mc{H}_{\nu}(d) \rightarrow \End_{\mc{U}_{\nu}(\mf{gl}_n)}(V^{\o d})^{op}
$$
for all $d \ge 1$. 
\end{prop}

The image $ \End_{\mc{H}_{\nu}(d)}(V^{\o d})$ of $\phi_d$ is the \textit{$\nu$-Schur algebra}, denoted $\mathsf{S}_{\nu}(n,d)$. The category $\mc{C}_{n,\nu}$ of finite-dimensional representations of $\mc{U}_{\nu}(\mf{gl}_n)$ is no longer semi-simple, in general. However, the simple modules in this category are still labeled $L_{\lambda}$ for $\lambda \in \Z^{n}$ such that $\lambda_i - \lambda_{i+1} \ge 0$ for all $1 \le i \le n-1$. Moreover, if we let $\mc{C}_{n,\nu}(d)$ denote the full subcategory of $\mc{C}_{n,\nu}$ consisting of all modules whose composition factors are $L_{\lambda}$ for $\lambda \in \mc{P}_d(n)$ then we again have $\mc{C}_{n,\nu} = \Lmod{\mathsf{S}_{\nu}(n,d)}$.  It is known that $\Lmod{\mathsf{S}_{\nu}(n,d)}$ is a highest weight category with standard modules $W_{\lambda}$.  

\subsection{Rouquier's equivalence}

As explained at the start of the lecture, in order to calculate the multiplicities
$$
m_{\lambda,\mu} = [\Delta(\lambda) : L(\mu)]
$$
one has to make a long chain of connections and reformulations of the question, the end answer relies on several remarkable results. The first of these is Rouquier's equivalence, the proof of which relies in a crucial way on the $\KZ$-functor introduced in the next lecture. 

\begin{thm}\label{thm:rouquierequiv}
Let\footnote{Note that Rouquier's rational Cherednik algebra is parameterized by $h = - \mbf{c}$.} $\mbf{c} \in \Q_{\ge 0}$ and set $\nu = \mathrm{exp}(2 \pi \sqrt{-1} \mbf{c})$. Then there is an equivalence of highest weight categories 
$$
\Psi : \mc{O} \stackrel{\sim}{\longrightarrow} \Lmod{\mathsf{S}_{\nu}(n)},
$$
such that $\Psi(\Delta(\lambda)) = W_{\lambda}$ and $\Psi(L(\lambda)) = L_{\lambda}$. 
\end{thm}

Thus, to calculate $m_{\lambda,\mu}$ it suffices to try and describe the numbers $[W_{\lambda} : L_{\mu}]$. To do this, we turn now to quantum affine enveloping algebra $\mc{U}_{q}(\widehat{\mf{sl}}_r)$ and the Fock space $\mc{F}_q$. 

\subsection{The quantum affine enveloping algebra}

Let $q$ be an indeterminate and let $I$ be the set $\{ 1, \ds, r \}$. We now we turn our attention to another quantized enveloping algebra, this time of an affine Lie algebra. The \textit{quantum affine enveloping algebra} $\mc{U}_{q}(\widehat{\mf{sl}}_r)$ is the $\Q(q)$-algebra generated by $E_i,F_i,K_i^{\pm 1}$ for $i \in I$ and satisfying the relations
$$
K_i K_i^{-1} = K_i^{-1} K_i = 1, \quad K_i K_j = K_j K_i, \quad \forall \ 1 \le i , j \le r
$$
\begin{equation}\label{eq:aij}
K_i E_j = q^{a_{i,j}} E_j K_i, \quad K_i F_j = q^{- a_{i,j}} F_j K_i, \quad \forall \ 1 \le i , j \le r
\end{equation}
$$
[E_i,F_j] = \delta_{i,j} \frac{K_i - K_i^{-1}}{q - q^{-1}}, \quad \forall \ 1 \le i , j \le r
$$
and the quantum Serre relations
\begin{equation}\label{eq:Es}
E_i^2 E_{i \pm 1} - (q + q^{-1}) E_i E_{i \pm 1} E_i + E_{i \pm 1} E_i^2 = 0, \quad \forall \ 1 \le i \le r
\end{equation}
\begin{equation}\label{eq:Fs}
F_i^2 F_{i \pm 1} - (q + q^{-1}) F_i F_{i \pm 1} F_i + F_{i \pm 1} F_i^2 = 0, \quad \forall \ 1 \le i \le r
\end{equation}
where the indicies in (\ref{eq:Es}) and (\ref{eq:Fs}) are taken modulo $r$ so that $0 = r$ and $r + 1 = 1$. In (\ref{eq:aij}), $a_{i,i} = 2$, $a_{i,i \pm 1} = -1$ and $0$ otherwise. In the case $r = 2$, we take $a_{i,j} = -2$ if $i \neq j$.  

\subsection{The $q$-deformed Fock space}\label{sec:affinefock} Let $\mathcal{F}_{q}$ be the \textit{level one Fock space} for $\mathcal{U}_{q}(\widehat{\mf{sl}}_r)$. It is a $\mathbb{Q}(q)$-vector space with standard basis $\{ | \lambda \rangle \}$, labeled by all partitions $\lambda$. The action of $\mathcal{U}_{q}(\widehat{\mf{sl}}_r)$ on $\mathcal{F}_{q}$ is combinatorially defined. Therefore, to describe it we need a little more of the language of partitions.

Let $\lambda$ be a partition. The \textit{content} of the box $(i,j) \in Y(\lambda)$ is $\cont(i,j) := i - j$. A \textit{removable} box is a box on the boundary of $\lambda$ which can be removed, leaving a partition of $|\lambda | - 1$. An \textit{indent} box is a concave corner on the rim of $\lambda$ where a box can be added, giving a partition of $|\lambda| + 1$. For instance, $\lambda = (4,3,1)$ has three removable boxes (with content $2,-1$ and $-3$), and four indent boxes (with content $3,1,-2$ and $-4$). If $\gamma$ is a box of the Young tableaux corresponding to the partition $\lambda$ then we say that the \textit{residue} of $\gamma$ is $i$, or we say that $\gamma$ is an $i$-box of $\lambda$, if the content of $\gamma$ equals $i$ modulo $r$.  Let $\lambda$ and $\mu$ be two partitions such that $\mu$ is obtained from $\lambda$ by adding a box $\gamma$ with residue $i$; see figure \ref{fig:part}. We define
\begin{align*}
N_i(\lambda) = & | \{ \textrm{indent $i$-boxes of $\lambda$} \} | - | \{ \textrm{removable $i$-boxes of $\lambda$} \} |, \\
N_i^l(\lambda,\mu) = & | \{ \textrm{indent $i$-boxes of $\lambda$ situated to the \textit{left} of $\gamma$ (not counting $\gamma$) } \} | \\
 &  - | \{ \textrm{removable $i$-boxes of $\lambda$ situated to the \textit{left} of $\gamma$} \} |,\\ 
N_i^r(\lambda,\mu) = & | \{ \textrm{indent $i$-boxes of $\lambda$ situated to the \textit{right} of $\gamma$ (not counting $\gamma$) } \} | \\
 &  - | \{ \textrm{removable $i$-boxes of $\lambda$ situated to the \textit{right} of $\gamma$} \} |,
\end{align*}

\begin{figure}\label{fig:part}
\begin{tikzpicture}

\draw [green!10!white,fill=green!10!white] (0,0) -- (0,4) -- (2,4) -- (2,2) -- (5,2) -- (5,0) -- (0,0);

\draw [red!10!white,fill=red!10!white] (2,2) -- (2,3) -- (3,3) -- (3,2) -- (2,2);

\draw[help lines] (0,0) grid (2,4);
\draw[help lines] (2,0) grid (5,2);
\draw[help lines] (2,2) grid (3,3);

\draw [thick,gray,->] (5,4) to [out=-180,in=70] (2.5,2.5);
\node at (5.3,4.1) {$\gamma$};

\node at (1.5,1.5) {$\lambda$};

\node at (-1,2) {$\mu =$};

\draw [thick,] (1,4.5) to [out=-90,in=90] (0.1,4.1);
\draw [thick] (1,4.5) to [out=-90,in=90] (1.9,4.1);

\node at (0,4.7) {Nodes to the left of $\gamma$};

\draw [thick,] (4,2.5) to [out=-90,in=90] (3.1,2.1);
\draw [thick] (4,2.5) to [out=-90,in=90] (4.9,2.1);

\node at (5.5,2.7) {Nodes to the right of $\gamma$};

\end{tikzpicture}
\caption{Nodes to the left and right of $\gamma$. The total partition is $\mu$, whist the green shaded subpartition is $\lambda$.}
\end{figure}
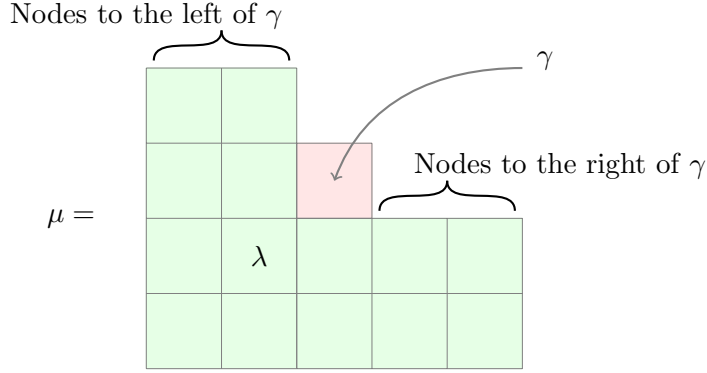

Then,
$$
F_i | \lambda \rangle = \sum_{\mu} q^{N^r_i(\lambda,\mu)} | \mu \rangle, \quad E_i | \mu \rangle = \sum_{\lambda} q^{N^l_i(\lambda,\mu)} | \lambda \rangle,
$$
where, in each case, the sum is over all partitions such that $\mu / \lambda$ is a $i$-node, and
$$
K_i | \lambda \rangle = q^{N_i(\lambda)} | \lambda \rangle.
$$
See \cite[Section 4.2]{LLTFock} for further details. 

\begin{example}
Let $\lambda  = (5,4,1,1)$ and $r = 3$ so that the Young diagram with residues of $\lambda$ is
$$
\Yboxdim18pt
\young(0::::,1::::,2012:,01201)
$$
Then $F_2 | \lambda \rangle = |(6,4,1,1) \rangle  + |(5,4,2,1) \rangle  + q | (5,4,1,1,1) \rangle$, $E_2 | \lambda \rangle = q^2 | (5,3,1,1) \rangle$, and $K_2 | \lambda \rangle = q^2 | (5,3,1,1) \rangle$. 
\end{example}

\begin{exercise}\label{ex:66311}
Let $\lambda  = (6,6,3,1,1)$ and $r = 4$. Calculate the action of $F_2$, $E_4$ and $K_1$ on $| \lambda \rangle$ (hint: draw a picture!).
\end{exercise}

\begin{exercise}
Write a program that calculates the action of the operators $K_i,F_i,E_i$ on $\mc{F}_q$. 
\end{exercise}

In fact, there is an action of a larger algebra, the quantum affine enveloping algebra $\mathcal{U}_{q}(\widehat{\mf{gl}}_r) \supset \mathcal{U}_{q}(\widehat{\mf{sl}}_r)$ on the space $\mc{F}_q$. The key point for us is that $\mc{F}_q$ is an \textit{irreducible} highest weight representation of $\mathcal{U}_{q}(\widehat{\mf{gl}}_r)$, with highest weight $| \emptyset \rangle$. As a $\mathcal{U}_{q}(\widehat{\mf{sl}}_r)$-module, the Fock space is actually a direct sum of infinitely many irreducible highest weight modules, see \cite{qDecomp}. 
 
\subsection{The bar involution} The key to showing that the Fock space $\mc{F}_q$ admits a canonical basis is to construct on it a sesquilinear involution, that is compatible in a natural way the the sesquilinear involution on $\mathcal{U}_{q}(\widehat{\mf{gl}}_r)$. However, in order to be able to do this we must first write the standard basis of $\mc{F}_q$ in terms of infinite $q$-wedges. The original motivation for doing this came from the Boson-Fermion correspondence in mathematical physics. Once we have constructed the involution, the existence of a canonical basis comes from general theory developed by Lusztig. 

The $\Q$-linear involution $q \mapsto \ol{q} := q^{-1}$ of $\Q(q)$ extends to an involution $v \mapsto \ol{v}$ of $\mc{F}_q$. In order to describe this involution, it suffices to say how to calculate $\overline{| \lambda \rangle}$. We begin by noting that a partition can also be describe as an infinite wedge as follows.

\begin{exercise}\label{ex:part}
Let $\mc{J}$ be the set of strictly decreasing sequences $I = (i_1,i_2,\ds )$ such that $i_k = - k + 1$ for $k \gg 0$. Show that there is a natural bijection between $\mc{J}$ and $\mc{P}$, the set of all partitions. This bijection sends $\mc{J}_n = \{ I \in \mc{J} \ | \ \sum_k (i_k + k - 1) = n \}$ to $\mc{P}_n$ the set of all partitions of $n$. 
\end{exercise}

Thus, to a partition $I \in \mc{J}$ we associate the infinite wedge 
$$
u_I = u_{i_1} \wedge u_{i_2} \wedge u_{i_3} \wedge \ds, 
$$
for instance
$$
u_{(3^2,2)} = u_3 \wedge u_2 \wedge u_0 \wedge u_{-3} \wedge u_{-4} \wedge \ds
$$
An infinite wedge is normally ordered if it equals $u_I$ for some $I \in \mc{J}$. Just as for usual wedge products, there is a ``normal ordering rule'' for the $q$-deformed wedge product. However, this normal ordering rule depends in a very nontrivial way on the integer $r$ (and more generally, on the level $l$ of the Fock space). In order to describe the normal ordering rule, it suffices to say how to swap two adjacent $u_i$'s. Let $i < j$ be integers with $j - i \equiv m \ \mathrm{mod} \ r$ for some $0 \le m < r$. If $m = 0$ then 
$$
u_i \wedge u_j = - u_j \wedge u_i,
$$
and otherwise  
\begin{align*}
u_i \wedge u_j = & - q^{-1} u_j \wedge u_i + (q^{-2} - 1) [u_{j - m} \wedge u_{i + m} - q^{-1} u_{j - r} \wedge u_{i + r}\\
 & + q^{-2} u_{j - m - r} \wedge u_{i + m + r} - q^{-3} u_{j - 2r} \wedge u_{i + 2r} + \cdots ]
\end{align*}
where the sum continues only as long as the terms are normally ordered. Let $I \in \mc{J}$ and write $\alpha_{r,k}(I)$ for the number of pairs $(a,b)$ with $1 \le a < b \le k$ and\footnote{There is a typo in definition of $\alpha_{r,k}(I)$ in \cite{LT}} $i_a - i_b \not\equiv 0 \ \textrm{mod} \ r$. 

\begin{prop}
For $k \ge n$, the $q$-wedge 
$$
\ol{u_I} = (-1)^{{k \choose 2 }} q^{\alpha_{r,k}(I)} u_{i_k} \wedge u_{i_{k-1}} \wedge \cdots \wedge u_{i_1} \wedge u_{i_{k+1}} \wedge u_{i_{k+2}} \wedge \cdots 
$$
is independent of $k$. 
\end{prop}

Therefore we can define a semi-linear map, $v \mapsto \ol{v}$ on $\mc{F}_q$ by 
$$
\ol{f(q) u_{I}} = f(q^{-1}) \ol{u_I}.
$$
This is actually an involution on $\mc{F}_q$. For $\mu \vdash n$ define
$$
\ol{| \mu \rangle} = \sum_{\lambda \vdash n} a_{\lambda,\mu}(q) | \lambda \rangle.
$$

\subsection{Canonical basis}

In order to describe the properties of the polynomials $a_{\lambda,\mu}(q)$ and also define the canonical basis, we need a couple more properties of partitions. There is a natural partial ordering on $\mc{P}_n$, the set of all partitions of $n$, which is the \textit{dominance ordering} and is defined by $\lambda \unlhd \mu$ if and only if
$$
\lambda_1 + \cdots + \lambda_k \le \mu_1 + \cdots + \mu_k , \quad \forall \ k.
$$ 
We also require the notion of $r$-rim-hooks and $r$-cores. An \textit{$r$-rim-hook} of $\lambda$ is a \textit{connected} skew partition $\lambda \backslash \mu$ of length $r$ that does not contain the subpartition $(2,2)$ i.e. it is a segment of length $r$ of the edge of $\lambda$. For example, $\{ (2,2),(3,2),(3,1),(4,1) \}$ is a $4$-rim-hook of $(4,3,1)$. The \textit{$r$-core} of $\lambda$ is the partition $\mu \subset \lambda$ obtained by removing, one after another, all possible $r$-rim-hooks of $\lambda$. It is known that the $r$-core is independent of the order in which the hooks are removed. For example, the $3$-core of $(4,3,1)$ is $(2)$. 

\begin{exercise}
Write a program that calculates the $r$-core of a partition. 
\end{exercise}

Then it is known, \cite[Theorem 3.3]{LT}, that the polynomials $a_{\lambda,\mu}(q)$ have the following properties. 

\begin{thm}
Let $\lambda,\mu \vdash n$.
\begin{enumerate}
\item $a_{\lambda,\mu}(q) \in \Z[q,q^{-1}]$.
\item $a_{\lambda,\mu}(q) = 0$ unless $\lambda \lhd \mu$ and $\lambda,\mu$ have the same $r$-core.
\item $a_{\lambda,\lambda}(q) = 1$. 
\item $a_{\lambda,\mu}(q) = a_{\mu',\lambda'}(q)$. 
\end{enumerate}
\end{thm}

\begin{example}
If $r = 3$ then
$$
\ol{|(4,3,1) \rangle} = |(4,3,1) \rangle + (q - q^{-1}) |(3,3,1,1) \rangle + (-1 + q^{-2}) |(2,2,2,2) \rangle + (q^2 - 1) |(2,1,1,1,1,1,1) \rangle.
$$
\end{example}

Leclerc and Thibon showed: 

\begin{thm}[\cite{LT}, Theorem 4.1]
There exist \textit{canonical basis} $\{ \mc{G}^+(\lambda) \}$ and $\{ \mc{G}^-(\lambda) \}$ characterized by 
\begin{enumerate}
\item $\ol{\mc{G}^+(\lambda)} = \mc{G}^+(\lambda)$, $\ol{\mc{G}^-(\lambda)} = \mc{G}^-(\lambda)$.
\item $\mathcal{G}^+(\lambda) \equiv | \lambda \rangle \ \mod \ q \Z[q]$ and $\mathcal{G}^-(\lambda) \equiv | \lambda \rangle \ \mod \ q^{-1} \Z[q^{-1}]$.
\end{enumerate}
\end{thm}

It seems that the above result has nothing to do with the fact that $\mc{F}_q$ is a $\mathcal{U}_{q}(\widehat{\mf{sl}}_r)$-module. All that is required is that there is some involution defined on the space. However, for an arbitrary involution there is no reason to expect a canonical basis to exist (and, indeed, one can check with small examples that it does not). Of course, the involution used by Leclerc and Thibon is not arbitrary. There is a natural involution on the algebra $\mathcal{U}_{q}(\widehat{\mf{gl}}_r)$. Since $\mc{F}_q$ is irreducible as a $\mathcal{U}_{q}(\widehat{\mf{gl}}_r)$-module, there is a unique involution on $\mc{F}_q$ such that $\ol{F| \emptyset \rangle} = \ol{ F} | \emptyset \rangle$ for all $F \in \mathcal{U}_{q}(\widehat{\mf{gl}}_r)$. It is this involution that Leclerc and Thibon use, though as we have seen they are able to give an explicit definition of this involution. Then, it follows from a general result by Lusztig, \cite[\S 7.10]{CanonicalBasis}, that an irreducible, highest weight module equipped with this involution admits a canonical basis. Set
$$
\mathcal{G}^+(\mu) = \sum_{\lambda} d_{\lambda,\mu}(q) |\lambda \rangle, \qquad \mathcal{G}^-(\lambda) = \sum_\mu e_{\lambda,\mu}(q) | \mu \rangle.
$$
The polynomials $d_{\lambda,\mu}$ and $e_{\lambda,\mu}$ have the following properties: 
\begin{itemize}
\item They are non-zero only if $\lambda$ and $\mu$ have the same $r$-core, $d_{\lambda,\lambda}(q) = e_{\lambda,\lambda}(q) = 1$; 
\item $d_{\lambda,\mu}(q) = 0$ unless $\lambda \le \mu$, and $e_{\lambda,\mu}(q) = 0$ unless $\mu \le \lambda$. 
\end{itemize}

\subsection{GAP}

In order to calculate the polynomial $\mathcal{G}^+(\mu)$ we use the computer package GAP. The file\footnote{Available from \url{http://www.maths.gla.ac.uk/~gbellamy/MSRI.html}.} \verb+Canonical.gap+ contains the functions 
\begin{verbatim}
APolynomial(lambda,mu,r)
\end{verbatim}
and
\begin{verbatim}
DPolynomial(lambda,mu,r)
\end{verbatim}
which, when given a pair of partitions and an integer $r$, returns the polynomials $a_{\lambda,\mu}(q)$ and $d_{\lambda,\mu}(q)$ respectively. For example:

\begin{verbatim}
gap>Read("Canonical.gap");
gap>APolynomial([2,1,1],[3,1],2);
q^2-1
gap>DPolynomial([1,1,1,1,1],[5],2);
q^2
gap>
\end{verbatim}

\begin{exercise}\label{ex:r2}
For $r = 2$, describe $\mc{G}(\lambda)$ for all $\lambda \vdash 4$ and $\lambda \vdash 5$.
\end{exercise}

\begin{exercise}
(Harder) Write a program that calculates the polynomials $e_{\lambda,\mu}(q)$. 
\end{exercise} 

\subsection{} 

Then, assuming that $r > 1$, \cite[Theorem 11]{VaragnoloVasserotDecomposition} says that 
$$
[W_\lambda : L_\mu] = d_{\lambda',\mu'}(1), \qquad [L_\lambda : W_\mu] = e_{\lambda,\mu}(1).
$$

Combining the results of \cite{LT}, \cite{VaragnoloVasserotDecomposition} and \cite{RouquierQSchur}:

\begin{thm}[Leclerc-Thibon, Vasserot-Varagnolo, Rouquier]\label{thm:combineall}
We have  
\beq{eq:decompmatrix}
[\Delta(\lambda) : L(\mu)] = d_{\lambda',\mu'}(1), \quad \textrm{ and } \quad [L(\lambda) : \Delta(\mu)] = e_{\lambda,\mu}(1).
\eeq
\end{thm}

\subsection{The character of $L(\lambda)$} After our grand tour of combinatorial representation theory, we return to our original problem of calculating the character of the simple $\H_{\mbf{c}}(\s_n)$-modules $L(\lambda)$. Theorem \ref{thm:combineall} gives us a way to do express this character in terms of the numbers $e_{\lambda,\mu}(1)$. For the symmetric group, the Euler element in $\H_{\mbf{c}}(\s_n)$ is\footnote{The Euler element defined here differs from the one in section \ref{sec:euler} by a constant.}  
$$
\eu = \frac{1}{2} \sum_{i = 1}^n x_i y_i + y_i x_i = \sum_{i = 1}^n x_i y_i + \frac{n}{2} - \frac{\mbf{c}}{2} \sum_{1 \le i \neq j \le n} s_{i,j}.
$$
Recall that $\eu$ acts on the space $1 \o \lambda \subset \Delta(\lambda)$ by a scalar, denoted $\mbf{c}_{\lambda}$. Associated to the partition $\lambda = (\lambda_1, \ds, \lambda_k)$ is the \textit{partition statistic}, which is defined to be
$$
n(\lambda) := \sum_{i = 1}^k (i - 1)\lambda_i.
$$ 

\begin{lem}\label{lem:kappacalculation}
For each $\lambda \vdash n$ and $\mbf{c} \in \C$,
\beq{eq:kappavalue}
\mbf{c}_{\lambda} = \frac{n}{2} + \mbf{c} ( n(\lambda) - n(\lambda')).
\eeq
\end{lem}

\begin{proof}
The \textit{Jucys-Murphy} elements in $\C [ \s_n]$ are defined to be $\Theta_i = \sum_{j < i} s_{ij}$, for all $i = 2, \ds, n$ so that
$$
\eu = \sum_{i = 1}^n x_i y_i + \frac{n}{2} - \mbf{c} \sum_{i = 2}^n \Theta_i.
$$
Let $\sigma$ be a standard tableau of shape $\lambda$ and $v_\sigma$ the corresponding vector in $\chi_\lambda$. Then 
$$
\Theta_i \cdot v_\sigma = \mathrm{ct}_\sigma(i) v_\sigma,
$$
where $c_\sigma(i)$ is the column of $\lambda$ containing $i$, $r_\sigma(i)$ is the row of $\lambda$ containing $i$ and $\mathrm{ct}_\sigma(i) := c_{\sigma}(i) - r_\sigma(i)$ is the content of the node containing $i$. Note that $\mathrm{ct}_\sigma(1)= 0$ for all standard tableaux $\sigma$. Therefore 
$$
\eu \cdot v_\sigma = \left(  \frac{n}{2} - \mbf{c} \sum_{i = 2}^n \mathrm{ct}_\sigma(i) \right) v_\sigma,
$$
and hence
$$
\mbf{c}_{\lambda} = \frac{n}{2} - \mbf{c} \sum_{i = 2}^n \mathrm{ct}_\sigma(i) = \frac{n}{2} - \mbf{c} \sum_{i = 1}^n \mathrm{ct}_\sigma(i).
$$
Now $\sum_{i = 1}^n r_\sigma(i) = \sum_{j = 1}^{\ell(\lambda)} (j - 1) \lambda_j = n(\lambda)$ and similarly $\sum_{i = 1}^n c_\sigma(i) = n(\lambda')$. This implies equation (\ref{eq:kappavalue}).
\end{proof}

\begin{prop}
We have 
\begin{equation}\label{eq:char}
\ch(L(\lambda)) = \frac{1}{(1 - t)^n} \cdot \left( \sum_{\mu \le \lambda} e_{\lambda,\mu}(1) \dim (\mu) t^{\mbf{c}_{\mu}} \right).
\end{equation}
\end{prop}

\begin{proof}
The corollary depends on two key facts about standard modules. Firstly, they form a $\Z$-basis of the Grothendieck ring $K_0(\mc{O})$ (exercise \ref{ex:Grothbasis}) and, secondly, it is easy to calculate the character of $\Delta(\lambda)$. Theorem \ref{thm:combineall} implies that we have 
$$
[L(\lambda)] = \sum_{\mu \le \lambda} e_{\lambda,\mu}(1) [\Delta(\mu)]
$$
in $K_0(\mc{O})$. Now the corollary follows from the fact that $\ch(\Delta(\mu)) = \frac{\dim (\mu) t^{\mbf{c}_{\mu}}}{(1 - t)^n}$.
\end{proof}

Notice that, though equation (\ref{eq:char}) looks very simple, it is extremely difficult to extract meaningful information from it. For instance, one cannot tell whether $L(\lambda)$ is finite dimensional by looking at this character formula. Also, note that the coefficients $e_{\lambda,\mu}(1)$ are often negative so the numerator is a polynomial with some negative coefficients. But expanding the fraction as a power-series around zero gives a series with \textit{only positive} integer coefficients - all negative coefficients magically disappear! 

\begin{example}
Let $n = 4$ and $\mbf{c} = \frac{5}{3}$. In this case, the numbers $e_{\lambda,\mu}(1)$ and $\mbf{c}_{\lambda}$ are: 
\begin{displaymath}
\begin{array}{c|ccccc}
\mu \backslash \lambda & (4) & (3,1) & (2,2) & (2,1,1) & (1,1,1,1)\\
\hline
{(4)} & 1 & 0 & 0 & 0 & 0 \\
{(3,1)} & 0 & 1 & 0 & 0 & 0 \\
{(2,2)} & -1 & 0 & 1 & 0 & 0 \\ 
{(2,1,1)} & 0 & 0 & 0 & 1 & 0 \\  
{(1,1,1,1)} & 1 & 0 & -1 & 0 & 1 \\  
\mbf{c}_{\lambda} & -8 & \frac{-4}{3} & 2 & \frac{16}{3} & 12 
\end{array}
\end{displaymath} 
If we define $\ch_{\lambda}(t) := (1 - t)^4 \cdot \ch(L(\lambda))$, then 
$$
\ch_{(4)}(t) = t^{-8} - 2t^{2} + t^{12}, \quad  \ch_{(3,1)}(t) = 3t^{-\frac{4}{3}}, 
$$
$$
\ch_{(2,2)}(t) = 2t^{2} - t^{12}, \quad \ch_{(2,1,1)}(t) = 3t^{\frac{16}{3}}, \quad \ch_{(1,1,1,1)}(t) = t^{12}.
$$
When $\mbf{c} = \frac{9}{4}$ we get:
\begin{displaymath}
\begin{array}{c|ccccc}
\mu \backslash \lambda & (4) & (3,1) & (2,2) & (2,1,1) & (1,1,1,1) \\
\hline
(4)          & 1  & 0  & 0  & 0 & 0 \\
{(3,1)}      & -1 & 1  & 0  & 0 & 0 \\
{(2,2)}      & 0  & 0  & 1  & 0 & 0 \\ 
{(2,1,1)}    & 1  & -1 & 0  & 1 & 0 \\  
{(1,1,1,1)}  & -1 & 1  & 0 & -1 & 1 \\  
\mbf{c}_{\lambda} & \frac{-23}{2} & \frac{-5}{2} & 2 & \frac{13}{2} & \frac{31}{2}    
\end{array}
\end{displaymath}
and
$$
\ch_{(4)}(t) = t^{-\frac{23}{2}} - 3t^{-\frac{5}{2}} + 3t^{\frac{13}{2}} - t^{\frac{31}{2}}, \quad \ch_{(3,1)}(t) = 3t^{-\frac{5}{2}} - 3t^{\frac{13}{2}} + t^{\frac{31}{2}}, 
$$
$$
\ch_{(2,2)}(t) = 2t^{2}, \quad \ch_{(2,1,1)}(t) = 3t^{\frac{13}{2}} - t^{\frac{31}{2}}, \quad \ch_{(1,1,1,1)}(t) = t^{\frac{31}{2}}.
$$
If $\s_4$ acts on $\C^4$ by permuting the basis elements $\epsilon_1, \ds, \epsilon_4$, then we let $\mathfrak{h} \subset \C^4$ be the three dimensional reflection representation, with basis $\{ \epsilon_1- \epsilon_2, \epsilon_2 - \epsilon_3, \epsilon_3 - \epsilon_4 \}$. One could also consider representations of $\H_{\mbf{c}}(\mathfrak{h},\s_4)$ instead of $\H_{\mbf{c}}(\C^4,\s_4)$. Since $\C^4 = \mathfrak{h} \oplus \C$ as a $\s_4$-module, the defining relations for $\H_{\mbf{c}}(\C^4,\s_4)$ make it clear that we have a decomposition $\H_{\mbf{c}}(\C^4,\s_4) = \H_{\mbf{c}}(\mathfrak{h},\s_4) \o \mc{D}(\C)$, which implies that $\Delta_{\C^4}(\lambda) = \Delta_{\mathfrak{h}}(\lambda) \o \C[x]$ for each $\lambda \vdash 4$. Then,
$$
\ch(L_{\C^4}(\lambda)) = \frac{1}{1 - t} \cdot \ch(L_{\mathfrak{h}}(\lambda)),
$$
and hence $\ch_{\lambda}(t) = (1-t)^3 \cdot \ch(L_{\mathfrak{h}}(\lambda))$. If we consider $\lambda = (4)$ with $\mbf{c} = \frac{9}{4}$ then notice that 
$$
\frac{t^{-\frac{23}{2}} - 3t^{-\frac{5}{2}} + 3t^{\frac{13}{2}} - t^{\frac{31}{2}}}{(1 - t)^3} = t^{-\frac{23}{2}} \left[ \frac{1 - 3t^{9} + 3t^{18} - t^{27}}{(1 - t)^3} \right] = t^{-\frac{23}{2}}(t^{24} + 3t^{23} + 6t^{22} + \cdots + 3t + 1).
$$
This implies that $L((4))$ is finite dimensional. Evaluating the above polynomial shows that it actually has dimension $729$.  
\end{example}

\begin{exercise}\label{ex:chars}
\begin{enumerate}
\item For $c = \frac{7}{2}$, compute the character of all the simple modules of category $\mc{O}$ for $\H_{\mbf{c}}(\s_5)$. 

\item For $c = \frac{11}{3}$, compute the character of all the simple modules of category $\mc{O}$ for $\H_{\mbf{c}}(\s_4)$. 

\item What are the blocks of $\mc{O}$ for $\s_5$ at $c = \frac{7}{2}$? at $c = \frac{103}{3}$ or at $c = \frac{29}{5}$? Hint: It is known that two partitions are in the same block of the $q$-Schur algebra if and only if they have the same $r$-core. 
\end{enumerate}
\end{exercise}

Since the action of $\eu$ on a module $M \in \mc{O}$ commutes with the action of $W$, the multiplicity space of a representation $\lambda \in \Irr (W)$ is a $\eu$-module. Therefore, one can refine the character $\ch$ to 
$$
\ch_W (M) = \sum_{\lambda} \ch(M(\lambda)) [\lambda],
$$
where $M = \bigoplus_{\lambda \in \Irr (W)} M(\lambda) \o \lambda$ as a $W$-module. Given $\lambda,\mu \in \Irr (W)$, we define the \textit{generalized fake polynomial} $f_{\lambda,\mu}(t)$ by 
$$
f_{\lambda,\mu} (t) = \sum_{i \in \Z} [\C[\h]^{co W}_i \o \lambda : \mu] t^i,
$$
a Laurent polynomial. 

\begin{exercise}\label{ex:s4Hard}
What is the $\s_4$-graded character of the simple modules, expressed in terms of generalized fake polynomials, in $\mc{O}$ when $c = \frac{5}{2}$?
\end{exercise}

\subsection{Yvonne's conjecture}\label{sec:Yvonne}

We end the lecture by explaining how to (conjecturally) generalize the picture for $\s_n$ to other complex reflection groups (this section is really for those who have a firm grasp of the theory of rational Cherednik algebras, and can be safely ignored if so desired). In many ways, the most interesting, and hence most intensely studied, class of rational Cherednik algebras are those associated to the complex reflection group  
$$
W = \s_n \wr \Z_l = \s_n \ltimes \Z_l^n,
$$
the wreath product of the symmetric group with the cyclic group of order $l$. Fix $\zeta$ a primitive $l$-th root of unity. The reflection representation for $\s_n \wr \Z_l$ is $\C^n$. If $y_1, \ds, y_n$ is the standard basis of $\C^n$ then 
$$
\Z_l^n = \{ g_i^j \ | \ 1 \le i \le n, \ 0 \le j \le l-1 \}
$$
and 
$$
g_i^j \cdot y_k = \left\{ \begin{array}{ll}
\zeta^j y_k & \textrm{if $k = i$}\\
y_k & \textrm{otherwise} 
\end{array} \right.
$$
The symmetric group acts as $\sigma \cdot y_i = y_{\sigma(i)}$. This implies that 
$$
\sigma \cdot g_i = g_{\sigma(i)} \cdot \sigma. 
$$
The conjugacy classes of reflections in $W$ are 
$$
R = \{ s_{i,j} g_i g_j^{-1}  \ | \ 1 \le i < j \le n \}, \quad S_i = \{ g_j^i \ | \ 1 \le j \le n \}, \ 1 \le i \le l-1.
$$
and we define $\mbf{c}$ by $\mbf{c}(s_{i,j}) = c$, $\mbf{c}(g_j^i) = c_i$ so that the rational Cherednik algebra is parameterized by $c,c_1, \ds, c_{l-1}$. To relate category $\mc{O}$ to a certain Fock space, we introduce new parameters $\mbf{h} = (h,h_0, \ds, h_{l-1})$, where 
$$
c = 2 h, \quad c_i = \sum_{j = 0}^{l-1} \zeta^{-ij} (h_j - h_{j+1}), \quad \forall \ 1 \le i \le l-1.
$$
Note that the above equations do not uniquely specify $h_0, \ds, h_{l-1}$; one can choose $h_0 + \cdots h_{l-1}$ freely. Finally, we define $\mbf{s} = (s_0, \ds, s_{l-1}) \in \Z^l$ and $e \in \N$ by $h = \frac{1}{e}$ and 
$$
h_j = \frac{s_j}{e} - \frac{j}{d}.
$$

The irreducible representations of $\s_n \wr \Z_l$ are labeled by \textit{$l$-multipartitions} of $n$, where $\boldsymbol{\lambda} = (\lambda^{(1)}, \ds, \lambda^{(l)})$ is an $l$-tuple of partitions $\lambda^{(i)}$ such that 
$$
\sum_{i = 1}^l |\lambda^{(i)}| = n.
$$
The set of all $l$-multipartitions is denoted $\mc{P}^l$ and the subset of all $l$-multipartitions of $n$ is denoted $\mc{P}^l_n$. Thus, the simple modules in category $\mc{O}$ are $L(\boldsymbol{\lambda})$, for $\boldsymbol{\lambda} \in \mc{P}^l_n$. It is actually quite easy to construct the simple $W$-modules $\boldsymbol{\lambda}$ once one has constructed all representations of the symmetric group.\\

Uglov define for each $l \ge 1$ and $\mbf{s} \in \Z^l$ a level $l$ Fock space $\mc{F}_q^l[\mbf{s}]$ with multi-charge $\mbf{s}$. This is again a representation of the quantum affine algebra $U_{q}(\widehat{\mf{sl}}_r)$, this time with $\Q(q)$-basis $| \boldsymbol{\lambda} \rangle$ given by $l$-multi-partitions. The action of the operators $F_i,E_i$ and $K_i$ have a similar combinatorial flavour as for $l = 1$. The space $\mc{F}_q^l[\mbf{s}]$ is also equipped with a $\Q$-linear involution. 

\begin{thm}\label{thm:Uglov}
There exists a unique $\Q(q)$-basis $\mc{G}(\boldsymbol{\lambda})$ of $\mc{F}_q^l[\mbf{s}]$ such that
\begin{enumerate}
\item $\overline{\mc{G}(\boldsymbol{\lambda})} = \mc{G}(\boldsymbol{\lambda})$.
\item $\mc{G}(\boldsymbol{\lambda}) - \boldsymbol{\lambda} \in \bigoplus_{\boldsymbol{\mu} \in \mc{P}^l_n} q \Q[q] | \boldsymbol{\mu} \rangle$ if $\boldsymbol{\lambda} \vdash n$. 
\end{enumerate}
\end{thm}

Therefore, for $\boldsymbol{\lambda}, \boldsymbol{\mu} \in \mc{P}_n^l$, we define $d_{\boldsymbol{\lambda}, \boldsymbol{\mu}}(q)$ by 
$$
\mc{G}(\boldsymbol{\mu}) = \sum_{\boldsymbol{\lambda} \in \mc{P}^l_n} d_{\boldsymbol{\lambda}, \boldsymbol{\mu}}(q) | \boldsymbol{\lambda} \rangle.
$$
The following conjecture, originally due to Yvonne, but in the generality stated here is due to Rouquier, relates the multiplicities of simple modules inside standard modules for $\H_{\mbf{c}}(\s_n \wr \Z_l)$ to the polynomials $d_{\boldsymbol{\lambda}, \boldsymbol{\mu}}(q)$. 

\begin{conjecture}
For all $\boldsymbol{\lambda}, \boldsymbol{\mu} \in \mc{P}_n^l$ we have 
$$
[\Delta(\boldsymbol{\lambda}) : L(\boldsymbol{\mu})] = d_{\boldsymbol{\lambda}^{\dagger}, \boldsymbol{\mu}^{\dagger}}(1). 
$$
\end{conjecture}

If $\boldsymbol{\lambda} = (\lambda^{(1)}, \ds, \lambda^{(m)})$ is an $m$-multi-partition then 
$$
\boldsymbol{\lambda}^{\dagger} = ((\lambda^{(m)})', \ds, (\lambda^{(1)})'),
$$
where $(\lambda^{(i)})'$ is the usual transpose of the partition $\lambda^{(i)}$.   

\subsection{Additional remark}

\begin{itemize}
\item Theorem \ref{thm:rouquierequiv} is given in \cite[Theorem 6.11]{RouquierQSchur}. Its proof is based on the uniqueness of quasi-hereditary covers of the Hecke algebra, as shown in \cite{RouquierQSchur}. 

\item Theorem \ref{thm:Uglov}, due to Uglov, is Theorem 2.5 of \cite{Uglov}. 

\item Yvonne's conjecture was originally stated in \cite[Conjecture 2.13]{Yvonne}. The refined version, given in section \ref{sec:Yvonne}, is due to Rouquier, \cite[Section 6.5]{RouquierQSchur}
\end{itemize}

\newpage

\section{The $\KZ$ functor}\label{sec:KZfunctor}

Beilinson and Bernstein, in proving their incredible localization theorem for semi-simple Lie algebras and hence confirming the Kazhdan-Luzstig conjecture, have shown that it is often very fruitfully to try and understand the representation theory of certain algebras by translating representation theoretic problems into topological problems via $\dd$-modules and the Riemann-Hilbert correspondence. This is the motivation behind the Knizhnik-Zamolodchikov ($\KZ$) functor. 

In the case of rational Cherednik algebras, the obvious way to try and relate them to $\dd$-module is to use the Dunkl embedding. Recall from the first lecture that this is the embedding $\H_{\mbf{c}}(W) \hookrightarrow \dd(\h_{\reg}) \rtimes W$, which becomes an isomorphism 
$$
\H_{\mbf{c}}(W)[\delta^{-1}] \stackrel{\sim}{\longrightarrow} \dd(\h_{\reg}) \rtimes W
$$ 
after localization. Given a module $M$ in category $\mc{O}$, it's localization $M[\delta^{-1}]$ becomes a $\dd(\h_{\reg}) \rtimes W$-module. We can understand $\dd(\h_{\reg}) \rtimes W$-modules as $\dd$-modules on $\h_{\reg}$ that are $W$-equivariant. Since $W$ acts freely on $\h_{\reg}$, this is the same as considering a $\dd$-module on $\h_{\reg} / W$. Finally, using the powerful Riemann-Hilbert correspondence we can thing of such a $\dd$-module as a local system on $\h_{\reg} / W$. That is, it become a representation of the fundamental group $\pi_1(\h_{\reg} / W)$ of $\h_{\reg} / W$. This is a purely topological object. The process of passing from a module in category $\mc{O}$ to a representation of $\pi_1(\h_{\reg} / W)$ is exactly what the $\KZ$-functor does. The goal of this lecture is to try and describe in some geometric way the image of category $\mc{O}$ under this procedure. Remarkably, what one gets in the end is a functor 
$$
\KZ : \mc{O} \longrightarrow \Lmod{\mc{H}_{\mbf{q}}(W)},
$$
from category $\mc{O}$ to the category of finitely generated modules over the \textit{Hecke algebra} $\mc{H}_{\mbf{q}}(W)$ associated to $W$.  

In the case of semi-simple Lie algebras, Beilinson and Bernstein localization result gave us a proof of the Kazhdan-Lusztig conjecture. What does the $\KZ$-functor give us for rational Cherednik algebras? It allows us to play off the representation theory of the rational Cherednik algebra against the representation theory of the Hecke algebra. Via the $\KZ$-functor, the Hecke algebra becomes crucial in proving deep results about the existence of finite-dimensional representations of $\H_{\mbf{c}}(W)$, \cite{Finitedimreps} and \cite{DiagonalInvariants} (which in turn confirmed a conjecture of Haiman's in algebraic combinatorics), Rouquier's equivalence Theorem \ref{thm:rouquierequiv}, and on the properties of restriction and induction functors for rational Cherednik algebras \cite{Shan}. Conversely, the rational Cherednik algebra has been used to prove highly non-trivial results about Hecke algebras e.g. section 6 of \cite{GGOR} and the article \cite{GordonGriffChlo}.  

\subsection{Integrable connections}

A $\dd(\h_{\reg})$-module which is finitely generated as a $\C[\h_{\reg}]$-module is called an \textit{integrable connection}. It is a classical result that every integrable connection is actually free as a $\C[\h_{\reg}]$-module, see \cite[Theorem 1.4.10]{HTT}. Therefore, if $N$ is an integrable connection, 
$$
N \simeq \bigoplus_{i = 1}^k \C[\h_{\reg}] u_i
$$
for some $u_i \in N$. If we fix coordinates $x_1, \ds, x_n$ such that $\C[\h] = \C[x_1, \ds, x_n]$ and let $\pa_i = \frac{\pa}{\pa x_i}$, then the action of $\dd(\h_{\reg})$ on $N$ is then completely encoded in the equations
\begin{equation}\label{eq:connection}
\pa_l \cdot u_i = \sum_{j = 1}^k f_{l,i}^j u_j,
\end{equation}
for some polynomials\footnote{The condition $[\pa_l,\pa_m] = 0$ needs to be satisfied, which implies that one cannot choose arbitrary polynomials $f_{i,i}^j$.} $f_{l,i}^k \in \C[\h_{\reg}]$. The integer $k$ is called the \textit{rank} of $N$. 

A natural approach to studying integrable connections is to look at their space of solutions. Since very few differential equations have polynomial solutions, this approach only makes sense in the analytic topology. So we'll write $\h_{\reg}^{an}$ for the same space, but now equipped with the analytic topology and $\C[\h^{an}_{\reg}]$  denotes the ring of \textit{holomorphic} functions on $\h_{\reg}^{an}$. Since $\dd(\h_{\reg}^{an}) = \C[\h^{an}_{\reg}] \o_{\C[\h_{\reg}]} \dd(\h_{\reg})$, we have a natural functor 
$$
\Lmod{\dd(\h_{\reg}) \rtimes W} \rightarrow \Lmod{\dd(\h_{\reg}^{an}) \rtimes W}, \quad M \mapsto M^{an} := \C[\h^{an}_{\reg}] \o_{\C[\h_{\reg}]} M.
$$
Since $\C[\h^{an}_{\reg}]$ is faithfully flat, this is exact and conservative i.e. $M^{an} = 0$ implies that $M = 0$. On any simply connected open subset $U$ of $\h_{\reg}^{an}$, the vector space 
$$
\Hom_{\dd(\h_{\reg}^{an})} (N^{an}, \C[U])
$$
is $k$-dimensional because it is the space of solutions of a $k \times k$ matrix of first order linear differential equations. These spaces of local solutions glue together to give a rank $k$ \textit{local system} $\Sol(N)$ on $\h_{\reg}$ i.e. a locally constant sheaf of $\C$-vector spaces such that each fiber has dimension $k$, see figure \ref{fig:bluecircles} for an illustration. The reader should have in mind the idea of analytic continuation of a locally defined holomorphic to a multivalued function from complex analysis, of which the notion of local system is a generalization.

\begin{exercise}
Let $\h_{\reg}= \Cs$. For each $\alpha \in \C$, write $M_{\alpha}$ for the $\dd(\Cs)$-module $\dd(\Cs) / \dd(\Cs) (x \pa - \alpha)$. Show that $M_{\alpha} \simeq M_{\alpha + 1}$ for all $\alpha$. Describe a multivalued holomorphic function which is a section of the one-dimensional local system $\Sol(M_{\alpha})$ (notice that in general, the function is genuinely multivalued and hence not well-defined on the whole of $\Cs$. This corresponds to the fact that the local system $\Sol(M_{\alpha})$ has \textit{no} global sections in general). Finally, the local system $\Sol(M_{\alpha})$ defines a one-dimensional representation of the fundamental group $\pi_1(\Cs) = \Z$. What is this representation?
\end{exercise}  

\subsection{Regular singularties} Assume now that $\dim \h = 1$, so that $\h = \C$, $\h_{\reg}= \Cs$  and $\C[\h_{\reg}] = \C[x,x^{-1}]$. Then, we say that $N$ is a \textit{regular connection} (or has \textit{regular singularties}) if, with respect to some $\C[\h_{\reg}^{an}]$-basis of $N^{an}$, equation (\ref{eq:connection}) becomes 
\begin{equation}\label{eq:connection2}
\pa \cdot u_i = \sum_{j = 1}^k \frac{a_{i,j}}{x} u_j, \quad a_{i,j} \in \C.
\end{equation}
When $\dim \h > 1$, we say that $N$ has regular singularities if the restriction $N |_C$ of $N$ to any smooth curve $C \subset \h_{\reg}$ has, after a suitable change of basis, the form (\ref{eq:connection2}). The category of integrable connections with regular singularities on $\h_{\reg}$ is denoted $\mathrm{Conn}^{\reg}(\h_{\reg})$. Similarly, let's write $\Loc (\h_{\reg})$ for the category of finite-dimensional local systems on $\h_{\reg}^{an}$. There are two natural functors $\mathrm{Conn}^{\reg}(\h_{\reg})$ to local systems on $\h_{\reg}^{an}$. Firstly, there is a \textit{solutions functor}
$$
N \mapsto \Sol(N) := \mc{H}om_{\dd_{\h_{\reg}^{an}}}(N^{an},\mc{O}_{\h_{\reg}}^{an}), 
$$
and secondly the \textit{deRham functor}
$$
N \mapsto \mathsf{DR}(N) := \mc{H}om_{\dd_{\h_{\reg}^{an}}}(\mc{O}_{\h_{\reg}}^{an},N^{an}).
$$
There are natural duality (i.e. they are \textit{contravariant} equivalences whose square is the identity) functors $\mathbb{D} : \mathrm{Conn}^{\reg}(\h_{\reg}) \stackrel{\sim}{\longrightarrow} \mathrm{Conn}^{\reg}(\h_{\reg})$ and $\mathbb{D}' : \Loc (\h_{\reg}) \stackrel{\sim}{\longrightarrow} \Loc (\h_{\reg})$, on integrable connection and local systems respectively. These duality functors intertwine the solutions and de-Rham functors:
\begin{equation}\label{eq:dual}
\mathbb{D}' \circ \mathsf{DR} = \mathsf{DR} \circ \mathbb{D}= \Sol, \quad \textrm{and hence} \quad \mathbb{D}' \circ \Sol = \Sol \circ \mathbb{D} = \mathsf{DR};
\end{equation}
 see \cite[Chapter 7]{HTT} for details. Sticking with conventions, we'll work with the de-Rham functor. Notics that there is a natural identification 
$$
\mathsf{DR}(N) = \{ n \in N^{an} \ | \ \pa_i \cdot n = 0 \ \forall \ i \} =: N^{\nabla},
$$
where $N^{\nabla}$ is usually referred to as the subsheaf of \textit{horizontal sections} of $N^{an}$. The fundemental group $\pi_1(\h_{\reg})$ of $\h_{\reg}$ is the group of loops in $\h_{\reg}$ from a fixed point $x_0$, up to homotopy. Following a local system $\mathsf{L}$ long a given loop $\gamma$ based at $x_0$ defines an invertible endomorphism $\gamma_* : \mathsf{L}_{x_0} \rightarrow \mathsf{L}_{x_0}$ of the stalk of $\mathsf{L}$ at $x_0$. The map $\gamma_*$ only depends on the homotopy class of $\gamma$ i.e. only on the object defined by $\gamma$ in $\pi_1(\h_{\reg})$. In this way, the stalk $\mathsf{L}_{x_0}$ becomes a representation of $\pi_1(\h_{\reg})$. The following is a well-know result  in algebraic topology: 

\begin{prop}\label{prop:fundementalequiv}
Assume that $\h_{\reg}$ is connected. The functor $\Loc(\h_{\reg}) \rightarrow \Lmod{\pi_1(\h_{\reg})}$, $\mathsf{L} \mapsto \mathsf{L}_{x_0}$ is an equivalence. 
\end{prop}

See I, Corollaire 1.4 of \cite{DelLMS} for a proof of this equivalence. 

\begin{figure}
\begin{tikzpicture}
\draw [->] (-3,-2) -- (-3,2);
\draw [->] (-5,0) -- (-1,0);
\draw [<-,thick,gray] (-2,0) to [out=-90,in=0] (-3,-1) to [out=180,in=-90] (-4,0) to [out=90,in=180] (-3,1) to [out=0,in=90] (-2,0);
\draw [fill=white] (-3,0) circle [radius=0.1]; 
\node at (-4,1) {$T$};

\draw [->] (3,-2) -- (3,2);
\draw [->] (1,0) -- (5,0);
\draw[fill,blue] (4,1) circle [radius=0.025];
\node [right] at (4,1) {$p_1$};
\draw[fill,blue] (2.6,-0.5) circle [radius=0.025];
\node [left] at (2.6,-0.5) {$p_2$};
\draw[fill,blue] (2,0.65) circle [radius=0.025];
\node [right] at (2,0.65) {$p_3$};
\draw[fill] (4.5,1.7) circle [radius=0.05];
\draw [<-,thick,gray,rounded corners] (4.5,1.7) .. controls (4,1.9) .. (3.3,1.6) .. controls (3.3,-0.5) .. (2.6,-2) .. controls (1.2,-0.5) .. (2.6,0.7) .. controls (3,-0.5) .. (2.6,-1) .. controls (2,-0.5) .. (2.6,0) ..  controls (2.9,0.7) .. (2.5,0.8) .. controls (2.3,0.9) .. (2,1) .. controls (1.4,0.7) .. (1.3,1) .. controls (1,0.3) .. (1.4,0) .. controls (3,0.3) .. (4.6,-0.2) .. controls (4.8,1) and (4.6,1.9) .. (4.5,1.7);
\node at (5,0.5) {$\gamma$};
\end{tikzpicture}
\caption{The generator $T$ of the braid group $\pi_1(\Cs)$, and non-trivial element $\gamma$ in $\pi_1(\C \backslash \{ p_1,p_2,p_2 \}, \bullet )$.}
\end{figure}

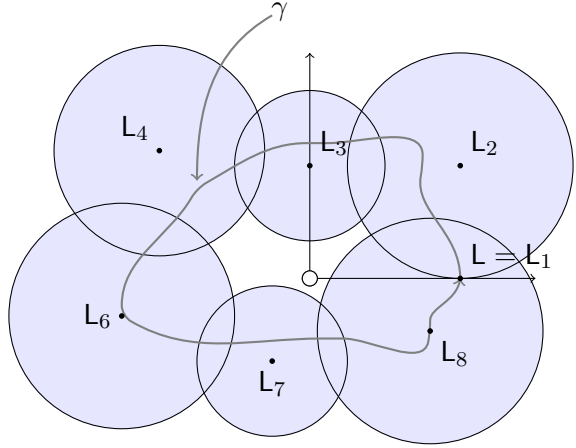
\begin{figure}\label{fig:bluecircles}
\begin{tikzpicture}
\draw [blue!10!white,fill=blue!10!white] (2,1.5) circle [radius=1.5]; 
\draw [blue!10!white,fill=blue!10!white] (0,1.5) circle [radius=1]; 
\draw [blue!10!white,fill=blue!10!white] (-2,1.7) circle [radius=1.4]; 
\draw [blue!10!white,fill=blue!10!white] (-2.5,-0.5) circle [radius=1.5]; 
\draw [blue!10!white,fill=blue!10!white] (-0.5,-1.1) circle [radius=1]; 
\draw [blue!10!white,fill=blue!10!white] (1.6,-0.7) circle [radius=1.5];

\draw [thin] (2,1.5) circle [radius=1.5]; 
\draw [thin] (0,1.5) circle [radius=1]; 
\draw [thin] (-2,1.7) circle [radius=1.4]; 
\draw [thin] (-2.5,-0.5) circle [radius=1.5]; 
\draw [thin] (-0.5,-1.1) circle [radius=1]; 
\draw [thin] (1.6,-0.7) circle [radius=1.5];
\draw [->] (0,0) -- (3,0);
\draw [->] (0,0) -- (0,3);

\draw[fill=white] (0,0) circle [radius=0.1];

\draw [->,thick,gray,rounded corners] (2,0) to [out=90,in=-100] (1.5,1.5) to [out=90,in=0] (0.1,1.8) to [out=180,in=30] (-1.5,1.2) to [out=-120,in=90] (-2.5,-0.5) to [out=-30,in=-180] (0.3,-0.8) to [out=0,in=-100] (1.6,-0.7) to [out=90,in=-100] (2,0);

\node [above right] at (2,0) {$\mathsf{L} = \mathsf{L}_1$};
\draw[fill] (2,0) circle [radius=0.03];
\node [above right] at (2,1.5) {$\mathsf{L}_2$};
\draw[fill] (2,1.5) circle [radius=0.03];
\node [above right] at (0,1.5) {$\mathsf{L}_3$};
\draw[fill] (0,1.5) circle [radius=0.03];
\node [above left] at (-2,1.7) {$\mathsf{L}_4$};
\draw[fill] (-2,1.7) circle [radius=0.03];
\node [left] at (-2.5,-0.5) {$\mathsf{L}_6$};
\draw[fill] (-2.5,-0.5) circle [radius=0.03];
\node [below] at (-0.5,-1.1) {$\mathsf{L}_7$};
\draw[fill] (-0.5,-1.1) circle [radius=0.03];
\node [below right] at (1.6,-0.7) {$\mathsf{L}_8$};
\draw[fill] (1.6,-0.7) circle [radius=0.03];

\draw[fill=white] (0,0) circle [radius=0.1];

\draw [thick,gray,->] (-0.5,3.5) to [out=-150,in=90] (-1.5,1.3);
\node at (-0.4,3.55) {$\gamma$};

\end{tikzpicture}
\caption{The locally constant sheaf (local system) $\mathsf{L}$ on $\Cs$. On each open set, the local system $\mathsf{L}_i$ is constant, and on overlaps we have $\mathsf{L}_i \simeq \mathsf{L}_{i+1}$. Composing these isomorphisms defines an automorphism $\mathsf{L} \stackrel{\sim}{\rightarrow} \mathsf{L}$, which is the action of the path $\gamma$ on the stalk of $\mathsf{L}$ at $1$.}
\end{figure}

\subsection{} Via the equivalence of Proposition \ref{prop:fundementalequiv}, we may (and will) think of $\mathsf{DR}$ and $\Sol$ as functors from $\mathrm{Conn}^{\reg}(\h_{\reg})$ to $\Lmod{\pi_1(\h_{\reg})}$. Deligne's version of the Riemann-Hilbert correspondence, \cite{DelLMS}, says that:

\begin{thm}
The de-Rham and solutions functors 
$$
\mathsf{DR}, \Sol : \mathrm{Conn}^{\reg}(\h_{\reg}) \rightarrow \Lmod{\pi_1(\h_{\reg})}
$$
are equivalences. 
\end{thm}

Why is the notion of regular connection crucial in Deligne's version of the Riemann-Hilbert correspondence? A complete answer to this question is beyond the authors understanding. But one important point is that there are, in general, many non-isomorphic integrable connections that will give rise to the same local system. So what Deligne showed is that, for a given local system $L$, the notion of regular connection gives one a canonical representative in the set of all connections whose solutions are $L$. 

\subsection{Equivariance}   

Recall that the group $W$ acts freely on the open set $\h_{\reg}$. This implies that the quotient map $\pi : \h_{\reg} \rightarrow \h_{\reg} /W$ is a finite covering map. In particular, its differential $d_x \pi : T_x \h_{\reg} \rightarrow T_{\pi(x)} (\h_{\reg} / W)$ is an isomorphism for all $x \in \h_{\reg}$. 

\begin{prop}\label{prop:diffiso}
There is a natural isomorphism $\dd(\h_{\reg})^W \simeq \dd(\h_{\reg} / W)$.
\end{prop}

\begin{proof}
Since $\dd(\h_{\reg})$ acts on $\C[\h_{\reg}]$, the algebra $\dd(\h_{\reg})^W$ acts on $\C[\h_{\reg}]^W = \C[\h_{\reg} / W]$. One can check from the definition of differential operators that this defines a map $\dd(\h_{\reg})^W \rightarrow \dd(\h_{\reg} /W)$. Since $\dd(\h_{\reg})^W$ is a simple ring, this map must be injective. Therefore it suffices to show that it is surjective. 

Let $x_1, \ds, x_n$ be a basis of $\h^*$ so that $\C[\h] = \C[x_1, \ds, x_n]$. By the Chevalley-Shephard-Todd Theorem, Theorem \ref{thm:CSTthm}, $\C[\h]^W$ is also a polynomial ring with homogeneous, algebraically independent generators $u_1, \ds, u_n$ say. Moreover, $\C[\h_{\reg}]^W = \C[u_1, \ds, u_n][\delta^{-r}]$ for some $r > 0$. Thus, 
$$
\dd(\h_{\reg} / W) = \C \left\langle u_1, \ds, u_n, \frac{\pa}{\pa u_1}, \ds, \frac{\pa}{\pa u_n} \right\rangle \left[\delta^{-r} \right]
$$
and to show surjectivity it suffices to show that each $ \frac{\pa}{\pa u_i}$ belongs to $\dd(\h_{\reg})^W$. That is, we need to find some $f_{i,j} \in \C[\h_{\reg}]$ such that $v_i := \sum_{j = 1}^n f_{i,j} \frac{\pa}{\pa x_j} \in \dd(\h_{\reg})^W$ and $v_i(u_k) = \delta_{i,k}$. Let $\Delta(\mbf{u})$ be the $n$ by $n$ matrix with $(i,j)$th entry $\frac{\pa u_i}{\pa x_j}$. Then, since $v_i(u_k) = \sum_{j = 1}^n f_{i,j} \frac{\pa u_k}{\pa x_j}$, we really need to find a matrix $F = (f_{i,j})$ such that $F \cdot \Delta(\mbf{u})$ is the $n \times n$ identity matrix. In other words, $F$ must be the inverse to $\Delta(\mbf{u})$. Clearly, $F$ exists if and only if $\det \Delta(\mbf{u})$ is invertible in $\C[\h_{\reg}]$. But, for each $t \in \h_{\reg}$,  $\det \Delta(\mbf{u}) |_{x = t}$ is just the determinant of the differential $d_t \pi : T_t \h_{\reg} \rightarrow T_{\pi(t)} (\h_{\reg} / W)$. As noted above, the linear map $d_t \pi$ is always an isomorphism. Thus, since $\C[\h_{\reg}]$ is a domain, the Nullstellensatz implies that $\det \Delta(\mbf{u})$ is invertible in $\C[\h_{\reg}]$. 

The final thing to check is that if $F = \Delta(\mbf{u})^{-1}$ then each $v_i$ is $W$-invariant. Let $w \in W$ and consider the derivation $v_i' := w(v_i) - v_i$. It is clear from the definition of $v_i$ that $v_i'$ acts as zero on all $u_j$ and hence on the whole of $\C[\h_{\reg}]^W$. Geometrically $v_i'$ is a vector field on $\h_{\reg}$ and defines a linear functional on $\mf{m} / \mf{m}^2$, for each maximal ideal $\mf{m}$ of $\C[\h_{\reg}]$. If $v_i' \neq 0$ then there is some point $x \in \h_{\reg}$, with corresponding maximal ideal $\mf{m}$, such that $v_i'$ is non-zero on $\mf{m} / \mf{m}^2$. Let $\mf{n} = \mf{m} \cap \C[\h_{\reg}]^W$ be the maximal ideal defining $\pi(x)$. Then $d_{x} \pi (v_i')$ is the linear functional on $\mf{n} / \mf{n}^2$ given by the composite 
$$
\mf{n} / \mf{n}^2 \rightarrow \mf{m} / \mf{m}^2 \stackrel{v_i'}{\longrightarrow} \C.
$$
Since $d_x \pi$ is an isomorphism, this functional is non-zero. But this contradicts the fact that $v_i'$ acts trivially on $\C[\h_{\reg}]^W$. Thus, $w(v_i) = v_i$.
\end{proof}

\begin{remark}
The polynomial $\det \Delta(\mbf{u})$ plays an important role in the theory of complex reflection groups, and is closely related to our  $\delta$; see \cite{SteinbergInvariants}. 
\end{remark}

A corollary of the above proposition is:

\begin{cor}\label{cor:equivWdecsend}
The functor $M \mapsto M^{W}$ defines an equivalence between the category of $W$-equivariant $\dd$-modules on $\h^{\reg}$ (i.e. the category of $\dd(\h_{\reg}) \rtimes W$-modules) and the category of $\dd(\h_{\reg}/W)$-modules. 
\end{cor}

\begin{proof}
Let $\mbf{e}$ be the trivial idempotent in $\C W$. We can identify $\dd(\h_{\reg})^W$ with $\mbf{e}(\dd(\h_{\reg}) \rtimes W) \mbf{e}$. Hence by Proposition \ref{prop:diffiso},  $\dd(\h_{\reg} / W) \simeq \mbf{e}(\dd(\h_{\reg}) \rtimes W) \mbf{e}$. Then the proof of the corollary is identical to the proof of Corollary \ref{cor:Morita}. We just need to show that if $M^W = \mbf{e} M$ is zero then $M$ is zero. But, since $W$ acts freely on $\h_{\reg}$, it is already clear for $\C[\h_{\reg}] \rtimes W$-modules that   $M^W = 0$ implies $M =0$ (see for instance the proof of Theorem \ref{thm:genericisregularrep}).  
\end{proof}

One often says that the $\dd(\h_{\reg}) \rtimes W$ $M$ \textit{descends} to the $\dd$-module $M^W$ on $\h_{\reg} / W$ (the terminology coming from descent theory in algebraic geometry). 

We are finally in a position to define the $\KZ$-functor. The the functor is a composition of four(!) functors, so we'll go through it one step at a time. Recall that we're starting with a module $M$ in category $\mc{O}$ and we eventually want a representation of the fundamental group $\pi_1(\h_{\reg} / W)$. First off, we localize and use the Dunkl embedding: $M \mapsto M[\delta^{-1}]$. This gives us a $\dd(\h_{\reg})\rtimes W$-module $M[\delta^{-1}]$. By Corollary \ref{cor:equivWdecsend}, $(M[\delta^{-1}])^W$ is a $\dd$-module on $\h_{\reg} / W$. Before going further we need to know what sort of $\dd$-module we've ended up with. It's shown in \cite{GGOR} that in fact $(M[\delta^{-1}])^W$ is an integrable connection on $\h_{\reg} / W$ \textit{with regular singularities}. Thus, we can apply Deligne's equivalence. The deRham functor $\mathsf{DR}(M[\delta^{-1}]^W)$ applied to $(M[\delta^{-1}])^W$ gives us a representation of $\pi_1(\h_{\reg} / W)$. Thus, we may define the $\KZ$-functor $\KZ : \mc{O} \rightarrow \Lmod{\pi_1(\h_{\reg} / W)}$ by 
$$
\KZ(M) := \mathsf{DR}(M[\delta^{-1}]^W) = (((M[\delta^{-1}])^W)^{an})^{\nabla}.
$$
The following diagram should help the reader unpack the definition of the $\KZ$-functor. 
\beq{eq:KZdiagram}
\xymatrix{
\mc{O} \ar[rr]^-{( - )[\delta^{-1}]} \ar@{.>}@/_3pc/[ddrrr]_{\KZ} & &  \Lmod{\dd(\h_{\reg}) \rtimes W} \ar[d]^{\wr}_{( - )^W} & \\
 & & \Lmod{\dd(\h_{\reg} / W)} \ar[r]^{( - )^{an}} & \Lmod{\dd(\h_{\reg}^{an} / W)} \ar[d]^{( - )^{\nabla}} \\
 & & & \Lmod{\pi_1(\h_{\reg} / W)}. 
}
\eeq

\begin{exercise}\label{ex:invdiff}
By considering the case of $\Z_2$, show that the natural map $\dd(\mf{h})^W \rightarrow \dd(\mf{h} / W)$ is not an isomorphism. Which of injectivity or surjectivity fails? Hint: for complete rigor, consider the associated graded map $\gr \dd(\mf{h})^W \rightarrow \gr \dd(\mf{h} / W)$.  
\end{exercise}

\subsection{A change of parameters}

In order to relate, in the next subsection, the rational Cherednik algebra $\H_{\mbf{c}}(W)$, via the $\KZ$-functor, with the cyclotomic Hecke algebra $\mc{H}_{\mbf{q}}(W)$, we need to change the way we parameterize $\H_{\mbf{c}}(W)$ (this is just some technical annoyance and you can skip this section unless you plan some hard core calculations using the $\KZ$-functor). Each complex reflection $s \in \mc{S}$ defines a reflecting hyperplane $H = \ker \alpha_s \subset \h$. Let $\mc{A}$ denote the set of all hyperplanes arrising this way. For a given $H \in \mc{A}$, the subgroup $W_H = \{ w \in W \ | \ w (H) \subset H \}$ of $W$ is cyclic. Let $W^*_H = W_H \backslash \{ 1 \}$. Then 
$$
\mc{S} = \bigcup_{H \in \mc{A}} W_H^*.
$$
We may, without loss of generality, assume that 
$$
\alpha_{H} := \alpha_{s} = \alpha_{s'}, \quad  \alpha_H^{\vee} := \alpha_s^{\vee} = \alpha_{s'}^{\vee}, \quad \forall \ s,s' \in W_H^*. 
$$
Then the original relations (\ref{eq:rel}) in the definition of the rational Cherednik algebra become
\begin{equation}\label{eq:rel2}
[y,x] = (y,x) - \sum_{H \in \mc{A}} (y,\alpha_H)(\alpha_H^\vee,x) \left( \sum_{s \in W_H^*} \mathbf{c}(s) s \right), \quad \forall \ x \in \h^*, y \in \h.
\end{equation}
Let $n_H = |W_H|$ and $e_{H,i} = \frac{1}{n_H} \sum_{w \in W_H} (\det w)^i w$ for $0 \le i \le n_H -1$ be the primitive idempotents in $\C W_H$. Define $k_{H,i} \in \C$ for $H \in \mc{A}$ and $0 \le i \le n_H -1$ by 
$$
\sum_{s \in W_H^*} \mathbf{c}(s) s = n_H \sum_{i = 0}^{n_H - 1} (k_{H,i+1} - k_{H,i}) e_{H,i},
$$
and $k_{H,0} = k_{H,n_H} = 0$. Note that this forces $k_{H,i} = k_{w(H),i}$ for all $w \in W$ and $H \in \mc{A}$. Thus, the parameters $k_{H,i}$ are $W$-invariant. This implies that 
$$
\mathbf{c}(s) = \sum_{i = 0}^{n_H - 1} (\det s)^i (k_{H,i+1} - k_{H,i}),
$$
and the relation (\ref{eq:rel2}) becomes
\begin{equation}\label{eq:rel3}
[y,x] = (y,x) - \sum_{H \in \mc{A}} (y,\alpha_H)(\alpha_H^\vee,x) n_H \sum_{i = 0}^{n_H - 1} (k_{H,i+1} - k_{H,i}) e_{H,i}, \quad \forall \ x \in \h^*, y \in \h.
\end{equation}
Therefore $\H_{\mbf{c}}(W) = \H_{\mbf{k}}(W)$, where 
$$
\mbf{k} = \{ k_{H,i} \ | \ H \in \mc{A}, \ 0 \le i \le n_H -1, k_{H,0} = k_{H,n_H} = 0 \textrm{ and } k_{H,i} = k_{w(H),i} \ \forall w,H,i \}.
$$ 

\subsection{The cyclotomic Hecke algebra}

The braid group $\pi_1(\h_{\reg} /W)$ has generators $\{ T_s \ | \ s \in \mc{S} \}$, where $T_s$ is an $s$-generator of the monodromy around $H$; see \cite[Section 4.C]{BMR} for the precise definition. The $T_s$ satisfy certain ``braid relations''. Fix $\mbf{q} = \{ q_{H,i} \in \Cs \ | \ H \in \mc{A}, \ 0 \le i \le n_H -1 \}$. The \textit{cyclotomic Hecke algebra} $\mc{H}_{\mbf{q}}(W)$ is the quotient of the group algebra $\C \pi_1(\h_{\reg} /W)$ by the two-sided ideal generated by 
$$
\prod_{i = 0}^{n_H - 1} (T_s - q_{H,i}), \quad \forall \ s \in \mc{S},
$$
where $H$ is the hyperplane defined by $s$. 

\begin{example}\label{examp:HeckeA} 
In type $A$ the Hecke algebra is the algebra generated by $T_1,\ds, T_{n-1}$, satisfying the braid relations
$$
T_i T_j = T_j T_i, \quad \forall \ |i - j| > 1,
$$
$$
T_{i} T_{i+1} T_i = T_{i+1} T_i T_{i+1} \quad \forall \ 1 \le i \le n-2
$$
and the additional relation
$$
(T_i - \mbf{q})(T_i + 1) = 0 \quad \forall \ 1 \le i \le n-1.
$$
\end{example}

For each $H \in \mc{A}$, fix a generator $s_H$ of $W_H$. Given a parameter $\mbf{k}$ for the rational Cherednik algebra, define $\mbf{q}$ by 
$$
q_{H,i} = (\det s_H)^{-i} \exp (2 \pi \sqrt{-1} k_{H,i}).
$$ 
Based on \cite[Theorem 4.12]{BMR}, the following key result was proved in \cite[Theorem 5.13]{GGOR}. 

\begin{thm}\label{thm:factor}
The $\KZ$-functor factors through $\Lmod{\mc{H}_{\mbf{q}}(W)}$. 
\end{thm}

Since each of the functors appearing in diagram \ref{eq:KZdiagram} is exact, the $\KZ$-functor is an exact functor. Therefore, by Watt's Theorem, \cite[Theorem 5.50]{Rotman}, there exists a projective module $P_{\KZ} \in \mc{O}$ such that 
\beq{eq:KZrep}
\KZ( - ) \simeq \Hom_{\H_{\mbf{c}}(W)}(P_{\KZ}, - ).
\eeq
Then, Theorem \ref{thm:factor} implies that $P_{\KZ}$ is a $(\H_{\mbf{c}}(W),\mc{H}_{\mbf{q}}(W))$-bimodule and the action of $\mc{H}_{\mbf{q}}(W)$ on the right of $P_{\KZ}$ defines an algebra morphism 
$$
\phi : \mc{H}_{\mbf{q}}(W) \longrightarrow \End_{\H_{\mbf{c}}(W)}(P_{\KZ})^{op}.
$$

\begin{lem}\label{lem:PKZdecomp}
We have a decomposition  
$$
P_{\KZ} = \bigoplus_{\lambda \in \Irr (W)} (\dim \KZ(L(\lambda))) P(\lambda).
$$
\end{lem}

\begin{proof}
By definition of projective cover, we have $\dim \Hom_{\H_{\mbf{c}}(W)}(P(\lambda), L(\mu)) = \delta_{\lambda,\mu}$. Therefore, if $P_{\KZ} = \bigoplus_{\lambda \in \Irr (W)} P(\lambda)^{\oplus n_{\lambda}}$, then $n_{\lambda} = \dim \Hom_{\H_{\mbf{c}}(W)}(P_{\KZ},L(\lambda))$. But, by definition, 
$$
\Hom_{\H_{\mbf{c}}(W)}(P_{\KZ},L(\lambda)) = \KZ(L(\lambda)).
$$
\end{proof}

\begin{lem}\label{lem:adjunction}
Let $\mc{A}$ be an abelian, Artinian category and $\mc{A}'$ a full subcategory, closed under quotients. Let $F : \mc{A}' \rightarrow \mc{A}$ be the inclusion functor. Define ${}^{\perp} F : \mc{A} \rightarrow \mc{A}'$ by setting ${}^{\perp} F (M)$ to be the largest quotient of $M$ contained in $\mc{A}'$. Then ${}^{\perp} F$ is left adjoint to $F$ and the adjunction $\eta : \mathrm{id}_{\mc{A}} \rightarrow F \circ ({}^{\perp} F)$ is surjective. 
\end{lem}

\begin{proof}
We begin by showing that ${}^{\perp} F$ is well-defined. We need to show that, for each $M \in \mc{A}$, there is a unique maximal quotient $N$ of $M$ contained in $\mc{A}'$. Let 
$$
K = \{ N' \subseteq M \ | \ M/ N \in \mc{A}' \}.
$$
Note that if $N_1'$ and $N_2'$ belong to $K$ then $N_1' \cap N_2'$ belongs to $K$. Therefore, if $N_1' \in K$ is not contained in all other $N' \in K$, we choose $N' \in K$ such that $N_1' \not\subset N'$ and set $N_2' = N' \cap N_1' \subsetneq N_1$. Continuing this way we construct a descending chain of submodules $N_1' \supsetneq N_2' \supsetneq \cdots$ of $M$. Since $\mc{A}$ is assumed to be Artinian, this chain must eventually stop. Hence, there is a unqiue minimal element under inclusion in $K$. It is clear that ${}^{\perp} F$ is left adjoint to $F$ and the adjunction $\eta$ just sends $M$ to the maximal quotient of $M$ in $\mc{A}'$, hence is surjective. 
\end{proof}

\begin{thm}[Double centralizer theorem]\label{thm:double2}
We have an isomorphism 
$$
\phi : \mc{H}_{\mbf{q}}(W) \stackrel{\sim}{\longrightarrow} \End_{\H_{\mbf{c}}(W)}(P_{\KZ})^{op}.
$$
\end{thm}

\begin{proof}
Let $\mc{A}$ denote the image of the $\KZ$ functor in $\Lmod{\mc{H}_{\mbf{q}}(W)}$. It is a full subcategory of $\Lmod{\mc{H}_{\mbf{q}}(W)}$. It is also closed under quotients. To see this, notice that it suffices to show that the image of $\mc{O}$ under the localization functor $( - )[\delta^{-1}]$ is closed under quotients. If $N$ is a non-zero quotient of $M[\delta^{-1}]$ for some $M \in \mc{O}$, then it is easy to check that the pre-image $N'$ of $M$ under $M \rightarrow M[\delta^{-1}] \twoheadrightarrow N$ is non-zero and generates $N$. The claim follows. Then, Lemma \ref{lem:adjunction} implies that $\phi$ is surjective. Hence to show that it is an isomorphism, it suffices to calculate the dimension of $\End_{\H_{\mbf{c}}(W)}(P_{\KZ})$. By Lemma \ref{lem:PKZdecomp}, 
$$
P_{\KZ} = \bigoplus_{\lambda \in \Irr (W)} \dim \KZ(L(\lambda)) P(\lambda).
$$
For any module $M \in \mc{O}$, we have $\dim \Hom_{\H_{\mbf{c}}(W)}(P(\lambda),M) = [M : L(\lambda)]$, the multiplicity of $L(\lambda)$ in a composition series for $M$. This can be proved by induction on the length of $M$, using the fact that $\dim \Hom_{\H_{\mbf{c}}(W)}(P(\lambda),L(\mu)) = \delta_{\lambda,\mu}$. Hence, using BGG reciprocity, we have 
\begin{align*}
\dim \End_{\H_{\mbf{c}}(W)}(P_{\KZ}) & = \bigoplus_{\lambda,\mu} \dim \KZ(L(\lambda)) \dim \KZ(L(\mu)) \Hom_{\H_{\mbf{c}}(W)}(P(\lambda),P(\mu)) \\
 & = \bigoplus_{\lambda,\mu} \dim \KZ(L(\lambda)) \dim \KZ(L(\mu)) [P(\mu):L(\lambda)] \\
 & = \bigoplus_{\lambda,\mu,\nu} \dim \KZ(L(\lambda)) \dim \KZ(L(\mu)) [P(\mu) : \Delta(\nu)] [\Delta(\nu) : L(\lambda)] \\
 & = \bigoplus_{\lambda,\mu,\nu} \dim \KZ(L(\lambda)) \dim \KZ(L(\mu)) [\Delta(\nu) : L(\mu)] [\Delta(\nu) : L(\lambda)] \\
 & = \bigoplus_{\nu} (\dim \KZ(\Delta(\nu)))^2\\
\end{align*}
Since $\Delta(\nu)$ is a free $\C[\h]$-module of rank $\dim (\nu)$, its localization $\Delta(\nu)[\delta^{-1}]$ is an integrable connection of rank $\dim(\nu)$. Hence, $\dim \KZ(\Delta(\nu)) = \dim (\nu)$ and thus $\dim \End_{\H_{\mbf{c}}(W)}(P_{\KZ}) = |W|$. 
\end{proof}

Let $\mc{O}_{\tor}$ be the Serre subcategory of $\mc{O}$ consisting of all modules that are torsion with respect to the Ore set $\{ \delta^N \}_{N \in \mathbb{N}}$. The torsion submodule $M_{\tor}$ of $M \in \mc{O}$ is the set $\{ m \in M \ | \  \exists \ N \gg 0 \ \textrm{s.t.} \ \delta^N \cdot m = 0 \}$. Then, $M$ is torsion if $M_{\tor} = M$. 

\begin{cor}
The $\KZ$-functor is a quotient functor with kernel $\mc{O}_{\tor}$ i.e.
$$
\KZ : \mc{O} / \mc{O}_{\tor} \stackrel{\sim}{\longrightarrow} \Lmod{\mc{H}_{\mbf{q}}(W)}.
$$
\end{cor}

\begin{proof}
Notice that, of all the functors in diagram \ref{eq:KZdiagram}, only the first, $M \mapsto M[\delta^{-1}]$ is not an equivalence. We have $M[\delta^{-1}] = 0$ if and only if $M$ is torsion. Therefore, $\KZ(M) = 0$ if and only if $M \in \mc{O}_{\tor}$. Thus, we just need to show that $\KZ$ is essentially surjective; that is, for each $N \in \Lmod{\mc{H}_{\mbf{q}}(W)}$ there exists some $M \in \mc{O}$ such that $\KZ(M) \simeq N$. We fix $N$ to be some finite dimensional $\mc{H}_{\mbf{q}}(W)$-module. Recall that $P_{KZ} \in \mc{O}$ is a $(\H_{\mbf{c}}(W),\mc{H}_{\mbf{q}}(W))$-bimodule. Therefore, $\Hom_{\H_{\mbf{c}}(W)}(P_{KZ},\H_{\mbf{c}}(W))$ is a $(\mc{H}_{\mbf{q}}(W),\H_{\mbf{c}}(W))$-bimodule and 
$$
M = \Hom_{\mc{H}_{\mbf{q}}(W)}(\Hom_{\H_{\mbf{c}}(W)}(P_{KZ},\H_{\mbf{c}}(W)),N)
$$
is a module in category $\mc{O}$. Applying (\ref{eq:KZrep}) and the double centralizer theorem, Theorem \ref{thm:double2}, we have 
\begin{align*}
\KZ(M) & = \Hom_{\H_{\mbf{c}}(W)}(P_{KZ},M) \\
 & = \Hom_{\H_{\mbf{c}}(W)}(P_{KZ},\Hom_{\mc{H}_{\mbf{q}}(W)}(\Hom_{\H_{\mbf{c}}(W)}(P_{KZ},\H_{\mbf{c}}(W)),N)) \\
 & \simeq \Hom_{\mc{H}_{\mbf{q}}(W)}(\Hom_{\H_{\mbf{c}}(W)}(P_{KZ},\H_{\mbf{c}}(W)) \o_{\H_{\mbf{c}}(W)} P_{KZ},N) \\
 & \simeq \Hom_{\mc{H}_{\mbf{q}}(W)}(\End_{\H_{\mbf{c}}(W)}(P_{KZ}), N) \\
 & \simeq \Hom_{\mc{H}_{\mbf{q}}(W)}(\mc{H}_{\mbf{q}}(W),N) = N,
\end{align*}
where we have used \cite[Proposition 4.4 (b)]{ARS} in the third line. 
\end{proof}

\subsection{Example}

Let's take $W = \Z_n$. In this case the Hecke algebra $\mc{H}_{\mbf{q}}(\Z_n)$ is generated by a single element $T := T_1$ and satisfies the defining relation
$$
\prod_i (T - q_i^{m_i}) = 0.
$$
Unlike examples of higher rank, the algebra $\mc{H}_{\mbf{q}}(\Z_n)$ is commutative. Let $\zeta \in \Cs$ be defined by $s (x) = \zeta x$. We fix $\alpha_s = \sqrt{2} x$ and $\alpha_s^{\vee} = \sqrt{2} y$, which implies that $\lambda_s = \zeta$. Let $\Delta(i) = \C[x] \o e_i$ be the standard module associated to the simple $\Z_n$-module $e_i$, where $s \cdot e_i = \zeta^i e_i$. The module $\Delta(i)$ is free as a $\C[x]$-module and the action of $y$ is uniquely defined by $y \cdot (1 \o e_i) = 0$. Since 
$$
y = \pa_x - \sum_{i = 1}^{m-1} \frac{2 \mbf{c}_i}{1 - \zeta^i} \frac{1}{x} (1 -s^i) 
$$
under the Dunkl embedding, we have $\Delta(i)[\delta^{-1}] = \C[x,x^{-1}] \o e_i$ with connection defined by 
$$
\pa_x \cdot e_i = \frac{a_i}{x} e_i
$$
where
$$
a_i := 2 \sum_{j = 1}^{m-1} \frac{\mbf{c}_j (1 - \zeta^{ij}}{1 - \zeta^j}.
$$
It is clear that this connection is regular. Then, the horizontal sections sheaf of $\Delta(i)[\delta^{-1}]$ (i.e. $\mathsf{DR}(\Delta(i)[\delta^{-1}])$) on $\h_{\reg}^{an}$ is dual to the sheaf of multivalued solutions $\C \cdot x^{a_i} = \Sol(\Delta(i)[\delta^{-1}])$ of the differential equation $x \pa_x - a_i = 0$. But this is not what we want. We first want to descend the $\dd(\Cs) \rtimes \Z_n$-module to the $\dd(\h_{\reg})^{\Z_n} = \dd(\h_{\reg} / \Z_n)$-module $(\Delta(i)[\delta^{-1}])^{\Z_n}$. Then $\KZ(\Delta(i))$ is defined to be the horizontal sections of $(\Delta(i)[\delta^{-1}])^{\Z_n}$. Let $z = x^n$ so that $\C[\h_{\reg} / \Z_n] = \C[z,z^{-1}]$. Then, an easy calculation shows that $\pa_z = \frac{1}{n x^{n-1}} \pa_x$ (check this!). Since
$$
(\Delta(i)[\delta^{-1}])^{\Z_n} = \C[z,z^{-1}] \cdot (x^{n-i} \o e_i) =: \C[z,z^{-1}] \cdot u_i,
$$
we see that 
$$
\pa_z \cdot (x^{n-i} \o e_i) = \frac{1}{n x^{n-1}} \pa_x \cdot (x^{n-i} \o e_i) = \frac{n - i + a_i}{n z} u_i.
$$
Hence, by equation (\ref{eq:dual}) $\KZ(\Delta(i))$ is the duality functor $\mathbb{D}'$ applied to the local system of solutions $\C \cdot z^{b_i} = \Sol((\Delta(i)[\delta^{-1}])^{\Z_n})$, where 
$$
b_i = \frac{n - i + a_i}{n}.
$$
At this level, the duality functor $\mathbb{D}'$ simply sends the local system $\C \cdot z^{b_i}$ to the local system $\C \cdot z^{-b_i}$. The generator $T$ of $\pi_1(\h_{\reg} / \Z_n)$ is represented by the loop $t \mapsto \exp(2 \pi \sqrt{-1} t)$. Therefore 
$$
T \cdot z^{-b_i} = \exp(-2 \pi \sqrt{-1} b_i) z^{-b_i}.
$$
It turns out that, in the rank one case, $L(i)[\delta^{-1}] = 0$ if $L(i) \neq \Delta(i)$. Thus,
$$
\KZ(L(i)) = \left\{ \begin{array}{ll}
\KZ(\Delta(i)) & \textrm{if $L(i) = \Delta(i)$} \\
0 & \textrm{otherwise}.
\end{array} \right.
$$

\subsection{Application}

As an application of the Double centralizer theorem, Theorem \ref{thm:double2}, we mention the following very useful result due to Vale, \cite{ValeThesis}.

\begin{thm}
The following are equivalent: 
\begin{enumerate}
\item $\H_{\mbf{k}}(W)$ is a simple ring. 
\item Category $\mc{O}$ is semi-simple.
\item The cyclotomic Hecke algebra $\mc{H}_{\mbf{q}}(W)$ is semi-simple.
\end{enumerate}
\end{thm}

\subsection{The $\KZ$ functor for $\Z_2$}

In this section, which is designed to be an exercise for the reader, we'll try to describe what the $\KZ$ functor does to modules in category $\mc{O}$ when $W = \Z_2$, our favourite example. The Hecke algebra $\mc{H}_{\mbf{q}}(\Z_2)$ is the algebra generated by $T := T_1$ and satisfying the relation $(T - 1)(T - \mbf{q}) = 0$. The defining relation for $\H_{\mbf{c}}(\Z_2)$ is
$$
[y,x] = 1 - 2 \mbf{c}s,
$$ 
see example \ref{example:s2}. We have $\mbf{q} = \det (s) \exp (2 \pi \sqrt{-1} \mbf{c}) = - \exp (2 \pi \sqrt{-1} \mbf{c})$. 

\begin{exercise}\label{ex:lambda2}
Describe $\KZ(\Delta(\lambda))$ as a $\mc{H}_{\mbf{q}}(\Z_2)$-module. 
\end{exercise}

Next, we will calculate $\KZ(P(\rho_1))$, assuming that $\mbf{c} = \frac{1}{2} + m$ for some $m \in \Z_{\ge 0}$.  For this calculation, we will use the explicit description of $P(\rho_1)$ given in section \ref{sec:quiver}. Recall that $P(\rho_1)$ is the $\H_{\mbf{c}}(\Z_2)$-module $\C[x] \o \rho_1 \oplus \C[x] \o \rho_0$ with 
$$
x \cdot (1 \o \rho_1) = x \o \rho_1 + 1 \o \rho_0, \quad x \cdot (1 \o \rho_0) = x \o \rho_0,
$$
$$
y \cdot (1 \o \rho_1) = x^{2m} \o \rho_0, \quad y \cdot (1 \o \rho_0) = 0.
$$
This implies that 
$$
\frac{1}{x} \cdot (1 \o \rho_1) = \frac{1}{x} \o \rho_1 - \frac{1}{x^{2}} \o \rho_0.
$$
If we write $P(\rho_1)[\delta^{-1}] = \C[x^{\pm 1}] \cdot a_1 \oplus \C[x^{\pm 1}] \cdot a_0$, where $a_1 = 1 \o \rho_1$ and $a_0 = 1 \o \rho_0$, then
$$
\pa_x \cdot a_1 = (y + \frac{c}{x} (1 - s)) \cdot a_1 = y \cdot a_1 + \frac{2c}{x} \cdot a_1. 
$$
Now,  
$$
y \cdot a_1 = x^{2m} \o \rho_0 = x^{2m} \cdot a_0,
$$
hence $\pa_x \cdot a_1 = \frac{2c}{x} \cdot a_1 + x^{2m} \cdot a_0$. Also, $\pa_x \cdot a_0 = 0$. A free $\C[z^{\pm 1}]$-basis of $P(\rho_0)[\delta^{-1}]^{\Z_2}$ is given by $u_1 = x \cdot a_1$ and $u_0 = a_0$. Therefore
$$
\pa_z \cdot u_1 = \frac{1 + 2c}{2 z} u_1 + \frac{1}{2} z^m u_0 = \frac{m + 1}{z} u_1 + \frac{1}{2} z^m u_0
$$
and $\pa_z \cdot u_0 = 0$, where we used the fact that $\mbf{c} = \frac{1}{2} + m$. Hence $\KZ(P(\rho_1))$ is given by the connection
$$
\pa_z + \left( \begin{array}{cc}
\frac{m+1}{z} & 0\\
\frac{1}{2} z^m & 0 
\end{array} \right) .
$$
Two linearly independent solutions of this equation are
$$
g_1 (z) = \left( \begin{array}{c}
z^{-(m+1)}\\
\frac{1}{2} \mathrm{ln} (z)   
\end{array} \right), \quad g_2 (z) = \left( \begin{array}{c}
0\\
1  
\end{array} \right).
$$
If, in a small, simply connected neighborhood of $1$, we choose the branch of $\mathrm{ln} (z)$ such that $\mathrm{ln} (1) = 0$, then $\gamma(0) = 0$ and $\gamma(1) = 2 \pi \sqrt{-1}$, where 
$$
\gamma : [0,1] \rightarrow \C, \quad \gamma(t) = \mathrm{ln}( \exp(2 \pi \sqrt{-1} t)).
$$
Therefore $\KZ(P(\rho_1))$ is the two-dimensional representation of $\mc{H}_{\mbf{q}}(\Z_2)$ given by 
$$
T \mapsto \left( \begin{array}{cc}
1 & 0\\
2 \pi \sqrt{-1} & 1 
\end{array} \right) .
$$
This is isomorphic to the left regular representation of $\mc{H}_{\mbf{q}}(\Z_2)$.

\begin{exercise}\label{ex:PKZ}
For all $\mbf{c}$, describe $P_{\KZ}$. 
\end{exercise}
 
\subsection{Additional remark}

\begin{itemize}
\item Most of the results in this lecture first appeared in \cite{GGOR} and our exposition is based mainly on this paper. 
\item Further details on the $\KZ$-functor are also contained in \cite{RouquierSurvey}. 
\end{itemize}

\newpage

\section{Symplectic reflection algebras at $t = 0$}\label{sec:five}

Recall from the first lecture that we used the Satake isomorphism to show that 
\begin{itemize}
\item The algebra $Z(\H_{0,\mbf{c}}(G))$ is isomorphic to $\mbf{e} \H_{0,\mbf{c}}(G) \mbf{e}$ and $\H_{0,\mbf{c}}(G)$ is a finite $Z(\H_{0,\mbf{c}}(G))$-module. 
\item The centre of $\H_{1,\mbf{c}}(G)$ equals $\C$. 
\end{itemize}

In this lecture we'll consider symplectic reflection algebras ``at $t = 0$'' and, in particular, the geometry of $\ZH_{\mbf{c}}(G) := Z(\H_{0,\mbf{c}}(G))$. 

\begin{defn}
The \textit{generalized Calogero-Moser space} $X_{\mbf{c}}(G)$ is defined to be the affine variety $\Spec \ZH_{\mbf{c}}(G)$. 
\end{defn}

The (classical) Calogero-Moser space was introduced by Kazhdan, Kostant and Sternberg \cite{KKS} and studied further by Wilson in the wonderful paper \cite{Wilson}. Calogero \cite{Calogero} studied the integrable system describing the motion of $n$ massless particles on the real line with a repulsive force between each pair of particles, proportional to the square of the distance between them. In \cite{KKS}, Kazhdan, Kostant and Sternberg give a description of the corresponding  phase space in terms of Hamiltonian reduction. By considering the real line as being the imaginary axis sitting in the complex plane, Wilson interprets the Calogero-Moser phase space as an affine variety
\begin{equation}\label{eq:ClassicalCalogeroMoser}
\mc{C}_n = \{ (X,Y;u,v) \in \textrm{Mat}_n(\C) \times \textrm{Mat}_n(\C) \times \C^n \times (\C^n)^* \, | \, [X,Y] + I_n = v \cdot u \, \} // GL_n(\C).
\end{equation}
He showed, \cite[\S 1]{Wilson}, that $\mc{C}_n$ is a smooth, irreducible, symplectic affine variety. For further reading see \cite{EtingofCalogeroMoser}. The relation to rational Cherednik algebras comes from an isomorphism by Etingof and Ginzburg between the affine variety $X_{\mbf{c} = 1}(\s_n) = \Spec \, \ZH_{\mbf{c} = 1}(S_n)$, and the Calogero-Moser space $\mc{C}_n$:
$$
\psi_n \, : \, X_{\mbf{c} = 1}(\s_n) \stackrel{\sim}{\longrightarrow} \mc{C}_n.
$$
It is an isomorphism of affine symplectic varieties and implies that $X_{\mbf{c}}(\s_n)$ is smooth when $\mbf{c} \neq 0$.\\

The filtration on $\H_{0,\mbf{c}}(G)$ induces, by restriction, a filtration on $\ZH_{\mbf{c}}(G)$. Since the associated graded of $\ZH_{\mbf{c}}(G)$ is $\C[V]^G$, $X_{\mbf{c}}(G)$ is reduced and irreducible. 

\begin{example}
When $G = \Z_2$ acts on $\C^2$, the centre of $\H_{\mbf{c}}(\Z_2)$ is generated by $A := x^2, B := xy - \mbf{c} s$ and $C = y^2$. Thus,
$$
X_{\mbf{c}}(\Z_2) \simeq \frac{\C[A,B,C]}{(AC - (B + \mbf{c})(B - \mbf{c}))}
$$
is the affine cone over $\mathbb{P}^1 \subset \mathbb{P}^2$ when $\mbf{c} = 0$, but is a smooth affine surface for all $\mbf{c} \neq 0$, see figure \ref{fig:resdef}.  
\end{example}

\subsection{Representation Theory}

Much of the geometry of the generalized Calogero-Moser space is encoded in the representation theory of the corresponding symplectic reflection algebra (a consequence of the double centralizer property!). In particular, a closed point of $X_{\mbf{c}}$ is singular if and only if there is a ``small'' simple module supported at that point - this statement is made precise in Proposition \ref{prop:singulariffsmall} below. 

The fact that $\H_{0,\mbf{c}}(G)$ is a finite module over its centre implies that it is an example of a P.I. (\textit{polynomial identity}) ring. This is a very important class of rings in classical ring theory and can be thought of as rings that are "close to being commutative". We won't recall the definition of a P.I. ring here, but refer the read to Appendix I.13 of the excellent book \cite{BrownGoodearlbook}. 

\begin{lem}\label{lem:dimapprox}
There exists some $N > 0$ such that $\dim L \le N$ for \textit{all} simple $\H_{0,\mbf{c}}(G)$-modules $L$. 
\end{lem}

\begin{proof}
It is a consequence of Kaplansky's Theorem, \cite[Theorem 13.3.8]{MR}, that every simple $\H_{0,\mbf{c}}(G)$-module is a finite dimensional vector space over $\C$. More precisely, to every prime P.I. ring is associated its P.I. degree. Then, Kaplansky's Theorem implies that if $L$ is a simple $\H_{0,\mbf{c}}(G)$-module then $\dim L \le \ \textrm{P.I. degree} \, (\H_{0,\mbf{c}}(G))$ and $\H_{0,\mbf{c}}(G) / \ann_{\H_{0,\mbf{c}}(G)} \, L \simeq \textrm{Mat}_m(\C)$.  
\end{proof}

Schur's lemma says that the elements of the centre $\ZH_{\mbf{c}}(G)$ of $\H_{0,\mbf{c}}(G)$ act as scalars on any simple $\H_{0,\mbf{c}}$-module $L$. Therefore, the simple module $L$ defines a character $\chi_L \, : \, \ZH_{\mbf{c}}(G) \rightarrow \C$ and the kernel of $\chi_L$ is a maximal ideal in $\ZH_{\mbf{c}}(G)$. Thus, the character $\chi_L$ corresponds to a closed point in $X_{\mbf{c}}(G)$. Without loss of generality, we will refer to this point as $\chi_L$ and denote by $\ZH_{\mbf{c}}(G)_{\chi_L}$ the localization of $\ZH_{\mbf{c}}(G)$ at the maximal ideal $\Ker \, \chi_L$. We denote by $\H_{0,\mbf{c}}(G)_{\chi}$ the central localization $\H_{0,\mbf{c}}(G) \otimes_{\ZH_{\mbf{c}}(G)} \ZH_{\mbf{c}}(G)_{\chi}$. The Azumaya locus of $\H_{0,\mbf{c}}(G)$ over $\ZH_{\mbf{c}}(G)$ is defined to be 
$$
\mc{A}_{\mbf{c}} := \{ \chi \in X_{\mbf{c}}(W) \, | \, \H_{0,\mbf{c}}(W)_{\chi} \textrm{ is Azumaya over } Z_{\mbf{c}}(W)_{\chi} \}.
$$
As shown in \cite[Theorem III.1.7]{BrownGoodearlbook}, $\mc{A}_{\mbf{c}}$ is a non-empty, open subset of $X_{\mbf{c}}(W)$. 

\begin{remark}\label{rem:ArtinProcesi}
If you are not familiar with the (slightly technical) definition of Azumaya locus, as given in \cite[Section 3]{BrownGoodearl}, then it suffices to note that it is a consequence of the Artin-Procesi Theorem \cite[Theorem 13.7.14]{MR} that the following are equivalent:
\begin{enumerate}
\item $\chi \in \mc{A}_{\mbf{c}}$;
\item $\dim L = \ \textrm{P.I. degree} \, (\H_{0,\mbf{c}}(G))$ for all simple modules $L$ such that $\chi_L = \chi$;
\item there exists a unique simple module $L$ such that $\chi_L = \chi$. 
\end{enumerate}
\end{remark}

In fact, one can say a great deal more about these simple modules of maximal dimension. The following result strengthens Lemma \ref{lem:dimapprox}.

\begin{thm}\label{thm:genericisregularrep}
Let $L$ be a simple $\H_{0,\mbf{c}}(G)$-module. Then $\dim L \le |G|$ and $\dim L = |G|$ implies that $L \simeq \C G$ as a $G$-module. 
\end{thm}

\begin{proof}
We will prove the theorem when $\H_{0,\mbf{c}}(G)$ is a rational Cherednik algebra, by using the Dunkl embedding. The proof for arbitrary symplectic reflection algebras is much harder. 

By the theory of prime P.I. rings and their Azumaya loci, as described above, it suffices to show that there is some dense open subset $U$ of $X_{\mbf{c}}(W)$ such that $L \simeq \C G$ for all simple modules $L$ supported on $U$. Recall that the Dunkl embedding at $t = 0$ gives us an identification $\H_{0,\mbf{c}}(G)[\delta^{-r}] \simeq \C[\h_{\reg} \times \h^*] \rtimes G$, where $r > 0$ such that $\delta^r \in \C[\h]^G$. Then it suffices to show that every simple $ \C[\h_{\reg} \times \h^*] \rtimes G$-module is isomorphic to $\C G$ as a $G$-module. The centre of $ \C[\h_{\reg} \times \h^*] \rtimes G$ (which is just the centre of $\H_{0,\mbf{c}}(G)$ localized at $\delta^{r}$) equals $\C[\h_{\reg} \times \h^*]^G$. For each maximal ideal $\mf{m} \lhd \C[\h_{\reg} \times \h^*]^G$, we will construct a module $L(\mf{m})$ such that $\mf{m} \cdot L(\mf{m}) = 0$ and $L(\mf{m}) \simeq \C G$ as $G$-modules. Then, we'll argue that $\dim L \ge |G|$ for all $\C[\h_{\reg} \times \h^*] \rtimes G$-modules. 

So fix a maximal ideal $\mf{m}$ in $\C[\h_{\reg} \times \h^*]^G$ and let $\mf{n} \lhd \C[\h_{\reg} \times \h^*]$ be a maximal ideal such that $\mf{n} \cap \C[\h_{\reg} \times \h^*]^G = \mf{m}$ (geometrically, we have a finite map $\rho : \h_{\reg} \times \h^* / G \twoheadrightarrow (\h_{\reg} \times \h^*)/G$ and we're choosing some point in the pre-image of $\mf{m}$). If $\C_{\mf{n}}$ is the one dimensional $\C[\h_{\reg} \times \h^*]$-module on which $\mf{n}$ acts trivially, then define
$$
L(\mf{m}) = (\C[\h_{\reg} \times \h^*] \rtimes G) \o_{\C[\h_{\reg} \times \h^*]} \C_{\mf{n}}.
$$
The fact that $\rho^{-1}(\mf{m})$ consists of a single free $G$-orbit implies that $L(\mf{m})$  has the desired properties. 

Now let $L$ be any finite dimensional $\C[\h_{\reg} \times \h^*] \rtimes G$-module. Then $L = \bigoplus_{p \in S} L_p$, as a $\C[\h_{\reg}]$-module. Here $S \subset \h_{\reg}$ is some finite set and $L_{p} = \{ l \in L \ | \ \mf{n}_p^N \cdot l = 0 \textrm{ some $N \gg 0$ } \}$. Now $S$ is $G$-stable and multiplication by $g \in G$ defines an isomorphism of vector spaces $L_p \rightarrow L_{g(p)}$. Since $G \cdot p$ is a free $G$-orbit in $S$, this implies that $\dim L \ge |G | \dim L_p$ as required.   
\end{proof}

\begin{thm}\label{prop:singulariffsmall}
Let $L$ be a simple $\H_{\mbf{c}}(G)$-module then $\dim L = |G|$ if and only if $\chi_L$ is a nonsingular point of $X_{\mbf{c}}(G)$.
\end{thm}

\begin{proof}[Outline of proof]
By Theorem \ref{thm:genericisregularrep}, the dimension of a generic simple module is $|W|$. Since the Azumaya locus $\mc{A}_{\mbf{c}}$ is dense in $X_{\mbf{c}}$, it follows that $\textrm{P.I. degree} \, (\H_{0,\mbf{c}}(G)) = |G|$. The proposition will then follow from the equality $\mc{A}_{\mbf{c}} = X_{\mbf{c}}(G)_{sm}$, where $X_{\mbf{c}}(G)_{sm}$ is the smooth locus of $X_{\mbf{c}}(G)$. As noted in Corollary \ref{cor:finiteglobal}, $\H_{0,\mbf{c}}(G)$ has finite global dimension. It is known, \cite[Lemma III.1.8]{BrownGoodearlbook}, that this implies that $\mc{A}_{\mbf{c}} \subseteq (X_{\mbf{c}})_{sm}$. The opposite inclusion is an application of a result by Brown and Goodearl, \cite[Theorem 3.8]{BrownGoodearl}. Their theorem says that $(X_{\mbf{c}})_{sm} \subseteq \mc{A}_{\mbf{c}}$ (in fact that we have equality) if $\H_{0,\mbf{c}}(G)$ has particularly nice homological properties - it must be Auslander-regular and Cohen-Macaulay, and the complement of $\mc{A}_{\mbf{c}}$ has codimension at least two in $X_{\mbf{c}}$. The fact that $\H_{0,\mbf{c}}(G)$ is Auslander-regular and Cohen-Macaulay can be deduced from the fact that its associated graded, the skew group ring, has these properties (the results that are required to show this are listed in the proof of \cite[Theorem 4.4]{BrownSurvey}). The fact that the complement of $\mc{A}_{\mbf{c}}$ has co-dimension at least two in $X_{\mbf{c}}$ is harder to show. It follows from the fact that $X_{\mbf{c}}$ is a symplectic variety, Theorem \ref{thm:symplvar}, and that the ``representation theory of $\H_{0,\mbf{c}}$ is constant along orbits'', Theorem \ref{thm:constant}.   
\end{proof}

Theorem \ref{prop:singulariffsmall} implies that to answer the question 

\begin{question}
Is the generalized Calogero-Moser space smooth?
\end{question}

it suffices to compute the dimension of simple $\H_{\mbf{c}}(G)$-modules. Unfortunately, this turns out to be rather difficult to do. 

\subsection{Poisson algebras}\label{sec:Poisson}

The extra parameter $t$ in $\H_{t,\mbf{c}}(G)$ gives us a canonical quantization of the space $X_{\mbf{c}}(G)$. As a consequence, this implies that $X_{\mbf{c}}(G)$ is a Poisson variety. Recall:

\begin{defn}
A \textit{Poisson algebra} $(A,\{ - , - \})$ is a commutative algebra with a bracket $\{ - , - \} : A \o A \rightarrow A$ such that 
\begin{enumerate}
\item The pair $(A,\{ - , - \})$ is a Lie algebra. 
\item $\{ a, - \} : A \rightarrow A$ is a derivation for all $a \in A$ i.e. 
$$
\{ a, bc \} = \{ a, b \} c + b \{ a, c \}, \quad \forall a,b,c \in A.
$$
\end{enumerate}
\end{defn}

An ideal $I$ in the Poisson algebra $A$ is called \textit{Poisson} if $\{ I , A \} \subseteq I$. 

\begin{exercise}
Let $I$ be a Poisson ideal in $A$. Show that $A /I$ natural inherits a Poisson bracket from $A$, making it a Poisson algebra. 
\end{exercise} 

Hayashi's construction, \cite{Hayashi}: We may think of $t$ as a variable so that $\H_{0,\mbf{c}}(G) = \H_{t,\mbf{c}}(G) / t \cdot \H_{t,\mbf{c}}(G)$. For $z_1, z_2 \in \ZH_{\mbf{c}}(G)$ define 
$$
\{ z_1, z_2 \} = \left( \frac{1}{t} [\hat{z}_1, \hat{z}_2] \right) \mod \ t \H_{t,\mbf{c}}(G), 
$$
where $\hat{z}_1, \hat{z}_2$ are arbitrary lifts of $z_1, z_2$ in $\H_{t,\mbf{c}}(G)$. 

\begin{prop}\label{prop:bracket}
Since $\H_{0,\mbf{c}}(W)$ is flat over $\C[t]$, $\{ - , - \}$ is a well-defined Poisson bracket on $\ZH_{\mbf{c}}(G)$. 
\end{prop}

\begin{proof}
Write $\rho : \H_{t,\mbf{c}}(G) \rightarrow \H_{0,\mbf{c}}(G)$ for the quotient map. Let us first check that the binary operation is well-defined. Let $\hat{z}_1,\hat{z}_2$ be arbitrary lifts of $z_1,z_2 \in \ZH_{\mbf{c}}(G)$. Then $\rho([\hat{z}_1,\hat{z}_2]) = [\rho(\hat{z}_1),\rho(\hat{z}_2)] = 0$. Therefore, there exists some $\hat{z}_3 \in \H_{t,\mbf{c}}(G)$ such that $[\hat{z}_1,\hat{z}_2] = t \cdot \hat{z}_3$. Since $t$ is a non-zero divisor, $\hat{z}_3$ is uniquely define. The claim is that $\rho(\hat{z}_3) \in \ZH_{\mbf{c}}(G)$: let $h \in \H_{0,\mbf{c}}(G)$ and $\hat{h}$ an arbitrary lift of $h$ in $\H_{t,\mbf{c}}(G)$. Then 
$$
[h,\rho(\hat{z}_3)] = \rho([\hat{h},\hat{z}_3]) = \rho \left( \left[ \hat{h},\frac{1}{t} [\hat{z}_1,\hat{z}_2]\right] \right)
$$
$$
= - \rho \left(\frac{1}{t} [\hat{z}_2,[\hat{h},\hat{z}_1]]\right) - \rho \left(\frac{1}{t}[\hat{z}_2,[\hat{h},\hat{z}_1]]\right).
$$
Since $\hat{z}_1$ and $\hat{z}_1$ are lifts of central elements, the expressions $[\hat{z}_2,[\hat{h},\hat{z}_1]]$ and $[\hat{z}_2,[\hat{h},\hat{z}_1]]$ are in $t^2 \H_{t,\mbf{c}}(G)$. Hence $[h,\rho(\hat{z}_3)] = 0$. Therefore, the expression $\{ z_1, z_2 \}$ is well-defined. The fact that $\rho([t \cdot \hat{z}_1, \hat{z}_2] / t) = [\rho(\hat{z}_1),\rho(\hat{z}_2)] = 0$ implies that the bracket is independent of choice of lifts. The fact that the bracket makes $\ZH_{\mbf{c}}(G)$ into a Lie algebra and satisfies the derivation property is a consequence of the fact that the commutator bracket of an algebra also has these properties.
\end{proof}

\begin{remark}
The same construction makes $e \H_{0,\mbf{c}}(G) e$ into a Poisson algebra such that the Satake isomorphism is an isomorphism of Poisson algebras.
\end{remark}

\subsection{Symplectic leaves}

In the algebraic world there are several different definitions of symplectic leaves, which can be shown to agree in ``good'' cases. We will define two of them here. First, assume that $X_{\mbf{c}}(G)$ is smooth. Then $X_{\mbf{c}}(G)$ may be considered as a complex analytic manifold equipped with the analytic topology. In this case, the \textit{symplectic leaf} through $\mf{m} \in X_{\mbf{c}}(G)$ is the maximal connected analytic submanifold $\mc{L}(\mf{m})$ of $X_{\mbf{c}}(G)$ which contains $\mf{m}$ and on which $\{ - , - \}$ is non-degenerate. An equivalent definition is to say that $\mc{L}(\mf{m})$ is the set of all points that can be reached from $\mf{m}$ by traveling along integral curves corresponding to the Hamiltonian vector fields $\{ z, - \}$ for $z \in \ZH_{\mbf{c}}(G)$. 

Let us explain in more detail what is meant by this. Let $v$ be a vector field on a complex analytic manifold $X$ i.e. $v$ is a holomorphic map $X \rightarrow T X$ such that $v(x) \in T_x X$ for all $x \in X$ ($v$ is assigning, in a continuous manner, a tangent vector to each point of $x$). An integral curve for $v$ through $x$ is a holomorphic function $\Phi_{x,v} : B_{\epsilon}(0) \rightarrow X$, where $B_{\epsilon}(0)$ is a closed ball of radius $\epsilon$ around $0$ in $\C$, such that $(d_0 \Phi_{x,v})(1) = v(x)$ i.e. the derivative of $\Phi_{x,v}$ at $0$ maps the basis element $1$ of $T_0 \C = \C$ to the tangent vector field $v(x)$. The existence and uniqueness of holomorphic solutions to ordinary differential equations implies that $\Phi_{x,v}$ exists, and is unique, for each choice of $v$ and $x$. 

Now assume that $v = \{ a , - \}$ is a Hamiltonian vector field, and fix $x \in X$. Then, the image of $\Phi_{x,\{a, - \} }$ is, by definition, contained in the symplectic leaf $\mc{L}_x$ through $x$. Picking another point $y \in \Phi_{x,\{a, - \} }(B_{\epsilon}(0))$ and another Hamiltonian vector field $\{ b, - \}$, we again calculate the integral curve $\Phi_{y,\{ b , - \} }$ and its image is again, by definition, contained in $\mc{L}_x$. Continuing in this way for as long as possible, $\mc{L}_x$ is the set of all points one can reach from $x$ by ``flowing along Hamiltonian vector fields''. 

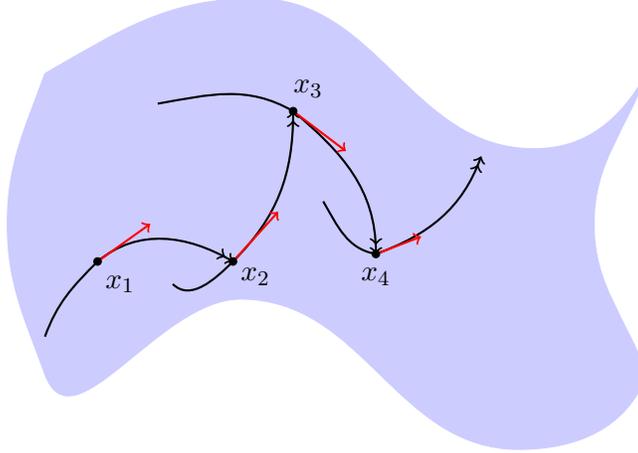
\begin{figure}\label{fig:intcurves}
\begin{tikzpicture}
\draw [blue!20!white, fill=blue!20!white] (-4,2) to [out=30,in=180] (-1.2,3) to [out=0,in=180] (2.5,1) to [out=0,in=-120] (4,2) to [out=-110,in=90] (3.3,0) to [out=-90,in=110] (4,-2) to [out=-110,in=0] (2.3,-3) to [out=180,in=0] (-1.4,-1) to [out=180,in=-70] (-4,-2) to [out=110,in=-90] (-4.5,0) to [out=90,in=-110] (-4,2);

\draw [thick,->>] (-4,-1.5) to [out=70,in=-135] (-3.3,-0.5) to [out=40,in=150] (-1.5,-0.5);
\draw [thick,red,->] (-3.3,-0.5) -- (-2.6,0);
\draw [fill] (-3.3,-0.5) circle [radius=0.05];
\node at (-3,-0.8) {$x_1$};

\draw [thick,->>] (-2.3,-0.8) to [out=-45,in=-135] (-1.5,-0.5) to [out=45,in=-90] (-0.7,1.5);
\draw [thick,red,->] (-1.5,-0.5) -- (-0.9,0.16);
\draw [fill] (-1.5,-0.5) circle [radius=0.05];
\node at (-1.2,-0.7) {$x_2$};

\draw [thick,->>] (-2.5,1.6) to [out=10,in=150] (-0.7,1.5) to [out=-40,in=90] (0.4,-0.4);
\draw [thick,red,->] (-0.7,1.5) -- (0,0.97);
\draw [fill] (-0.7,1.5) circle [radius=0.05];
\node at (-0.5,1.8) {$x_3$};

\draw [thick,->>] (-0.3,0.3) to [out=-60,in=170] (0.4,-0.4) to [out=20,in=-110] (1.8,0.9);
\draw [thick,red,->] (0.4,-0.4) -- (1,-0.17);
\draw [fill] (0.4,-0.4) circle [radius=0.05];
\node at (0.4,-0.7) {$x_4$};

\end{tikzpicture}
\caption{Flowing along integral curves in a symplectic leaf. The red vector at $x_i$ is a Hamiltonian vector field and the curve through $x_i$ is the corresponding integral curve.}
\end{figure}

In particular, this defines a stratification of $X_{\mbf{c}}(G)$. If, on the other hand, $X_{\mbf{c}}(G)$ is not smooth, then we first stratify the smooth locus of $X_{\mbf{c}}(G)$. The singular locus $X_{\mbf{c}}(G)_{\mathrm{sing}}$ of $X_{\mbf{c}}(G)$ is a Poisson subvariety. Therefore, the smooth locus of $X_{\mbf{c}}(G)_{\mathrm{sing}}$ is again a Poisson manifold and has a stratification by symplectic leaves. We can continue by considering the ``the singular locus of the singular locus'' of $X_{\mbf{c}}(G)$ and repeating the argument... This way we get a stratification of the whole of $X_{\mbf{c}}(G)$ by symplectic leaves.

\subsection{Symplectic cores} Let $\mf{p}$ be a prime ideal in $\ZH_{\mbf{c}}(G)$. Then there is a (necessarily unique) largest Poisson ideal $\mc{P}(\mf{p})$ contained in $\mf{p}$. Define an equivalence relation $\sim$ on $X_{\mbf{c}}(G)$ by saying 
$$
\mf{p} \sim \mf{q} \Leftrightarrow \mc{P}(\mf{p}) = \mc{P}(\mf{q}).
$$
The \textit{symplectic cores} of $X_{\mbf{c}}(G)$ are the equivalence classes defined by $\sim$. We write
$$
\mc{C}(\mf{p}) = \{ \mf{q} \in X_{\mbf{c}}(G) \ | \ \mc{P}(\mf{p}) = \mc{P}(\mf{q}) \}.
$$
Then, each symplectic core $\mc{C}(\mf{p})$ is a locally closed subvariety of $X_{\mbf{c}}(G)$ and $\overline{\mc{C}(\mf{p})} = V(\mc{P}(\mf{p}))$. The set of all symplectic cores is a partition of $X_{\mbf{c}}(G)$ into locally closed subvarieties. As one can see from the examples below, a Poisson variety $X$ will typically have an infinite number of symplectic leaves and an infinite number of sympectic cores. 

\begin{defn}
We say that the Poisson bracket on $X$ is \textit{algebraic} if $X$ has only finitely many symplectic leaves.
\end{defn}

Proposition 3.7 of \cite{PoissonOrders} says:

\begin{prop}
If the Poisson bracket on $X$ is algebraic then the symplectic leaves are locally closed algebraic sets and that the stratification by symplectic leaves equals the stratification by symplectic cores. 
\end{prop}

Thats it, $\mc{L}(\mf{m}) = \mc{C}(\mf{m})$ for all maximal ideals $\mf{m} \in X_{\mbf{c}}(G)$.

\begin{example}
We consider the Poisson bracket on $\C^2 = \Spec \C[x,y]$ given by $\{ x,y \} = y$, and try to describe the symplectic leaves in $\C^2$. From the definition of a Poisson algebra, it follows that each function $f \in \C[x,y]$ defines a vector field $\{ f, - \}$ on $\C^2$. For the generators $x,y$, these vector fields are $\{ x,  - \} = y \pa_y$, $\{ y , - \} = - y \pa_x$ respectively. In order to calculate the symplectic leaves we need to calculate the integral curve through a point $(p,q) \in \C^2$ for each of these vector fields. Then the leaf through $(p,q)$ will be the submanifold traced out by all these curves. We begin with $y \pa_y$. The corresponding integral curve is $a = (a_1(t),a_2(t)) : B_{\epsilon}(0) \rightarrow \C^2$ such that $a(0) = (p,q)$ and 
$$
a'(t) = (y \pa_y)_{a(t)}, \quad \forall t \in B_{\epsilon}(0).
$$
Thus, $a_1'(t) = 0$ and $a_2'(t) = a_2(t)$ which means $a = (p,q e^{t})$. Similarly, if $b = (b_1,b_2)$ is the integral curve through $(p,q)$ for $-y \pa_x$ then $b = (- q t + p, q)$. Therefore, there are only two symplectic leaves, $\{0 \}$ and $\C^2 \backslash \{ 0 \}$.
\end{example}

\begin{example}
For each finite dimensional Lie algebra $\g$, there is a natural Poisson bracket on $\C[\g^*] = \mathrm{Sym} (\g)$, uniquely defined by $\{ X , Y \} = [X,Y]$ for all $X,Y \in \g$. Recall that $\mf{sl}_2 = \C \{ E,F,H \}$ with $[E,F] = H$, $[H,E] = 2 E$ and $[H,F] = - 2 F$. As in the previous example, we will calculate the symplectic leaves of $\mf{sl}_2^*$. The Hamiltonian vector fields of the generators $E,F,H$ of the polynomial ring $\C[\mf{sl}_2]$ are $X_E = H \pa_F - 2 E \pa_H$, $X_F = - H \pa_E + 2 F \pa_H$ and $X_H = 2 E \pa_E - 2 F \pa_F$. Let $a = (a_E(t),a_F(t),a_H(t))$ be an integral curve through $(p,q,r)$ for a vector field $X$.
\begin{itemize}
\item For $X_E$, $a(t) = (p,-pt^2 + rt + q, - 2pt + r)$. 
\item For $X_F$, $a(t) = (-q t^2 - rt + p,q, 2qt + r)$. 
\item For $X_H$, $a(t) = (p \exp (2 t),q \exp (- 2t), r)$. 
\end{itemize}
Thus, for all $(p,q,r) \neq (0,0,0)$, $X_E, X_F,X_F$ span a two-dimensional subspace of $T_{(p,q,r)} \mf{sl}_2^*$. Also, one can check that the expression $a_E(t) a_F(t) + \frac{1}{4} a_H(t)^2 = pq + \frac{1}{4} r^2$ is independent of $t$ for each of the three integral curves above e.g. for $X_F$ we have
$$
(-qt^2 - rt + p) q + \frac{1}{4}(2 q t + r)^2 = pq + \frac{1}{4} r^2.
$$
Therefore, for $s \neq 0$, $V(EF + \frac{1}{4} H^2 = s)$ is a smooth, two-dimensional Poisson subvariety on which the symplectic form is everywhere non-degenerate. This implies that it is a symplectic leaf.

The nullcone $\mc{N}$ is defined to be $V(EF + \frac{1}{4} H^2)$ i.e. we take $s = 0$. If we consider $U = \mathcal{N} \backslash \{ (0,0,0) \}$, then this is also smooth and the symplectic form is everywhere non-degenerate. Thus, it is certainly contained in a symplectic leaf of $\mathcal{N}$. However, the whole of $\mathcal{N}$ cannot be a leaf because it is singular. Therefore, the symplectic leaves of $\mathcal{N}$ are $U$ and $\{ (0,0,0) \}$.  
\end{example}

In the case of symplectic reflection algebras, we have:

\begin{thm}\label{thm:symplvar}
The symplectic leaves of the Poisson variety $X_{\mbf{c}}(W)$ are precisely the symplectic cores of $X_{\mbf{c}}(W)$. In particular, they are finite in number, hence the bracket $\{ - , - \}$ is algebraic.
\end{thm}

Remarkably, the representation theory of symplectic reflection algebras is ``constant along symplectic leaves'', in the following precise sense. For each $\chi \in X_{\mbf{c}}(G)$, let $\H_{\chi,\mbf{c}}(G)$ be the finite-dimensional quotient $\H_{0,\mbf{c}}(G) / \mf{m}_{\chi} \H_{0,\mbf{c}}(G)$, where $\mf{m}_{\chi}$ is the kernel of $\chi$. If $\chi \in \mc{A}_{\mbf{c}} = X_{\mbf{c}}(G)_{\sm}$ then   
$$
\H_{\chi,\mbf{c}}(G) \simeq \mathrm{Mat}_{|G|}(\C), \quad \dim \H_{\chi,\mbf{c}}(G) = |G|^2.
$$
This is not true if $\chi \in X_{\mbf{c}}(G)_{\mathrm{sing}}$. 

\begin{thm}\label{thm:constant}
Let $\chi_1, \chi_2$ be two points in $\mc{L}$, a symplectic leaf of $X_{\mbf{c}}(G)$. Then,  
$$
\H_{\chi_1,\mbf{c}}(G) \simeq \H_{\chi_2,\mbf{c}}(G). 
$$
\end{thm}

\begin{example}
Let's consider again our favorite example $W = \Z_2$. When $\mbf{c}$ is non-zero, one can check, as in the right hand side of figure \ref{fig:resdef}, that $X_{\mbf{c}}(\Z_2)$ is smooth. Therefore, it has only one symplectic leaf i.e. it is a symplectic manifold. Over each closed point of $X_{\mbf{c}}(\Z_2)$ there is exactly one simple $\H_{0,\mbf{c}}(\Z_2)$-module, which is isomorphic to $\C \Z_2$ as a $\Z_2$-module. If, on the other hand, $\mbf{c} = 0$ so that $\H_{0,0}(\Z_2) = \C[x,y] \rtimes \Z_2$, then there is one singular point and hence two symplectic leaves - the singular point and its compliment. On each closed point of the smooth locus, there is exactly one simple $\C[x,y] \rtimes \Z_2$-module, which is again isomorphic to $\C \Z_2$ as a $\Z_2$-module. However, above the singular point there are two simple, one-dimensional, modules, isomorphic to $\C e_0$ and $\C e_1$ as $\C \Z_2$-modules. See figure \ref{fig:sing}.      
\end{example}

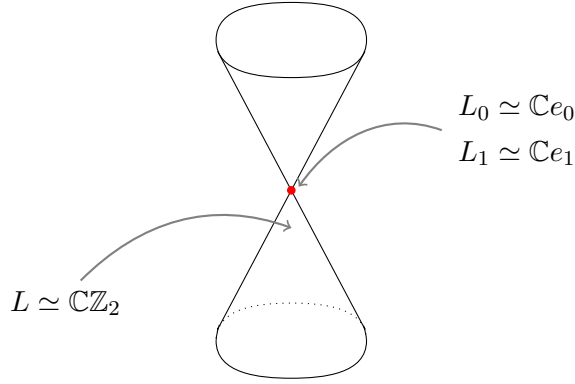
\begin{figure}\label{fig:sing}
\begin{tikzpicture}
\draw[dotted] (1,-2) to [out=90,in=0] (0,-1.5) to [out=180,in=90] (-1,-2);
\draw (-1,-2) to [out=-90,in=180] (0,-2.5) to [out=0,in=-90] (1,-2);
\draw (1,2) to [out=90,in=0] (0,2.5) to [out=180,in=90] (-1,2) to [out=-90,in=180] (0,1.5) to [out=0,in=-90] (1,2);
\draw (0.98,-1.85) -- (-0.98,1.85);
\draw (-0.98,-1.85) -- (0.98,1.85);
\draw (1,-2) to [out=90,in=110] (0.98,-1.85);
\draw (-0.98,-1.85) to [out=70,in=90] (-1,-2);
\node at (3,0.5) {$L_1 \simeq \C e_1$};
\node at (3,1.1) {$L_0 \simeq \C e_0$};

\draw [thick,gray,->] (2,0.8) to [out=160,in=55] (0.1,0.05);
\draw [thick,gray,->] (-2.8,-1.2) to [out=45,in=160] (0,-0.5);

\node at (-3,-1.5) {$L \simeq \C \Z_2$};

\draw [fill,red] (0,0) circle [radius=0.05]; 
\end{tikzpicture}
\caption{The leaves and simple modules of $\H_{0,0}(\Z_2)$.}
\end{figure}

\subsection{Restricted rational Cherednik algebras}\label{sec:restrictedRCA}

In order to be able to say more about the simple modules for $\H_{0,\mbf{c}}(G)$, e.g. to describe their possible dimensions, we restrict ourselves to considering rational Cherednik algebras. Therefore, in this subsection, we let $W$ be a complex reflection group and $\H_{0,\mbf{c}}(W)$ the associated rational Cherednik algebra, as defined in lecture one. In the case of Coxeter groups the following was proved in \cite[Proposition 4.15]{EG}, and the general case is due to \cite[Proposition 3.6]{Baby}. 

\begin{prop}
Let $\H_{0,\mbf{c}}(W)$ be a rational Cherednik algebra associated to the complex reflection group $W$.
\begin{enumerate}
\item The subalgebra $\C[\h]^W \otimes \C[\h^*]^W$ of $\H_{0,\mbf{c}}(W)$ is contained in $\ZH_{\mbf{c}}(W)$.
\item The centre $\ZH_{\mbf{c}}(W)$ of $\H_{0,\mbf{c}}(W)$ is a free $\C[\h]^W \otimes \C[\h^*]^W$-module of rank $|W|$.
\end{enumerate}
\end{prop}

The inclusion of algebras $A := \C[\mf{h}]^W \otimes \C[\mf{h}^*]^W \hookrightarrow \ZH_{\mathbf{c}}(W)$ allows us to define the \textit{restricted rational Cherednik algebra} $\overline{\H}_{\mbf{c}}(W)$ as
\begin{displaymath}
\overline{\H}_{\mbf{c}}(W) = \frac{\H_{\mathbf{c}}(W)}{A_+ \cdot H_{\mathbf{c}}(W)},
\end{displaymath}
where $A_+$ denotes the ideal in $A$ of elements with zero constant term. This algebra was originally introduced, and extensively studied, in the paper \cite{Baby}. The PBW theorem implies that 
$$
\overline{\H}_{\mbf{c}}(W) \cong \C [\h]^{co W} \otimes \C W \otimes \C [\h^*]^{coW}
$$
as vector spaces. Here 
$$
\C [\h]^{co W} = \C [\h] / \langle \C [\h]^W_+ \rangle
$$
is the \textit{coinvariant algebra}. Since $W$ is a complex reflection group, $\C [\h]^{co W}$ has dimension $|W|$ and is isomorphic to the regular representation as a $W$-module. Thus, $\dim \overline{\H}_{\mbf{c}}(W) = |W|^3$. Denote by $\Irr (W)$ a set of complete, non-isomorphic simple $W$-modules.

\begin{defn}
Let $\lambda \in \Irr (W)$. The \textit{baby Verma module} of $\overline{\H}_{\mbf{c}}(W)$, associated to $\lambda$, is 
$$
\overline{\Delta}(\lambda) := \overline{\H}_{\mbf{c}}(W) \otimes_{\C [\h^*]^{co W} \rtimes W} \lambda,
$$
where $\C [\h^*]^{co W}_+$ acts on $\lambda$ as zero. 
\end{defn}

The rational Cherednik algebra is $\Z$-graded (no such grading exists for general symplectic reflection algebras). The grading is defined by $\deg(x) = 1$, $\deg(y) = -1$ and $\deg(w) = 0$ for $x \in \h^*$, $y \in \h$ and $w \in W$. At $t = 1$, this is just the natural grading coming from the Euler operator. Since the restricted rational Cherednik algebra is a quotient of $\H_{\mbf{c}}(W)$ by an ideal generated by homogeneous elements, it is also a graded algebra. This means that the representation theory of $\overline{\H}_{\mbf{c}}(W)$ has a rich combinatorial structure and one can use some of the combinatorics to better describe the modules $L(\lambda)$. In particular, since $\C [\h^*]^{co W} \rtimes W$ is a graded subalgebra of $\overline{\H}_{\mbf{c}}(W)$, the baby Verma module $\overline{\Delta}(\lambda)$ is a graded $\overline{\H}_{\mbf{c}}(W)$-module, where $1 \o \lambda$ sits in degree zero. By studying quotients of baby Verma modules, it is possible to completely classify the simple $ \overline{\H}_{\mbf{c}}(W)$-module. 

\begin{prop}\label{prop:simplebaby}
Let $\lambda, \mu \in \Irr (W)$.
\begin{enumerate}
\item The baby Verma module $\overline{\Delta}(\lambda)$ has a simple head, $L(\lambda)$. Hence $\overline{\Delta}(\lambda)$ is indecomposable.
\item $L(\lambda)$ is isomorphic to $L(\mu)$ if and only if $\lambda \simeq \mu$.  
\item The set $\{  L(\lambda) \, | \, \lambda \in \Irr (W) \}$ is a complete set of pairwise non-isomorphic simple $\overline{\H}_{\mbf{c}}(W)$-modules.
\end{enumerate}
\end{prop}

\begin{proof}
We first recall some elementary facts from the representation theory of finite dimensional algebras. The \textit{radical} $\rad R$ of a ring $R$ is the intersection of all maximal left ideals. When $R$ is a finite dimensional algebra, $\rad R$ is also the union of all nilpotent ideals in $R$ and is itself a nilpotent ideal; see \cite[Proposition 3.1]{ARS}.

\begin{claim}\label{claim:grrad}
If $R$ is $\Z$-graded, finite dimensional, then $\rad R$ is a homogeneous ideal. 
\end{claim}

\begin{proof}
By a homogeneous ideal we mean that if $a \in I$ and $a = \sum_{i \in \Z} a_i$ is the decomposition of $a$ into homogeneous pieces, then every $a_i$ belongs to $I$. If $I$ is an ideal, let $\mathrm{hom}(I)$ denote the ideal in $R$ generated by all homogeneous parts of all elements in $I$ i.e. $\mathrm{hom}(I)$ is the smallest ideal in $R$ containing $I$. It suffices to show that if $I,J$ are ideals in $R$ such that $I J = 0$, then $\mathrm{hom}(I) \mathrm{hom}(J) = 0$. In fact, we just need to show that $\mathrm{hom}(I) J = 0$. Let the length of $a \in I$ be the number of integers $n$ such that $a_n \neq 0$. Then, an easy induction on length shows that $a_n J = 0$ for all $n$. Since $I$ is finite dimensional, this implies that $\mathrm{hom}(I) J = 0$ as required. 
\end{proof}

It is a classical (but difficult) result by G. Bergman that the hypothesis that $R$ is finite dimensional in the above claim is not necessary.  

If $M$ is a finitely generated $\overline{\H}_{\mbf{c}}(W)$-module, then the radical of $M$ is defined to be the intersection of all maximal submodules of $M$. It is known, e.g. \cite[Section 1.3]{ARS}, that the radical of $M$ equals $(\rad \overline{\H}_{\mbf{c}}(W)) M$ and that the quotient $M / \rad M$ (which is defined to be the \textit{head} of $M$) is semi-simple. Therefore, Claim \ref{claim:grrad} implies that if $M$ is graded, then its radical is a graded submodule and equals the intersection of all maximal graded submodules of $M$. In particular, if $M$ contains a unique maximal graded submodule, then the head of $M$ is a simple, graded modules. 

Thus, we need to show that $\overline{\Delta}(\lambda)$ has a unique maximal graded submodule. The way we have graded $\overline{\Delta}(\lambda)$, we have $\overline{\Delta}(\lambda)_n = 0$ for $n < 0$ and $\overline{\Delta}(\lambda)_0 = 1 \o \lambda$. Let $M$ be the sum of all graded submodules $M'$ of $\overline{\Delta}(\lambda)$ such that $M_0' = 0$. Clearly, $M \subset \rad \overline{\Delta}(\lambda)$. If this is a proper inclusion then $(\rad \overline{\Delta}(\lambda))_0 \neq 0$. But $\lambda$ is an irreducible $W$-module, hence $(\rad \overline{\Delta}(\lambda))_0 = \lambda$. Since $\overline{\Delta}(\lambda)$ is generated by $\lambda$ this implies that $\rad \overline{\Delta}(\lambda) = \overline{\Delta}(\lambda)$ - a contradiction. Thus, $M = \rad \overline{\Delta}(\lambda)$. By the same argument, it is clear that $\overline{\Delta}(\lambda) / M$ is simple. This proves part (1).

Part (2): For each $\overline{\H}_{\mbf{c}}(W)$-module $L$, we define $\mathrm{Sing}(L) = \{ l \in L \ | \ \mf{h} \cdot l = 0 \}$. Notice that  $\mathrm{Sing}(L)$ is a (graded if $L$ is graded) $W$-submodule and if $\lambda \subset \mathrm{Sing}(L)$ then there exists a unique morphism $\overline{\Delta}(\lambda) \rightarrow L$ extending the inclusion map $\lambda \hookrightarrow \mathrm{Sing}(L)$. Since $L(\lambda)$ is simple, $\mathrm{Sing}(L(\lambda)) = L(\lambda)_0 = \lambda$. Hence $L(\lambda) \simeq L(\mu)$ implies $\lambda \simeq \mu$. 

Part (3): It suffices to show that if $L$ is a simple $\overline{\H}_{\mbf{c}}(W)$-module, then $L \simeq L(\lambda)$ for some $\lambda \in \Irr (W)$. It is well-known that the simple modules of a graded finite dimensional algebra can be equipped with a (non-unique) grading. So we may assume that $L$ is graded. Since $L$ is finite dimensional and $\mf{h} \cdot L_n \subset L_{n-1}$, $\mathrm{Sing}(L) \neq 0$. Choose some $\lambda \subset \mathrm{Sing}(L)$. Then there is a non-zero map $\overline{\Delta}(\lambda) \rightarrow L$. Hence $L \simeq L(\lambda)$.
\end{proof}   

\subsection{The Calogero-Moser partition}\label{sec:defnCalogeroMoser}

Since the algebra $\overline{\H}_{\mbf{c}}(W)$ is finite dimensional, it will decompose into a direct sum of \textit{blocks}:
$$
\overline{\H}_{\mbf{c}}(W) = \bigoplus_{i = 1}^k B_i,
$$
where $B_i$ is a block if it is indecomposable as an algebra. If $b_i$ is the identity element of $B_i$ then the identity element $1$ of $\overline{\H}_{\mbf{c}}(W)$ is the sum $1 = b_1 + \ds + b_k$ of the $b_i$. For each simple $\overline{\H}_{\mbf{c}}(W)$-module $L$, there exists a unique $i$ such that $b_i \cdot L \neq 0$. In this case we say that $L$ \textit{belongs to the block} $B_i$. By Proposition \ref{prop:simplebaby}, we can (and will) identify $\textrm{Irr} \,  \overline{\H}_{\mbf{c}}(W)$ with $\textrm{Irr} \, (W)$. We define the \textit{Calogero-Moser partition} of $\textrm{Irr} \, (W)$ to be the set of equivalence classes of $\textrm{Irr} \, (W)$  under the equivalence relation $\lambda \sim \mu$ if and only if $L(\lambda)$ and $L(\mu)$ belong to the same block.\\

To aid intuition it is a good idea to have a geometric intepretation of the Calogero-Moser partition. The image of the natural map $\ZH_{\mbf{c}} / A_+ \cdot \ZH_{\mbf{c}} \rightarrow \overline{\H}_{\mbf{c}}(W)$ is clearly contained in the centre of $\overline{\H}_{\mbf{c}}(W)$. In general it does not equal the centre of $\overline{\H}_{\mbf{c}}(W)$ (though one can use the Satake isomorphism to show that it is injective). However, it is a consequence of a theorem by M\"uller, see \cite[Corollary 2.7]{Ramifications}, that the primitive central idempotents of $\overline{\H}_{\mbf{c}}(W)$ (the $b_i$'s) are precisely the images of the primitive idempotents of $\ZH_{\mbf{c}} / A_+ \cdot \ZH_{\mbf{c}}$. Geometrically, this can be interpreted as follows. The inclusion $A \hookrightarrow \ZH_{\mbf{c}}(W)$ defines a finite, surjective morphism 
$$
\Upsilon \ : \ X_{\mbf{c}}(G) \longrightarrow \h / W \times \h^* /W
$$
where $\h / W \times \h^* /W = \Spec A$. M\"uller's theorem is saying that the natural map $\LW \rightarrow \Upsilon^{-1}(0), \, \lambda \mapsto \textrm{Supp} \, (L(\lambda)) = \chi_{L(\lambda)}$, factors through the Calogero-Moser partition (here $\Upsilon^{-1}(0)$ is considered as the set theoretic pull-back):
\begin{displaymath}
\vcenter{
\xymatrix{
\LW \ar@{->}[dr] \ar@{->>}[d] &  \\
CM(W) \ar@{.>}[r]^\sim & \Upsilon^{-1}(0) }
}
\end{displaymath}

Using this fact, one can show that the geometry of $X_{\mbf{c}}(W)$ is related to Calogero-Moser partitions in the following way. 

\begin{thm}\label{thm:Mo}
The following are equivalent:
\begin{itemize}
\item The generalized Calogero-Moser space $X_{\mbf{c}}(W)$ is smooth. 
\item The Calogero-Moser partition of $\Irr (W')$ is trivial for all parabolic subgroup $W'$ of $W$.
\end{itemize} 
\end{thm}

Here $W'$ is a \textit{parabolic} subgroup of $W$ if there exists some $v \in \h$ such that $W' = \mathrm{Stab}_W(v)$. It is a remarkable theorem by Steinberg \cite[Theorem 1.5]{SteinbergDifferential} that all parabolic subgroups of $W$ are again complex reflection groups. 

\subsection{Graded characters} 

In this section we assume that $W$ is a complex reflection group. In some cases it is possible to use the $\Z$-grading on $\overline{\H}_{\mbf{c}}(W)$ to calculate the graded character of the simple modules $L(\lambda)$. 

\begin{lem}\label{lem:singblock}
The element $\lambda \in \Irr (W)$ defines a block $\{ \lambda \}$ of the restricted rational Cherednik algebra $\overline{\H}_{\mbf{c}}(W)$ if and only if $\dim L(\lambda) = |W|$ if and only if $L(\lambda) \simeq \C W$ as a $W$-module. 
\end{lem}

\begin{proof}
Recall that M\"uller's Theorem from section \ref{sec:defnCalogeroMoser} says that the primitive central idempotents of $\overline{\H}_{\mbf{c}}(W)$ (the $b_i$'s) are the images of the primitive idempotents of $\ZH_{\mbf{c}} / A_+ \cdot \ZH_{\mbf{c}}$ under the natural map $\ZH_{\mbf{c}} / A_+ \cdot \ZH_{\mbf{c}} \rightarrow \overline{\H}_{\mbf{c}}(W)$. The primitive idempotents of $R := \ZH_{\mbf{c}} / A_+ \cdot \ZH_{\mbf{c}}$ are in bijection with the maximal ideals in this ring: given a maximal ideal $\mf{m}$ in $R$ there is a unique primitive idempotent $b$ whose image in $R/\mf{m}$ is non-zero. Under this correspondence, the simple module belonging to the block of $\overline{\H}_{\mbf{c}}(W)$ corresponding to $b$ are precisely those simple modules supported at $\mf{m}$. Thus, the block $b$ has just one simple module if and only if there is a unique simple module of $\H_{\mbf{c}}(W)$ supported at $\mf{m}$. But, by the Artin-Procesi Theorem, remark \ref{rem:ArtinProcesi} (3), there is a unique simple module supported at $\mf{m}$ if and only if $\mf{m}$ is in the Azumaya locus of $X_{\mbf{c}}(W)$. As noted in the proof of Corollary \ref{prop:singulariffsmall}, the smooth locus of $X_{\mbf{c}}(W)$ equals the Azumaya locus. Thus, to summarize, the module $L(\lambda)$ is on its own in a block if and only if its support is contained in the smooth locus. Then, the conclusions of the lemma follow from Theorem \ref{thm:genericisregularrep} and Corollary \ref{prop:singulariffsmall}.  
\end{proof}

When $\lambda$ satisfies the conditions of Lemma \ref{lem:singblock}, we say that $\lambda$ is in a block on its own. Recall that $\C[\h]^{co W}$, the coinvariant ring of $W$, is defined to be the quotient $\C[\h]/ \langle \C[\h]^W_+ \rangle$. It is a graded $W$-module. Therefore, we can define the \textit{fake degree} of $\lambda \in \Irr (W)$ to be the polynomial
$$
f_{\lambda}(t) = \sum_{i \in Z} [\C[\h]^{co W}_i : \lambda] t^i
$$
where $\C[\h]^{co W}_i$ is the part of $\C[\h]^{co W}$ of degree $i$ and $[\C[\h]^{co W}_i : \lambda]$ is the multiplicity of $\lambda$ in $\C[\h]^{co W}_i$. For each $\lambda \in \Irr (W)$, we define $b_{\lambda}$ to be the degree of smallest monomial appearing in $f_{\lambda}(t)$ e.g. if $f_{\lambda}(t) = 2 t^4 - t^6 + $ higher terms, then $b_{\lambda} = 4$. Given a finite dimensional, graded vector space $M$, the Poincar\'e polynomial of $M$ is defined to be 
$$
P(M,t) = \sum_{i \in \Z} \dim M_i t^i.
$$

\begin{lem}\label{lem:trivlem}
Assume that $\lambda$ is in a block on its own. Then the Poincar\'e polynomial of $L(\lambda)$ is given by 
$$
P(L(\lambda),t) = \frac{(\dim \lambda) t^{b_{\lambda^*}} P(\C[\h]^{co W},t)}{f_{\lambda^*}(t)}.
$$
\end{lem}

\begin{proof}
By Proposition \ref{prop:simplebaby}, the baby Verma module $\overline{\Delta}(\lambda)$ is indecomposable. This implies that all its composition factors belong to the same block. We know that $L(\lambda)$ is one of these composition factors. But, by assumption, $L(\lambda)$ is on its own in a block. Therefore every composition factor is isomorphic to $L(\lambda)$. So let's try and calculate the graded multiplicities of $L(\lambda)$ in $\overline{\Delta}(\lambda)$. In the graded Grothendieck group of $\overline{\H}_{\mbf{c}}(W)$, we must have 
\beq{eq:deltaL}
[\overline{\Delta}(\lambda)] = [L(\lambda)][i_1] + \cdots + [L(\lambda)][i_{\ell}]
\eeq
where $[L(\lambda)][k]$ denotes the class of $L(\lambda)$, shifted in degree by $k$. By Lemma \ref{lem:singblock}, $\dim L(\lambda) = |W|$ and it is easy to see that $\dim \overline{\Delta}(\lambda) = |W| \dim \lambda$. Thus, $\ell = \dim \lambda$. Since $L(\lambda)$ is a graded quotient of $\overline{\Delta} (\lambda)$ we may also assume that $i_1 = 0$. Recall that $\C[\h]^{co W}$ is isomorphic to the regular representation as a $W$-module. Therefore, the fact that $[\mu \o \lambda : \mathrm{triv}] \neq 0$ if and only if $\mu \simeq \lambda^*$ (in which case it is one) implies that the multiplicity space of the trivial representation in $\overline{\Delta} (\lambda)$ is $(\dim \lambda)$-dimensional. The Poincar\'e polynomial of this multiplicity space is precisely the fake polynomial $f_{\lambda^*}(t)$. On the other hand, the trivial representation only occurs once in $L(\lambda)$ since it is isomorphic to the regular representation. Therefore, comparing graded multiplicities of the trivial representation on both sides of equation (\ref{eq:deltaL}) implies that, up to a shift, $t^{i_1} + \cdots + t^{i_{\ell}} = f_{\lambda^*}(t)$. What is the shift? Well, the lowest degree\footnote{It is not complete obvious that $0$ is the lowest degree in $t^{i_1} + \cdots + t^{i_{\ell}}$, see \cite[Lemma 4.4]{Baby}.} occurring in $t^{i_1} + \cdots + t^{i_{\ell}}$ is $i_1 =0$. But the lowest degree in $f_{\lambda}(t)$ is $b_{\lambda^*}$. Thus, $t^{i_1} + \cdots + t^{i_{\ell}} = t^{-b_{\lambda^*}}f_{\lambda^*}(t)$. This implies that  
$$
P(L(\lambda),t) = \frac{t^{b_{\lambda^*}} P(\overline{\Delta}(\lambda),t)}{f_{\lambda^*}(t)}.
$$
Clearly, $P(\overline{\Delta}(\lambda),t) = (\dim \lambda) P(\C[\h]^{co W},t)$. 
\end{proof}

The module $L(\lambda)$ is finite dimensional. Therefore $P(L(\lambda),t)$ is a Laurent polynomial. However, one can often find representations $\lambda$ for which $f_{\lambda^*}(t)$ does not divide $P(\C[\h]^{co W},t)$. In such cases the above calculation show that $\lambda$ is \textit{never} in a block on its own. 

\begin{exercise}\label{ex:G2}
The character table of the Weyl group $G_2$ is given by 
\begin{displaymath}
\begin{array}{c|ccccccc}
\textrm{Class} & 1 & 2 & 3 & 4 & 5 & 6\\
\textrm{Size} & 1 & 1 & 3 & 3 & 2 & 2 \\
\textrm{Order} & 1 & 2 & 2 & 2 & 3 & 6\\
\hline
T & 1 & 1 & 1 & 1 & 1 & 1 \\
S & 1 & 1 & -1 & -1 & 1 & 1 \\
V_1 & 1 & -1 & 1 & -1 & 1 & -1 \\ 
V_2 & 1 & -1 & -1 & 1 & 1 & -1 \\  
\mathfrak{h}_1 & 2 & 2 & 0 & 0 & -1 & -1 \\  
\mathfrak{h}_2 & 2 & 2 & 0 & 0 & -1 & 1  
\end{array}
\end{displaymath}
The fake polynomials are 
$$
f_{T}(t) = 1, \quad f_S(t) = t^6, \quad f_{V_1}(t) = t^3, \quad f_{V_2}(t) = t^3, \quad f_{\mathfrak{h}_1}(t) = t^2 + t^4, \quad f_{\mathfrak{h}_2}(t) = t + t^5.
$$
Is $X_{\mbf{c}}(G_2)$ ever smooth?
\end{exercise}

\begin{exercise}\label{ex:LMS}
(Harder) For this exercise you'll need to have GAP 3, together with the package ``CHEVIE'' installed. Using the code\footnote{Available from \url{http://www.maths.gla.ac.uk/~gbellamy/MSRI.html}.} \verb+fake.gap+, show that there is (at most one) exceptional complex reflection group $W$ for which the space $X_{\mbf{c}}(W)$ can ever hope to be smooth. Which exceptional group is this? For help with this exercise, read \cite{Singular}. 
\end{exercise}

\subsection{Symplectic resolutions}

Now we return to the original question posed at the start of lecture one: How singular is the space $V/G$? The usual way of answering this question is to look at resolutions of singularities of $V/G$. 

\begin{defn}
A \textit{(projective) resolution of singularities} is a birational morphism $\pi : Y \rightarrow V/G$ from a smooth variety $Y$, projective over $V/G$, such that the restriction of $\pi$ to $\pi^{-1}((V/G)_{\textrm{sm}})$ is an isomorphism. 
\end{defn}

If $V_{\reg}$ is the open subset of $V$ on which $G$ acts freely, then $V_{\reg} / G \subset V/G$ is the smooth locus and it inherits a symplectic structure from $V$ i.e. $V_{\reg} / G$ is a symplectic manifold. 

\begin{defn}
A projective resolution of singularities $\pi : Y \rightarrow V/G$ is said to be \textit{symplectic} if $Y$ is a symplectic manifold and the restriction of $\pi$ to $\pi^{-1}((V/G)_{\textrm{sm}})$ is an isomorphism of symplectic manifolds. 
\end{defn}

The existence of a symplectic resolution for $V/G$ is a very strong condition and implies that the map $\pi$ has some very good properties e.g. $\pi$ is \textit{semi-small}. Therefore, as one might expect, symplectic resolutions exist only for very special groups. 

\begin{thm}\label{thm:sympresdef}
Let $(V,\omega,G)$ be an irreducible symplectic reflection group. 
\begin{itemize}
\item The quotient singularity $V/G$ admits a symplectic resolution if and only it admits a smooth Poisson deformation. 
\item The quotient singularity $V/G$ admits a smooth Poisson deformation if and only if $X_{\mbf{c}}(G)$ is smooth for generic parameters $\mbf{c}$.
\end{itemize}
\end{thm}

The irreducible symplectic reflection group have been classified by Cohen, \cite{CohenQuaternionic}. Using the above theorem, work of several people (Verbitsky \cite{Verbitsky}, Ginzburg-Kaledin \cite{GK}, Gordon \cite{Baby}, Bellamy \cite{Singular}, Bellamy-Schedler \cite{smoothsra}) means that the classification of quotient singularities admitting symplectic resolutions is (almost) complete.

\begin{example}
Let $G \subset SL_2(\C)$ be a finite group. Since $\dim \C^2 / G = 2$, there is a \textit{minimal resolution} $\widetilde{\C^2} / G$ of $\C^2 / G$ through which all other resolutions factor. This resolution can be explicitly constructed as a series of blowups. Moreover, $\widetilde{\C^2} / G$ is a symplectic manifold and hence provides a symplectic resolution of $\C^2 / G$. The corresponding generalized Calogero-Moser space $X_{\mbf{c}}(G)$ is smooth for generic parameters $\mbf{c}$. The corresponding symplectic reflection algebras are closely related to deformed preprojective algebras, \cite{CrawleyBoeveyHolland}. 
\end{example}

\begin{figure}
\includegraphics[scale = 0.2]{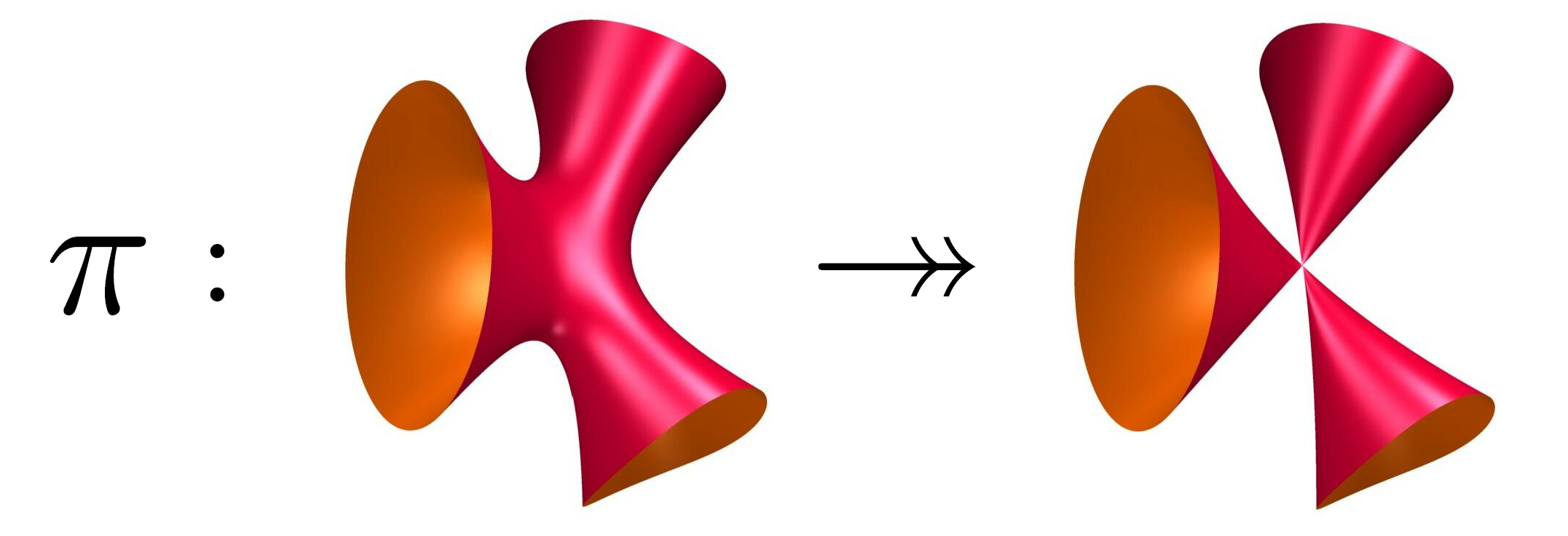}
\caption{A representation of the resolution of the $D_4$ Kleinian singularity.}
\end{figure}



\subsection{Additional remark}

\begin{itemize}
\item Theorem \ref{thm:genericisregularrep} is proven in \cite[Theorem 1.7]{EG}.

\item The fact that $X_{\mbf{c}}(W)$ has finitely many symplectic leaves, Theorem \ref{thm:symplvar}, is \cite[Theorem 7.8]{PoissonOrders}. 

\item The beautiful result that the representation theory of symplectic reflection algebras is constant along leaves, Theorem \ref{thm:constant}, is due to Brown and Gordon, \cite[Theorem 4.2]{PoissonOrders}.

\item Proposition \ref{prop:simplebaby} is Proposition 4.3 of \cite{Baby}. It is based on the general results of \cite{HN}, applied to the restricted rational Cherednik algebra.

\item The Calogero-Moser partition was first defined in \cite{GordonMartinoCM}.

\item Theorem \ref{thm:Mo} is stated for rational Cherednik algebras at $t = 1$ in positive characteristic in \cite[Theorem 1.3]{MoPositiveChar}. However, the proof given there applies word for word to rational Cherednik algebras at $t = 0$ in characteristic zero. 

\item Theorem \ref{thm:sympresdef} follows from the results in \cite{GK} and \cite{Namikawa}.

\end{itemize}

\newpage

\section{Solutions to exercises}

We include (partial) solutions to some of the exercises scattered throughout the course notes. 

\subsection{Lecture \ref{sec:one}}

\begin{proof}[Solution to exercise \ref{ex:centre}]
Let $z = \sum_{g \in G} f_g \cdot g \in Z(\C[V] \rtimes G)$. Choose some $g \neq 1$. Since $G \subset GL(V)$, there exists $h \in \C[V]$ such that $g(h) \neq h$. Then 
$$
[h,z] = \sum_{g \in G} f_g(h - g(h)) \cdot g = 0
$$
implies that $f_g = 0$ for all $g \neq 1$. Therefore $Z(\C[V] \rtimes G) \subset \C[V]$. But it is clear that $Z(\C[V] \rtimes G) \cap \C[V] \subseteq \C[V]^G$. On the other hand one can easily see that $\C[V]^G \subset Z(\C[V] \rtimes G)$.
\end{proof}

\begin{proof}[Solution to exercise \ref{ex:PBW}]
Taking away the relation $[y,x] = 1$ from $(y - s) x = xy + s$ gives $sx = -s$ and hence $x = -1$. Then $[y,x] = 1$ implies that $1 = 0$.  
\end{proof}

\begin{proof}[Solution to exercise \ref{ex:ex1.3}]
At $t = 0$, $\dd_t(\h_{\reg}) \rtimes W = \C[\h_{\reg} \times \h^*] \rtimes W$. Therefore the image of $\mbf{e} \H_{0,\mbf{c}}(W) \mbf{e}$ is contained in $\mbf{e} (\C[\h_{\reg} \times \h^*] \rtimes W) \mbf{e}\simeq \C[\h_{\reg} \times \h^*]^W$. This is a commutative ring. 
\end{proof}

\begin{proof}[Solution to exercise \ref{ex:ex1.6}]
Part (1): Since the isomorphism $\h^* \stackrel{\sim}{\rightarrow} \h$, $x \mapsto \tilde{x} = (x, - )$ is $W$-equivariant, the only thing to check is that the commutation relation 
$$
[y,x] = t x(y) - \sum_{s \in \mathcal{S}} \mathbf{c}(s) \alpha_s(y) x(\alpha_s^\vee) s, \quad \forall \ y \in \h, \ x \in \h^*
$$
still holds after applying $\widetilde{(-)}$ everywhere. This follows from two trivial observations. First, 
$$
x(y) = \tilde{y}(\tilde{x}) = (x,\tilde{y}) = (y, \tilde{x}),
$$
and secondly, possibly after some rescaling, we have $(\alpha_s,\alpha_s) = (\alpha_s^{\vee},\alpha_s^{\vee}) = 2$ and $\tilde{\alpha}_s = \alpha_s^{\vee}$, $\tilde{\alpha}_s^{\vee} = \alpha_s$. For part (2), just follow the hint.  
\end{proof}

\subsection{Lecture \ref{sec:two}}

\begin{proof}[Solution to exercise \ref{ex:example1}]
Part (1): Choose some $0 \neq x \in \h^* \subset \C[\h]$ and take $M = \H_{\mbf{c}}(W) / \H_{\mbf{c}}(W) \cdot (\eu - x)$. This module cannot be a direct sum of its generalized eigenspaces.

Part (2): Let $L$ be a finite-dimensional $\H_{\mbf{c}}(W)$-module. Then, it is a direct sum of its generalized $\eu$-eigenspaces because $\eu \in \End_{\C}(L)$ and a finite dimensional vector space decomposes as a direct sum of generalized eigenspaces under the action of any linear operator. If $l \in L_a$ for some $a \in \C$, then the relation $[\eu,y] = -y$ implies that $y \cdot l \in L_{a-1}$. Hence $y_1 \cdots y_k \cdot l \in L_{a - k}$. But $L$ is finite dimensional which implies that $L_{a - k} = 0$ for $k \gg 0$. Hence $\h$ acts locally nilpotently on $L$.  
\end{proof}

\begin{proof}[Solution to exercise \ref{ex:charpoly}]
Since $M$ is finitely generate as a $\C[\h]$-module, we may choose a finite dimensional, $\eu$ and $\h$-stable subspace $M_0$ of $M$ such that $M_0$ generates $M$ as a $\C[\h]$-module. Therefore, there is a surjective map of $\C[\h]$-modules $\C[\h] \o_{\C} M_0 \rightarrow M$, $a \o m \mapsto am$. This can be made into a morphism of $\C[\eu]$-modules by defining $\eu \cdot (a \o m) = [\eu,a] \o m + a \o \eu \cdot m$. If $f_0(t) \in \Z[x^a \ | \ a \in \C]$ is the character of $M_0$ then the character of $\C[\h] \o M_0$ is $\frac{1}{(1 - t)^n} f_0(t) \in \bigoplus_{a \in \C} t^a \Z[[t]]$ and hence $\ch (M) \in \bigoplus_{a \in \C} t^a \Z[[t]]$ too. 
\end{proof}

\begin{proof}[Solution to exercise \ref{ex:decompO}]
Let $M \in \mc{O}$. We can decompose $M$ as a $\C[\eu]$-module as 
$$
M = \bigoplus_{\bar{a} \in \C / \Z} M^{\bar{a}}
$$
where $M^{\bar{a}} = \bigoplus_{b \in \bar{a}} M_b$. It suffices to show that each $M^{\bar{a}}$ is a $\H_{\mbf{c}}(W)$-submodule of $M$. But if $x \in \h^*$ and $m \in M_b$ then $x \cdot m \in M_{b+1}$ and $b \in \bar{a}$ iff $b + 1 \in \bar{a}$. A similar argument applies to $y \in \h^*$ and $w \in W$. Thus $M^{\bar{a}}$ is a $\H_{\mbf{c}}(W)$-submodule of $M$. 
\end{proof}

\begin{proof}[Solution to exercise \ref{ex:Oss2}]
Exercise \ref{ex:decompO} implies that 
$$
\mc{O} = \bigoplus_{\lambda \in \Irr (W)} \mc{O}^{\bar{\mbf{c}}_{\lambda}}
$$
with $\Delta(\lambda) \in \mc{O}^{\bar{\mbf{c}}_{\lambda}}$. Then exercise $3.13$ implies that $\Delta(\lambda) = L(\lambda)$ and hence category $\mc{O}$ is semi-simple. 
\end{proof}

\begin{proof}[Solution to exercise \ref{ex:O21}]
(1) If $\mbf{c} \notin \frac{1}{2} + \Z$ then category $\mc{O}$ is semi-simple. The commutation relation of $x$ and $y$ implies that 
$$
[y,x^n] = \left\{ \begin{array}{ll}
n x^{n-1} - 2 \mbf{c} x^{n-1} s & \textrm{ $n$ odd }\\
n x^{n-1} & \textrm{ $n$ even } 
\end{array} \right. 
$$
This implies that 
$$
L(\rho_0) = \frac{\C[x] \o \rho_0}{x^{2 m + 1} \C[x] \o \rho_0}, \quad L(\rho_1) = \Delta(\rho_1)
$$
when $\mbf{c} = \frac{1}{2} + m$ for some $m \in \Z_{\ge 0}$. Similarly,  
$$
L(\rho_0) = \Delta(\rho_0), \quad L(\rho_1) = \frac{\C[x] \o \rho_1}{x^{2 m + 1} \C[x] \o \rho_1}
$$
when $\mbf{c} = \frac{-1}{2} - m$ for some $m \in \Z_{\ge 0}$.

For part (2), we have $\eu = xy - \mbf{c} s$. Therefore $\mbf{c}_0 = - \mbf{c}$ and $\mbf{c}_1 = \mbf{c}$. Thus, $\rho_0 \le_{\mbf{c}} \rho_1$ if and only if $2 \mbf{c} \in \Z_{\ge 0}$. Similarly, $\rho_1 \le_{\mbf{c}} \rho_0$ if and only if $2 \mbf{c} \in \Z_{\le 0}$. For all other $\mbf{c}$, $\rho_0$ and $\rho_1$ are incomparable.
\end{proof}

\begin{proof}[Solution to exercise \ref{ex:sl2}]
Both these steps are direct calculations. You should get $\alpha = \frac{\mbf{c}}{2}(2 - \mbf{c})$. 
\end{proof}

\begin{proof}[Solution to exercise \ref{ex:GDT}]
Part (1): By the proof of Corollary $1.17$, we know that $\mbf{c}$ is aspherical if and only if 
$$
I := \H_{\mbf{c}}(W) \cdot \mbf{e} \cdot \H_{\mbf{c}}(W)
$$
is a proper two-sided ideal of $\H_{\mbf{c}}(W)$. If this is the case then there exists some primitive ideal $J$ such that $I \subset J$. Hence Ginzburg's result implies that there is a simple module $L(\lambda)$ in category $\mc{O}$ such that $I \cdot L(\lambda) = 0$. But clearly this happens if and only if $\mbf{e} \cdot L(\lambda) = 0$. 

Part (2): The only aspherical value for $\Z_2$ is $\mbf{c} = -\frac{1}{2}$. 
\end{proof}  

\subsection{Lecture \ref{sec:three}}

\begin{proof}[Solution to exercise \ref{ex:66311}]
The Young diagram with residues is 
$$
\Yboxdim18pt
\young(0:::::,1:::::,230:::,301230,012301)
$$
Then $F_2 | \lambda \rangle = q | (6,6,3,2,1) \rangle$, $K_1 | \lambda \rangle = | \lambda \rangle$ and 
$$
E_4 | \lambda \rangle = q^{-2} | (6,5,3,1,1) \rangle  + q^{-1} | (6,6,2,1,1) \rangle  + | (6,6,3,1) \rangle.
$$
\end{proof}

\begin{proof}[Solution to exercise \ref{ex:part}]
By adding an infinite number of zeros to the end of $\lambda \in \mc{P}$, we may consider it as an infinite sequence $(\lambda_1,\lambda_2, \ds)$ with $\lambda_i \ge \lambda_{i+1}$ and $\lambda_N = 0$ for all $i$ and all $N \gg 0$. Define $I(\lambda)$ by $i_k = \lambda_k + 1 - k$. It is clear that this rule defines a bijection with the required property. 
\end{proof}

\begin{proof}[Solution to exercise \ref{ex:r2}]
We have 
$$
\mc{G}([4]) = [4] + q [3,1] + +q [2,1,1] + q^2 [1,1,1,1], \quad \mc{G}([3,1]) = [3,1] + q [2,2] + q^2 [2,1,1],
$$
$$
\mc{G}([2,2]) = [2,2] + q [2,1,1], \quad \mc{G}([2,1,1]) = [2,1,1] + q [1,1,1,1], \quad \mc{G}([1,1,1,1]) = [1,1,1,1].
$$
and
$$
\mc{G}([5]) = [5] + q [3,1,1] + q^2 [1,1,1,1,1], \quad \mc{G}([4,1]) = [4,1] + q [2,1,1,1],
$$
$$
\mc{G}([3,2]) = [3,2] + q [3,1,1] + q^2 [2,2,1], \quad \mc{G}([2,2,1]) = [2,2,1],
$$
$$
\mc{G}([3,1,1]) = [3,1,1] + q [2,2,1] + q [1,1,1,1,1], \quad \quad \mc{G}([2,1,1,1]) = [2,1,1,1], 
$$
$$
\mc{G}([1,1,1,1,1]) = [1,1,1,1,1].
$$
\end{proof}

\begin{proof}[Solution to exercise \ref{ex:chars}]
\begin{enumerate}
\item In this case, the numbers $e_{\lambda,\mu}(1)$, $\mbf{c}_{\lambda}$ and $\dim \lambda$ are: 
\begin{displaymath}
\begin{array}{c|ccccccc}
\mu \backslash \lambda & (5) & (4,1) & (3,2) & (3,1,1) & (2,2,1) & (2,1,1,1) & (1,1,1,1,1) \\
\hline
{(5)} & 1 & 0 & 0 & 0 & 0 & 0 & 0 \\
{(4,1)} & 0 & 1 & 0 & 0 & 0 &  0 & 0 \\
{(3,2)} & 0 & 0 & 1 & 0 & 0 & 0 & 0 \\
{(3,1,1)} & -1 & 0 & -1 & 1& 0 & 0 & 0 \\ 
{(2,2,1)} & 1 & 0 & 0 & -1 & 1 & 0 & 0 \\ 
{(2,1,1,1)} & 0 & -1 & 0 & 0 & 0 & 1 & 0 \\ 
{(1,1,1,1,1)} & 0 & 0 & 1 & -1 & 0 & 0 & 1 \\  
\mbf{c}_{\lambda} & -\frac{65}{2} & -15 & \frac{-9}{2} & 0 & \frac{19}{2} & 20 & \frac{75}{2}  \\
\dim \lambda & 1 & 4 & 5 & 6 & 5 & 4 & 1 
\end{array}
\end{displaymath} 
If we define $\ch_{\lambda}(t) := (1 - t)^5 \cdot \ch(L(\lambda))$, then 
$$
\ch_{(5)}(t) = t^{-\frac{65}{2}} - 6t + 5 t^{\frac{19}{2}}, \quad  \ch_{(4,1)}(t) = 4t^{-15} - 4 t^{20}, 
$$
$$
\ch_{(3,2)}(t) = 5t^{-\frac{9}{2}} - 6t + t^{\frac{75}{2}}, \quad \ch_{(3,1,1)}(t) = 6t - 5 t^{\frac{19}{2}} - t^{\frac{75}{2}},
$$ 
$$
\ch_{(2,2,1)}(t) = 5 t^{\frac{19}{2}}, \quad \ch_{(2,1,1,1)}(t) = 4t^{20}, \quad \ch_{(1,1,1,1,1)}(t) = t^{\frac{75}{2}}.
$$

\item -

\item It suffices to calculate the $2$ (resp. $3$ and $5$) cores of the partitions of $5$. For $r = 2$ we get 
$$
\{ [5], [3,2], [2,2,1],[1,1,1,1,1] \}, \quad \{ [4,1],[2,1,1,1] \},
$$
where the $2$-cores are $[1]$ and $[2,1]$ respectively. For $r = 3$ we get 
$$
\{ [5], [2,2,1],[2,1,1,1] \}, \quad \{ [4,1],[3,2],[1,1,1,1,1] \},
$$
where the $3$-cores are $[2]$ and $[1,1]$ respectively. For $r = 5$ we get 
$$
\{ [5], [4,1], [2,1,1,1],[1,1,1,1,1] \}, \quad \{ [3,2] \}, \quad \{ [2,2,1] \},
$$
where the $5$-cores are $\emptyset$, $[3,2]$ and $[2,2,1]$ respectively (for an explaination of what is going on in this final example read \cite{niteezdimreps}). 
\end{enumerate}
\end{proof}

\begin{proof}[Solution to exercise \ref{ex:s4Hard}]
The argument is essentially identical to that for $\ch(L(\lambda))$. We use the polynomials $e_{\lambda,\mu}(q)$ to express the character $\ch_W(L(\lambda))$ in terms of the character of the standard modules $\Delta(\mu)$. Then, we just need to calculate $\ch_W(\Delta(\lambda))$. Note that 1) $\C[\h] \o \lambda$ is a graded $W$-module 2) the $\eu$-character of $\Delta(\lambda)$ is simply the graded character of $\C[\h] \o \lambda$ multiplied by $t^{\mbf{c}_{\lambda}}$. Hence, it suffices to describe $\C[\h] \o \lambda$ is a graded $W$-module as a graded $W$-module. This factors as 
$$
\C[\h] \o \lambda = \C[\h]^W \o \C[\h]^{co W} \o \lambda
$$
and hence
$$
\ch_W(\C[\h] \o \lambda) = \frac{1}{\prod_{i = 1}^n (1 - t^i)} \ch_W(\C[\h]^{co W} \o \lambda).
$$
We have
$$
\ch_W(\C[\h]^{co W} \o \lambda) = \sum_{\mu \in \Irr(W)} \left( \sum_{i \in \Z} [\C[\h]^{co W}_i \o \lambda : \mu] t^i \right) [\mu] = \sum_{\mu \in \Irr (W)} f_{\lambda,\mu}(t) [\mu].
$$
In fact, for $W = \s_4$, one can explicitly calculate the generalized fake polynomials. 
\begin{displaymath}
\begin{array}{c|ccc}
\mu \backslash \lambda & (4) & (3,1) & (2,2) \\
\hline
{(4)} & 1 & t + t^2 + t^3 & t^2 + t^4 \\
{(3,1)} & t + t^2 + t^3 & 1 + t + 2 t^2 + 2 t^3 + 2 t^4 + t^5 & t + t^2 + 2 t^3 + t^4 + t^5 \\
{(2,2)} & t^2 + t^4 & t + t^2 + 2 t^3 + t^4 + t^5 & t + 2t^3 + t^5  \\
{(2,1,1)} & t^3 + t^4 + t^5 & t + 2 t^2 + 2 t^3 + 2 t^4 + t^5 + t^6 & 1 + 2t^2 + t^3 + t^4 + t^5 \\ 
{(1,1,1,1,1)} & t^6 & t^3 + t^4 + t^5 & t + t^3 
\end{array}
\end{displaymath} 

\begin{displaymath}
\begin{array}{c|cc}
\mu \backslash \lambda & (2,1,1) & (1,1,1,1) \\
\hline
{(4)} & t^3 + t^4 + t^5 & t^6 \\
{(3,1)} & t + 2t^2 + 3 t^3 + 2 t^4 + t^6 & t^3 + t^4 + t^5 \\
{(2,2)} & t + t^2 + 2 t^3 + t^4 + t^5 & t^2 + t^4 \\
{(2,1,1)} & 1 + t + 2 t^2 + 2 t^3 + 2 t^4 + t^5 & t + t^2 + t^3 \\ 
{(1,1,1,1,1)} & t + t^2 + t^3 & 1   
\end{array}
\end{displaymath}

\end{proof}

\subsection{Lecture \ref{sec:KZfunctor}}

\begin{proof}[Solution to exercise \ref{ex:invdiff}]
Let $z = x^2$ so that $\C[x]^{\Z_2} = \C[x^2] = \C[z]$. The ring $\dd(\mf{h})^{W}$ is generated by $x^2, x \pa_x$ and $\pa_x^2$ and $\dd(\mf{h} / W) = \C \langle z, \pa_z \rangle$. Since 
$$
x \pa_x (z^n) = x \pa_x (x^{2n}) =2n x^{2n} = 2 n z^n
$$
and
$$
\pa_x^{2} (z^n) = 2n(2n-1) z^{n-1}
$$
we see that $\phi :  \dd(\mf{h})^{W} \rightarrow \dd(\mf{h} / W)$ sends $x \pa_x$ to $2 z \pa_z$ and $\pa_x^2$ is sent to $\pa_z ( 4z \pa_z - 2)$. Then $\pa_z$ is not in the image of $\phi$ so it is not surjective. A rigorous way to show this is as follows: the morphism $\phi$ is filtered. Therefore, it induces a morphism on associated graded, this is the map
$$
\C[A,B,C] / (AC - B^2) \rightarrow \C[D,E],
$$
$$
A \mapsto D, B \mapsto 2 DE, C \mapsto 4DE^2.
$$
This is a proper embedding. 
\end{proof}

\begin{proof}[Solution to exercise \ref{ex:lambda2}]
Under the Dunkl embedding, 
$$
y = \pa_x - \frac{\mbf{c}}{x} (1 - s),
$$ 
which implies that $\pa_x \cdot \rho_0 = 0$ and $\pa_x \cdot \rho_1 = \frac{2\mbf{c}}{x} e_1$. So the $\Z_2$-equivariant local systems on $\C^{\times}$ corresponding to $\Delta(\rho_0)[\delta^{-1}]$ and $\Delta(\rho_1)[\delta^{-1}]$ are given by the differential equations $\pa_x = 0$ and $\pa_x - \frac{2 \mbf{c}}{x} = 0$ respectively. Now we need to construct the corresponding local systems on $\C^{\times} / \Z_2$. If $z := x^2$ then $\pa_z = \frac{1}{2x} \pa_x$ and $\Delta(\rho_0)[\delta^{-1}]^{\Z_2}$, resp. $\Delta(\rho_1)[\delta^{-1}]^{\Z_2}$ has basis $a_0 = 1 \o \rho_0$, resp. $a_1 = x \o \rho_1$, as a free $\C[z^{\pm 1}]$-module. We see that 
$$
\pa_z \cdot a_0 = 0, \quad \pa_z \cdot a_1 = \frac{1 + 2 \mbf{c}}{2z} a_1.
$$
Since the solutions of these equations are $1$ and $z^{\frac{1 + 2 \mbf{c}}{2}}$, the monodromy of these equations is given by $t \mapsto 1$ and $t \mapsto - \exp (2 \pi \sqrt{-1} \mbf{c} t)$ respectively. Therefore $\KZ(\Delta(\rho_0))$ is the one dimensional representation $\C b_0$ of $\H_{\mbf{c}}(\Z_2)$ defined by $T \cdot b_0 = b_0$ and $\KZ(\Delta(\rho_1))$ is the one dimensional representation $\C b_1$ of $\H_{\mbf{c}}(\Z_2)$ defined by $T \cdot b_1 = - \exp (2 \pi \sqrt{-1} \mbf{c}) b_1$ respectively. 
\end{proof}

\begin{proof}[Solution to exercise \ref{ex:PKZ}]
If $\mbf{c} = \frac{1}{2} + m$ for some $m \in \Z_{\ge 0}$ then $P_{KZ} = P(\rho_1)$. If $\mbf{c} = -\frac{1}{2} - m$ for some $m \in \Z_{\ge 0}$ then $P_{KZ} = P(\rho_0)$. Otherwise, $P_{KZ} = P(\rho_0) \oplus P(\rho_1)$, 
\end{proof}

\subsection{Lecture \ref{sec:five}}

\begin{proof}[Solution to exercise \ref{ex:G2}]
The space $X_{\mbf{c}}(G_2)$ is never smooth. The Poincar\'e polynomial of $\C[\mathfrak{h}_1]^{\mathrm{co} G_2}$ is 
$$
(1 - t^2)(1 - t^6) / (1 - t)^2 = 1 + 2 t + 2 t^2 + \cdots + 2 t^5 + t^6.
$$
The polynomials $t^2 + t^4$ and $t + t^5$ do not divide this polynomial in $\Z[t,t^{-1}]$. Therefore 
$$
\dim L(\mathfrak{h}_1), \dim L(\mathfrak{h}_2) < 12
$$
for any parameter $\mbf{c}$, which implies the claim.  
\end{proof}

\newpage

\def\cprime{$'$} \def\cprime{$'$} \def\cprime{$'$} \def\cprime{$'$}
  \def\cprime{$'$} \def\cprime{$'$} \def\cprime{$'$} \def\cprime{$'$}
  \def\cprime{$'$} \def\cprime{$'$} \def\cprime{$'$} \def\cprime{$'$}
  \def\cprime{$'$}


\end{document}